\documentclass[11pt,a4paper]{amsbook}

\usepackage{inputenc}
\usepackage{amsmath}%
\usepackage{amsfonts}%
\usepackage{amssymb}%
\usepackage{graphicx}
\usepackage[Conny]{fncychap}
\usepackage{tikz}

\usepackage{hyperref}

\usepackage[small, nohug,heads=LaTeX]{diagrams}
\usepackage{appendix}
\usepackage{epsfig}
\usepackage{epstopdf}
\usepackage[english,afrikaans,spanish]{babel}
\usepackage{setspace}

\usepackage[all,top]{background}
\SetBgContents{\includegraphics[width=28mm]{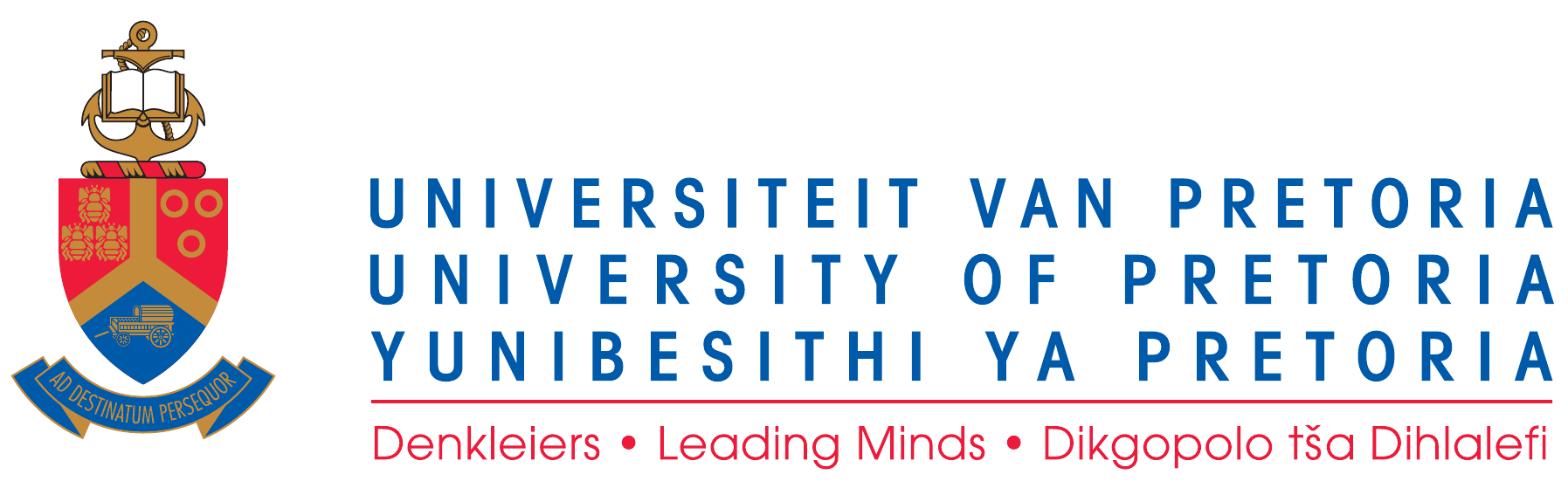}} 
\SetBgScale{1}
\SetBgColor{black}
\SetBgAngle{0}
\SetBgOpacity{1}
\SetBgPosition{1in+\hoffset+0.5\marginparwidth+\marginparsep+0.5\textwidth-32mm,1in+\voffset}
\SetBgVshift{-0.5in-\voffset}
\SetBgHshift{-\hoffset}

\usepackage{fancyhdr}
\newtheoremstyle{mytheorem}{3pt}% ⟨Space above⟩
{3pt}% ⟨Space below⟩
{\itshape}% ⟨Body font⟩
{}% ⟨Indent amount⟩
{\bf}% ⟨Theorem head font⟩
{.}% ⟨Punctuation after theorem head⟩
{.5em}% ⟨Space after theorem head⟩
{}% ⟨Theorem head spec (can be left empty, meaning ‘normal’)⟩

\theoremstyle{mytheorem}
\newtheorem{theorem}{Theorem}[section]
\newtheorem{lemma}[theorem]{Lemma}
\newtheorem{proposition}[theorem]{Proposition}

\newtheorem{claim}[theorem]{Claim}

\newtheoremstyle{mydefinition}{3pt}% ⟨Space above⟩
{3pt}% ⟨Space below⟩
{}% ⟨Body font⟩
{}% ⟨Indent amount⟩
{\bf}% ⟨Theorem head font⟩
{.}% ⟨Punctuation after theorem head⟩
{.5em}% ⟨Space after theorem head⟩
{}% ⟨Theorem head spec (can be left empty, meaning ‘normal’)⟩

\theoremstyle{mydefinition}
\newtheorem{definition}[theorem]{Definition}

\newtheorem{remark}[theorem]{Remark}

\numberwithin{section}{chapter}
\numberwithin{equation}{section}

\newcommand{\I}{\mathbb{I}}
\newcommand{\be}{\begin{equation}\nn}
\newcommand{\ee}{\end{equation}}               
\newcommand{\nn}{\nonumber}     
\newcommand{\bea}{\begin{eqnarray}}         
\newcommand{\eea}{\end{eqnarray}}

\makeindex
\usepackage{etoolbox}
\makeatletter
\patchcmd{\DOTI}{\vskip 40\p@}{\vskip 5\p@}{}{}
\patchcmd{\DOTIS}{\vskip 40\p@}{\vskip 5\p@}{}{}% for unnumbered chapters
\makeatother

\hypersetup{
pdfauthor = {Mauritz van den Worm},
pdftitle = {A noncommutative sigma model},
pdfsubject = {Mathematical physics},
pdfkeywords = {Noncommutative geometry, String theory, Quantum Torus, Irrational Rotation Algebra}
}

\begin{document}
    \setcounter{secnumdepth}{5} % om die een of ander rede laat hierdie
    \setcounter{tocdepth}{5} %command toe dat die subsections in die inhoudsopgewe verskyn
\onehalfspacing
\setcounter{chapter}{-1}
\frontmatter

\title[A Noncommutative Sigma Model]{A Noncommutative Sigma Model}

%    Remove any unused author tags.

%    author one information
\author{Mauritz van den Worm}
%\address{Univesity of Pretoria}
%\email{mauritzvdworm@gmail.com}
%\thanks{Dr. R. Duvenhage}

%\subjclass[2000]{Primary }
%    For books to be published after 1 January 2010, you may use
%    the following version:
%\subjclass[2010]{Primary }


\date{}
\maketitle

\renewenvironment{abstract}{%
\begin{center}%
{\bfseries Abstract\vspace{-.5em}}%
\end{center}%
\quotation}{%
\endquotation}

\newenvironment{uittreksel}{%
\otherlanguage{afrikaans}%
\begin{center}%
{\bfseries Uittreksel\vspace{-.5em}}%
\end{center}%
\quotation}{%
\endquotation
\endotherlanguage}

\newenvironment{abstracto}{%
\otherlanguage{spanish}%
\begin{center}%
{\bfseries Abstracto\vspace{-.5em}}%
\end{center}%
\quotation}{%
\endquotation
\endotherlanguage}

\thispagestyle{empty}
\vspace*{5cm}
\begin{center}
\begin{center}
\textbf{DECLARATION} 
\end{center}
\vspace{1cm}
I, the undersigned, declare that the dissertation, which I hereby submit for the degree Magister Scientiae at the University of Pretoria is my own work and has not previously been submitted by me for any degree at this or any other tertiary institution.
\vspace{5cm}
\begin{center}
\begin{tabular}{ll}
Signature: & ........................................\\
&\\
Name: & Mauritz van den Worm\\
&\\
Date: & ........................................
\end{tabular}
\end{center}
   
\end{center}
\cleardoublepage
\newpage
%\vspace*{\fill}
\thispagestyle{empty}
\begin{uittreksel}\vspace{0.8cm}
Ons vervang die klasieke stringteorie begrippe van die parameterruimte en w\^ereldtyd met nie-kommutatiewe torusse en beskou dan afbeeldings tussen hierdie ruimtes.  Die dinamika van hierdie afbeeldings is bestudeer en `n nie-kommutatiewe veralgemening van die Polyakov-aksie is afgelei.  In besonder is die kwantumtorus in al sy wiskundige besonderhede bestudeer asook afbeeldings tussen verskillende kwantumtorusse.  `n Eindig dimensionele vootstelling van die kwantum torus is ondersoek en spesifieke waardes is verkry vir die partisiefunksie sowel as ander padintegrale.  Laastens is bestaanstellings vir afbeeldings tussen kwantumtorusse bewys.
\end{uittreksel}

%\newpage
%\thispagestyle{empty}
%\vspace*{\fill}
\selectlanguage{english}%
\begin{abstract}\vspace{0.8cm} 
We replaced the classical string theory notions of parameter space and world-time with noncommutative tori and consider maps between these spaces.  The dynamics of mappings between different noncommutative tori were studied and a noncommutative generalization of the Polyakov action was derived.  The quantum torus was studied in detail as well as *-homomorphisms between different quantum tori.  A finite dimensional representation of the quantum torus was studied and the partition function and other path integrals were calculated.  At the end we proved existence theorems for mappings between different noncommutative tori.
\end{abstract}
%\vspace*{\fill}

%\newpage
%\thispagestyle{empty}
%\vspace*{\fill}
%\selectlanguage{spanish}%
\begin{abstracto}\vspace{0.8cm}
Nosotros sustituimos el concepto de cadena cl\'asica teoria del espacio de par\'ametros as\'i como espacio-tiempo con tori no conmutativo.  Asignaciones entre los diferentes tori se estudiaron y versiones no conmutativo de la acci\'on de Polyakov se derivaran.  Una representac\'ion de dimensi\'on finita del tori no conmutativo se construido.  En este caso que pudimos determinar la funci\'on partici\'on as\'i como otros integrales camino.  Finalmente hemos demostrado la existencia de asignaciones entre diferentes tories no conmutativo.
\end{abstracto}
\vspace*{\fill}

%    Dedication.  If the dedication is longer than a line or two,
%    remove the centering instructions and the line break.
\clearpage

\thispagestyle{empty}
\null\vspace{\stretch{1}}
\begin{flushright}
Por Lerinza, \textsl{solo pienso en ti}
\end{flushright}
\vspace{\stretch{2}}\null

\clearpage
\thispagestyle{empty}

\vspace*{\fill}
\begin{center}
 \begin{quote}
  \textsl{``In the past decades theoretical physicists have been using ever more sophisticated mathematics to model the universe and its fundamental forces. Quantum Theory and Geometry are the two main ingredients but there are different schools of thought on how to fuse them together. Einstein, with his success in General Relativity, argued for the primacy of Geometry and Dirac said we should be guided by beauty.  I belong to this camp and am tentatively exploring some new ideas." \newline - Sir Michael Atiyah \newline(1 February 2011, Salle 5 of College de France)}
 \end{quote}
\end{center}
\vspace*{\fill}

\newpage
\chapter*{Acknowledgments}
\begin{center}
\begin{quote}
 \textsl{``If you attack a mathematical problem directly, very often you come to a dead end, nothing you do seems to work and you feel that if only you could peer round the corner there might be an easy solution. There is nothing like having somebody else beside you, because he can usually peer round the corner.''\newline- Sir Micheal Atiyah}
\end{quote}
\end{center}
\vspace*{\fill}
There are numerous people who helped me peer around the corner, but the person to whom I owe the most gratitude with respect to this thesis is Dr. Rocco Duvenhage.  Patience is without a doubt one of his virtues.  I would like to thank the physics department for not ostracizing the small assembly of mathematicians (Rocco and myself) with whom you share these humble corridors.  Discussions with Danie van Wyk and Gusti van Zyl proved valuable.  Needless to say that without a steady supply of coffee from the lab this thesis would not have been possible.  I would like to express my sincere acknowledgements to NITheP for their financial assistance regarding my studies and travel expenses.  Last but not least I would like to thank Lerinza who spent a countably infinite number of evenings with me while working on this thesis.
\vspace*{\fill}

\clearpage

%    Change page number to 6 if a dedication is present.
\setcounter{page}{7}
    {\pagestyle{plain}
\tableofcontents
\listoffigures
\cleardoublepage}

\chapter{Introduction}\noindent
\section{String Theory}\noindent
In string theory we are interested in mappings from a two dimensional parameter space which we call $\Sigma$ to our spacetime, $X$.  The parameter space has coordinates $\sigma$ and $\tau$, which are roughly related to position along the string and the time on the string respectively.  If we assume that spacetime has $d$ spatial coordinates and one time coordinate we can write any spacetime coordinate as
\begin{equation*}
 x=(x^0,x^1,\cdots,x^d).
\end{equation*}
Any spacetime surface can then be described using the mapping functions
\begin{equation*}
 \varphi^{\mu}(\tau,\sigma)
\end{equation*}
which take some region in the $(\tau,\sigma)$ parameter space and map it into spacetime.  If we take any fixed point in the parameter space and map it to spacetime, we obtain the point
\begin{equation*}
 \varphi(\tau,\sigma)=\left(\varphi^0(\tau,\sigma),\varphi^1(\tau,\sigma),\cdots,\varphi^d(\tau,\sigma)\right).
\end{equation*}
A string in spacetime at time $\tau$ can then be seen as the one dimensional collection of points joining $\varphi(\tau,0)$ to $\varphi(\tau,\sigma_1)$ through the mapping functions.  Here we took $\sigma_1$ to be the upper limit of the $\sigma$-coordinate in parameter space.  The world lines of the string endpoints have constant values of $\sigma$ and we can then parametrize the string by $\tau$.

In studying the dynamics that arise in string theory one usually starts out with the Nambu-Goto action.  In Lagrangian mechanics the action of a free point particle is proportional to its proper time which we can also see as the length of its world-line in spacetime.  The Nambu-Goto action generalizes this idea to a one dimensional object which sweeps out a two dimensional surface in spacetime.  The Nambu-Goto action is then defined to be the proper area of the surface swept out by the one dimensional object.  The Nambu-Goto action can be written as:
\begin{equation*}
 S_{NG}=-\frac{T_0}{c}\int^{\tau_{f}}_{\tau_i}\int^{\sigma_1}_{0}\sqrt{\left(\frac{\partial\varphi}{\partial \tau}\frac{\partial\varphi}{\partial\sigma}\right)^2-\left(\frac{\partial\varphi}{\partial\tau}\right)^2\left(\frac{\partial\varphi}{\partial\sigma}\right)^2}d\tau d\sigma
\end{equation*}
where $T_0$ is the tension in the string and $c$ is the speed of light.  Here we are integrating time from some initial time $\tau_i$ to a later final time $\tau_f$.  The proper area in spacetime is
\begin{equation*}
 A_{\text{proper}}=\int^{\tau_{f}}_{\tau_i}\int^{\sigma_1}_{0}\sqrt{\left(\frac{\partial\varphi}{\partial \tau}\frac{\partial\varphi}{\partial\sigma}\right)^2-\left(\frac{\partial\varphi}{\partial\tau}\right)^2\left(\frac{\partial\varphi}{\partial\sigma}\right)^2}d\tau d\sigma.
\end{equation*}

\begin{figure}[h!]
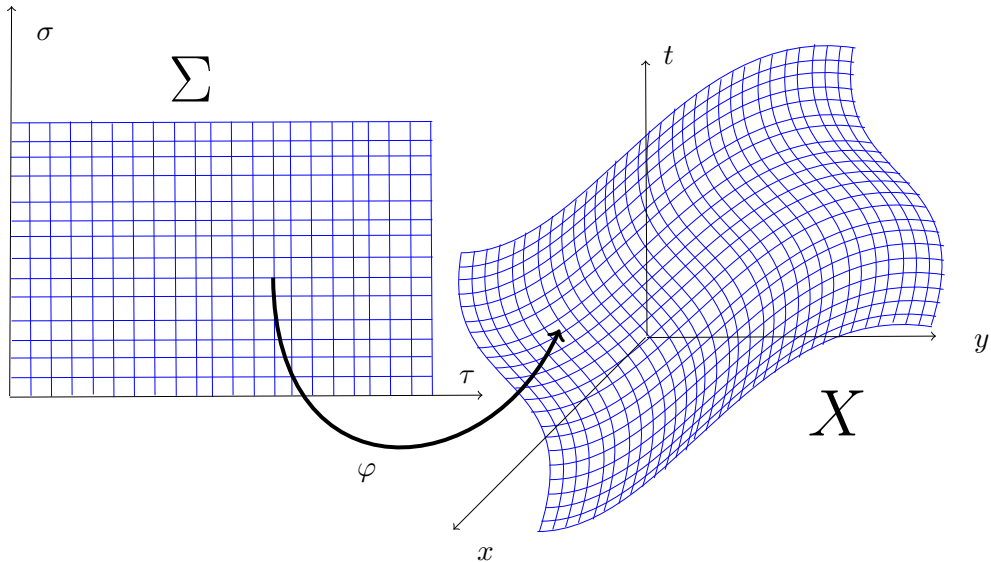

\ifx\du\undefined
  \newlength{\du}
\fi
\setlength{\du}{15\unitlength}

\caption{Mapping from parameter space $\Sigma$ to world-time $X$.}
\end{figure}

The action we will mostly be concerned with is known as the Polyakov action.  The Polyakov and Nambu-Goto actions give the same equations of motion for the relativistic string and are in fact equivalent.  However in the framework of modern string theory as well as in this thesis the Polyakov action will be a more natural choice to work with.  For completeness sake we state the Polyakov action here:
\begin{equation*}
 S=\frac{-1}{4\pi\alpha^{\prime}}\int \sqrt{-h}h^{\alpha\beta}\partial_{\alpha}\varphi^{\mu}\partial_{\beta}\varphi^{\nu}\eta_{\mu\nu}d\tau d\sigma.
\end{equation*}
Notice that this equation makes use of Einsteinian notation where repeated indices in the subscript and superscript imply summation.  $h^{\alpha\beta}$ and $\eta_{\mu\nu}$ are the metrics of repectively the parameter space and world-time.  The reader should not be concerned with these metrics as we will only consider the case of flat Euclidean space and make the necessary simplifying assumptions in which case the Polyakov action reduces to
\begin{equation*}
 S=\int\partial_{\alpha}X_{\mu}\partial^{\alpha}X^{\mu}d\sigma d\tau.
\end{equation*}

The notion of compactification the parameter space also plays a major role in string theory and will also help in visualizing the process which we are going to follow in this thesis.
\begin{figure}[h!]
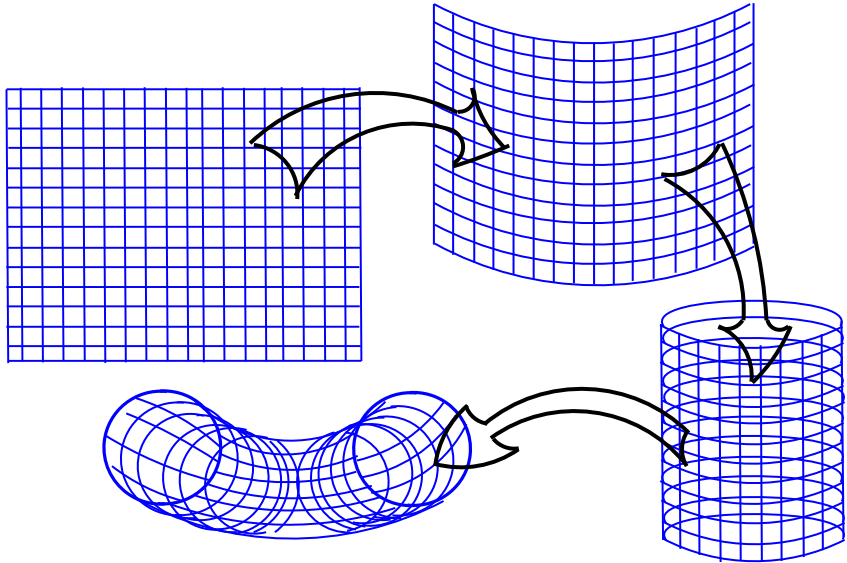

\ifx\du\undefined
  \newlength{\du}
\fi
\setlength{\du}{15\unitlength}
 
\caption{Showing the compactifying of parameter space.}
\end{figure}\noindent
By identifying $\tau_i$ and $\tau_f$ lines in parameter space we obtain a cylinder.  We can then identify the $\sigma=0$ and $\sigma=\sigma_1$ lines which correspond the to top and bottom boundaries of the cylinder to form a torus.  The reason we are mentioning compactifying of the parameter space is because we would like to replace the notions of classical string theory regarding parameter space and spacetime with noncommutative spaces, in particular noncommutative $C^*$-algebras.  

The quintessential example of a noncommutative space is the \textsl{quantum torus}.  So it makes sense to visualise the parameter space as being compactified into a torus and to generalize the results of classical string theory on these classical spaces to more general $C^*$-algebras.  In this thesis we replace both parameter space and spacetime by different noncommutative tori and consider *-homomorphisms between these noncommutative $C^*$-algebras which play the role of the mapping functions.  In this way we develop the theory of a noncommutative $\sigma$-model.  Noncommutative versions of the mapping functions and the Polyakov action are also derived.

\section{Structure of this Dissertation}\noindent
Starting off in Chapter 1 we consider the $C^*$-algebra generated by a collection of operators.  Following these ideas we introduce the \textsl{quantum torus} and study some of its more interesting properties, such as the fact that it can be written as a \textsl{crossed product} which will greatly aid us in determining its K-theory and the \textsl{unique trace} on
the quantum torus.  The final section of Chapter 1 deals with a \textsl{finite dimensional} representation of the quantum torus.  We consider the special case of two $n$ by $n$ unitary matrices satisfying the commutation relation that defines the quantum torus.  In Chapter 2 we introduce the concept of a $\sigma$-model and show that we can generalize this classical notion to the noncommutative world.  This is primarily done by showing that the Polyakov action can be written in a noncommutative manner where integration is replaced by taking the trace and norms squared are replaced by taking the trace of the product of an operator with its adjoint.  We explore 
a finite dimensional $\sigma$-model based on the construction of the finite dimensional representation of the quantum torus.  In this case we are able to determine an explicit \textsl{partition function} of our finite dimensional representation and calculate \textsl{minima} and \textsl{maxima} for different parameters.  If we make certain assumptions regarding our system we are also able to calculate thermodynamic quantities such as the expectation value of the energy as well as the entropy of the finite dimensional $\sigma$-model.  Up until this point we have assumed the existence of such *-homomorphisms between different quantum tori.  To be able to prove the existence of such *-homomorphisms we require some knowledge regarding the \textsl{K-theory} of $C^*$-algebras and \textsl{Morita Equivalence}, an overview of these topics can be found in Chapter 3.  In order to calculate the K-groups of the quantum torus in Chapter 3 we require an extremely deep result from K-theory known as the Pimsner-Voiculescu-
sequence \cite[Theorem 2.4]{PV1980}.  We state the Pimsner-Voiculescu-sequence without proof and only use it to determine the K-theory of the quantum torus.  The most technical section of this thesis is found in Chapter 4 where we prove the existence of *-homomorphisms between different quantum tori.  In Chapter 4 we make use of a fundamental paper by Rieffel \cite{Rieffel1981} and use his results without proof.  In the appendix the reader can find some information regarding short exact sequences.
\mainmatter

\chapter{The Quantum Torus}\label{chapter_2}\noindent
\begin{center}
 \begin{quotation}
  \textsl{``Young man, in mathematics you don't understand things. You just get used to them.''\newline
-John von Neumann's reply to Felix T. Smith who had said ``I'm afraid I don't understand the method of characteristics.''}
 \end{quotation}
\end{center}
In order to construct a noncommutative $\sigma$-model we require a noncommutative space onto which our world sheet will be mapped.  The quintessential example of a noncommutative space in the irrational rotation algebra or quantum torus.  Here we present a detailed account of the quantum torus.
\section{$C^*$-Algebra Generated by a set of operators}\noindent
We start off this section with some general background which will serve as preparation in the construction of the quantum torus.  Consider the set $\mathcal{O}\subset B$ with $B$ a $C^*$-algebra.  Let $A_{0}$ be the set of all finite linear combinations of finite products of elements in $\mathcal{O}\cup\mathcal{O}^*$.  

We now proceed to show that $A_0$ is a subspace of $B$.  Consider the elements $a,b\in A_{0}$ which we can write as
\begin{eqnarray*}
a=\sum_i\alpha_i\prod_{j_i}a_{j_i},\quad
b=\sum_i\beta_i\prod_{j_i}b_{j_i}
\end{eqnarray*}
where $a_{j_i},b_{j_i}\in\mathcal{O}\cup\mathcal{O}^*$ and $\alpha_i,\beta_i\in\mathbb{C}$.  When we consider the sum of any two elements of such form we find
\begin{equation*}
a+b=\sum_i\alpha_i\prod_{j_i}a_{j_i}+\sum_i\beta_i\prod_{j_i}b_{j_i}
\end{equation*}
which is once again a finite linear combination of finite products of elements in $\mathcal{O}\cup\mathcal{O}^*$.  Similarly when considering scalar multiplication we find that
\begin{equation*}
\gamma a=\gamma\sum_i\alpha_i\prod_{j_i}a_{j_i}=\sum_i\gamma\alpha_i\prod_{j_i}a_{j_i}
\end{equation*}
which is of course a finite linear combination of finite products of elements in $\mathcal{O}\cup\mathcal{O}^*$.  So $A_0$ is a vector subspace of $B$.  

Let $A$ denote the closure of $A_0$.  We show that $A$ forms a subspace of $B$.  Consider any $a$ and $b$ in $A$ and let $\{a_n\}$ and $\{b_{n}\}$ be two sequences in $A_0$ which respectively converge to $a$ and $b$ in $A$.  We can now also consider the sequence $\{a_n+b_n\}$ and show that it converges to $a+b$
\begin{equation*}
\lim_{n\rightarrow\infty}\left(a_n+b_n\right)=\lim_{n\rightarrow\infty}a_n+\lim_{n\rightarrow\infty}b_n=a+b
\end{equation*}
So $a+b\in A$.  Next consider the sequence $\{\alpha a_n\}$ with $\alpha\in\mathbb{C}$
\begin{equation*}
\lim_{n\rightarrow\infty}\alpha a_n=\alpha\lim_{n\rightarrow\infty}a_n=\alpha a
\end{equation*}
So $\alpha a_{n}\in A$.  Since we have closure under vector addition and scalar multiplication $A$ is a vector subspace of $B$.  

Next we show that $A_0$ is an algebra.  Consider the elements $a$ and $b$ as defined earlier.  Then
\begin{eqnarray*}
ab&=&\left(\sum_i\alpha_i\prod_{j_i}a_{j_i}\right)\left(\sum_k\beta_k\prod_{j_k}b_{j_k}\right)\\
&=&\sum_{i,k}\alpha_i\beta_k\prod_{j_i,j_k}a_{j_i}b_{j_k}.
\end{eqnarray*}
which according to the definition is again in $A_0$ since it is a finite linear combination of finite products of elements if $\mathcal{O}\cup\mathcal{O}^*$.  

Now we show that $A$ also forms an algebra.  Consider any $a$ and $b$ in $A$ and any two sequences $\{a_n\}$ and $\{b_n\}$ in $A_0$ that converge respectively to $a$ and $b$ in $A$.  So given $\epsilon>0$, there is some $N$ such that
\begin{eqnarray*}
\|a_n-a\|<\frac{\epsilon}{2M},\quad
\|b_n-b\|<\frac{\epsilon}{2M}
\end{eqnarray*}
for $n\geq N$, where we define $M:=\max\{\|a_n\|,\|b\|\}$.  We will show that the sequence $\{a_nb_n\}$ converges to $ab\in A$:
\begin{eqnarray*}
\|a_nb_n-ab\|&=&\|a_nb_n-a_nb+a_nb-ab\|\\
&\leq&\|a_n\|\|b_n-b\|+\|b\|\|a_n-a\|\\
&\leq&M\frac{\epsilon}{2M}+M\frac{\epsilon}{2M}\\
&=&\epsilon
\end{eqnarray*}
So $ab\in A$ and we once again have the bilinear map with $A$ in the place of $A_0$.  This implies that $A$ is an algebra.  

We can define the involution on $A$ as that inherited from $B$.  For any $a$ in $A$ there is a sequence $\{a_n\}$ in $A_{0}$ converging to $a$ in $A$.  Recalling that $A_0$ is defined as the set of all finite linear combinations of finite products of elements in $\mathcal{O}\cup\mathcal{O}^*$.  This implies that for each $n$, $\{a^*_n\}$ is in $A_0$.  Now, since $A$ is the closure of $A_0$ it follows that $a^*$ is in $A$.

Now since $A$ is a closed, self-adjoint *-subalgebra of the $C^*$-algebra $B$, $A$ is a $C^*$-algebra in its own right.  We will refer to $A$ as the $C^*$-subalgebra of $B$ generated by $\mathcal{O}$.  Sometimes we will make use of the notation $C^*(\mathcal{O})$ to emphasize that we are considering the $C^*$-algebra generated by $\mathcal{O}$.

\section{The Quantum Torus}\label{qt_section}\noindent
Topologically we can define the classical two-torus, $\mathbb{T}^2$ by the quotient space
\begin{equation*}
\mathbb{T}^2:=\mathbb{R}^2/2\pi \mathbb{Z}^2
\end{equation*}
The underlying Hilbert space we will be dealing with here is $L^2(\mathbb{T}^2)$ and we define the following operators
\begin{eqnarray*}
Uf(x,y)=e^{2\pi ix}f\left(x,y+\frac{\theta}{2}\right)\\
Vf(x,y)=e^{2\pi iy}f\left(x-\frac{\theta}{2},y\right)
\end{eqnarray*}
for any $f\in L^2(\mathbb{T}^2)$.
%Moet nog die maat op hierdie ruimte spesifiseer.  Dit moet die Haar maat wees, maar weet nie veel van hierdie goed nie.
\begin{lemma}
 The operators $U$ and $V$ as defined above are unitary. 
\end{lemma}
\begin{proof}
 Consider the operators defined by
\begin{eqnarray*}
Af(x,y)&=&e^{-2\pi ix}f\left(x,y-\frac{\theta}{2}\right):=g(x,y)\\
Bf(x,y)&=&e^{-2\pi iy}f\left(x+\frac{\theta}{2},y\right)
\end{eqnarray*}
We show that these operators are the inverse operators of $U$ and $V$ respectively:
\begin{eqnarray*}
\left[UAf\right](x,y)&=&Ug(x,y)\\
&=&e^{2\pi ix}g\left(x,y+\frac{\theta}{2}\right)\\
&=&e^{2\pi ix}e^{-2\pi ix}f\left(x,y-\frac{\theta}{2}+\frac{\theta}{2}\right)\\
&=&f(x,y)
\end{eqnarray*} 
Let $h(x,y):=Uf(x,y)$.  Then
\begin{eqnarray*}
\left[AUf\right](x,y)&=&Ah(x,y)\\
&=&e^{-2\pi ix}h\left(x,y-\frac{\theta}{2}\right)\\
&=&e^{-2\pi ix}e^{2\pi ix}f\left(x,y+\frac{\theta}{2}-\frac{\theta}{2}\right)\\
&=&f(x,y)
\end{eqnarray*}
So $AU=UA=\I$ and using the same arguments we can also show that $BV=VB=\I$.  Now consider the inner product of the functions $f,g\in L^2(\mathbb{T}^2)$
\begin{eqnarray*}
\langle Uf,Ug\rangle&=&\int^{2\pi}_{0}\int^{2\pi}_{0}\overline{e^{2\pi ix}f\left(x,y+\frac{\theta}{2}\right)}e^{2\pi ix}g\left(x,y+\frac{\theta}{2}\right)dxdy\\
&=&\int^{2\pi}_{0}\int^{2\pi}_{0}\overline{f\left(x,y+\frac{\theta}{2}\right)}g\left(x,y+\frac{\theta}{2}\right)dxdy\\
&=&\langle f ,g\rangle
\end{eqnarray*}
So $U^*U=\I$.  So $U$ is unitary, in the same way we can show that $V$ is also unitary.
\end{proof}\noindent
We will now find the commutation relation of the operators $U$ and $V$.  Let $g(x,y):=Uf(x,y)$ and $h(x,y):=Vf(x,y)$ and consider
\begin{eqnarray*}
 \left[UVf\right](x,y)&=&Uh(x,y)\\
&=&e^{2\pi ix}h\left(x,y+\frac{\theta}{2}\right)\\
&=&e^{2\pi ix}e^{2\pi i\left(y+\frac{\theta}{2}\right)}f\left(x-\frac{\theta}{2},y+\frac{\theta}{2}\right)\\
&=&e^{2\pi i(x+y)}e^{\pi i\theta}f\left(x-\frac{\theta}{2},y+\frac{\theta}{2}\right).
\end{eqnarray*}
Similarly we calculate $VU$
\begin{eqnarray*}
\left[VUf\right](x,y)&=&Vg(x,y)\\
&=&e^{2\pi iy}g\left(x-\frac{\theta}{2},y\right)\\
&=&e^{2\pi iy}e^{2\pi i\left(x-\frac{\theta}{2}\right)}f\left(x-\frac{\theta}{2},y+\frac{\theta}{2}\right)\\
&=&e^{2\pi i\left(x+y\right)}e^{-\pi i\theta}f\left(x-\frac{\theta}{2},y+\frac{\theta}{2}\right).
\end{eqnarray*}
Hence if we define $k:=e^{2\pi i\theta}$ then we can write the commutation relation of the operators $U$ and $V$ as
\begin{equation*}
 UV=kVU.
\end{equation*}

\begin{definition}(The Quantum Torus)\label{QT}\cite[p. 109]{Weaver}\index{$C^{\hbar}(T^2)$}\newline
Let $\theta\in [0,1)$.  Let $A_{\theta}$ be the $C^*$-subalgebra of $B(L^2(\mathbb{T}^2))$ generated by two unitary operators $\{U,V\}$, satisfying the commutation relation:
\begin{eqnarray}\label{commutation_relation}
 UV=e^{2\pi i\theta}VU.
\end{eqnarray}
\end{definition}

\begin{proposition}\label{kontinu}
In the case where $\theta=0$ the quantum torus $A_{0}\cong C\left(\mathbb{T}^2\right)$.
\end{proposition}
\begin{proof}
In this limiting case where $\theta=0$ the operators $U$ and $V$ reduce to
\begin{eqnarray*}
Uf(x,y)=e^{2\pi ix}f\left(x,y\right),\quad
Vf(x,y)=e^{2\pi iy}f\left(x,y\right).
\end{eqnarray*}
So we can think of them simply as multiplication by an exponential function.  Now $A_{0}$ is the commutative unital $C^*$-algebra generated by $U$ and $V$.  $A_0$ clearly separates the points of $\mathbb{T}^2$ and by the Stone-Weierstrass Theorem we can conclude that $A_{0}$ is dense in $C\left(\mathbb{T}^2\right)$.  However $C\left(\mathbb{T}^2\right)=\overline{A_0}=A_0$ which concludes the proof. 
\end{proof}

\noindent
The reader will notice that equation (\ref{commutation_relation}) has the exact form of the Weyl commutation relations.

\section{The n-Dimensional Quantum Torus}\label{highT}\noindent
For the remainder of the thesis we will use the \textsl{quantum torus} as defined in the previous section, however we can extend the results to a more general case.  In this section we extend the quantum torus generated by two unitary operators with the commutation relation (\ref{commutation_relation}) to a ``higher dimensional'' version.  A similar construction used to describe the normal \textsl{quantum torus} can now be followed to construct any quantum n-torus.  We define the generating operators as follows
\begin{equation*}
U_{j}f(x_1,\dots,x_n)=e^{ix_j}f(x_1+z_{j,1},\dots,x_n+z_{j,n})
\end{equation*}
for every $j=1,\dots,n$ and $f\in L^2(\mathbb{T}^n)$ where $\mathbb{T}^n$ is the n-torus which we define topologically
\begin{equation*}
\mathbb{T}^n:=\mathbb{R}^n/2\pi\mathbb{Z}^n
\end{equation*}
and where each $z_{j,i}$ is some constant.
\begin{lemma}
 The operators $U_j$ defined above are unitary for every $j=1,\dots,n$
\end{lemma}
\begin{proof}
 Just as in the case of the two-torus we define the operators $A_{j}$ in the following way
\begin{equation*}
A_jf(x_1,\dots,x_n):=e^{-ix_j}f(x_1-z_{j,1},\dots,x_n-z_{j,n})
\end{equation*}
Clearly $A_jU_j=U_jA_j=\I$ so $A_j=U^{-1}_j$.  We consider the inner product with $f,g\in L^2(\mathbb{T}^n)$
\begin{eqnarray*}
\langle U_jf,U_jf\rangle\
&=&\int^{2\pi}_{0}\dots\int^{2\pi}_{0}\overline{e^{ix_j}f\left(x+z_j\right)}e^{ix_j}g\left(x+z_j\right)dx_1\dots dx_n\\
&=&\int^{2\pi}_{0}\dots\int^{2\pi}_{0}\overline{f\left(x+z_j\right)}g\left(x+z_j\right)dx_1\dots dx_n\\
&=&\langle f ,g\rangle.
\end{eqnarray*}
Here $x$ denotes the n-dimensional vector on the n-torus and $z_j$ the constant we add to each component of $x$ under the action of the $U$ operator. This shows that $U^*_jU_j=\I$.  This shows that $U_j$ is unitary for every $j=1,\dots,n$.
\end{proof}
We can also determine the commutation relation between any two of these $U_j$ operators.  Let 
\begin{equation*}
U_kf(x_1,\cdots,x_n)=e^{ix_k}f(x_1+z_{k,1},\cdots,x_n+z_{k,n}):=g(x_1,\cdots,x_n) 
\end{equation*}
and
\begin{equation*}
U_lf(x_1,\cdots,x_n)=e^{ix_l}f(x_1+z_{l,1},\cdots,x_n+z_{l,n}):=h(x_1,\cdots,x_n).
\end{equation*}
Now consider
\begin{eqnarray*}
 \left[U_kU_lf\right](x_1,\cdots,x_n)&=&\left[U_kh\right](x_1,\cdots,x_n)\\
&=&e^{ix_{k}}h(x_1+z_{k,1},\cdots,x_n+z_{k,n})\\
&=&e^{i(x_k+x_l)}e^{iz_{k,l}}f(x_1+z_{k,1}+z_{l,1},\cdots,x_n+z_{k,n}+z_{l,n}).
\end{eqnarray*}
Similarly we can determine
\begin{eqnarray*}
 \left[U_lU_kf\right](x_1,\cdots,x_n)&=&\left[U_lg\right](x_1,\cdots,x_n)\\
&=&e^{ix_{l}}g(x_1+z_{l,1},\cdots,x_n+z_{l,n})\\
&=&e^{i(x_k+x_l)}e^{iz_{l,k}}f(x_1+z_{k,1}+z_{l,1},\cdots,x_n+z_{k,n}+z_{l,n}).
\end{eqnarray*}
This enables us to determine the commutation relation between any $U_k$ and $U_l$, which we find to be
\begin{equation}\label{highDcomrel}
 U_kU_l=e^{i(z_{k,l}-z_{l,k})}U_lU_k:=e^{2\pi i\theta_{k,l}}U_lU_k.
\end{equation}
We are now able to generalize our earlier definition in the following way.
\begin{definition}(The Quantum n-Torus)\newline
 Let $C^{\theta}(\mathbb{T}^n)$ be the $C^*$-subalgebra of $B(L^2(\mathbb{T}^n))$ generated by the unitary operators $\{U_j:j=1,\cdots,n\}$ that satisfy the commutation relation (\ref{highDcomrel}).
\end{definition}
\noindent
For the remainder of this thesis we will be working with the normal quantum two torus, the $C^*$-algebra generated by two unitary operators $U$ and $V$ satisfying the commutation relation (\ref{commutation_relation}).
 
\section{The Quantum Torus as a Crossed Product}
This section is based on the work done by Dana Williams in his book \cite{Williams}.  We will follow the broad outline laid out in his book to construct the crossed product of the quantum torus.  We include the necessary lemmas and propositions from \cite{Williams} some of which, without proof, since the reader can easily find them in the book.  For most of this section we will simply use the results mentioned in \cite{Williams} to show that the quantum torus can be written as a crossed product of two $C^*$-algebras.
\begin{definition}\textsl{($C^*$-Dynamical System)}\label{dynamical_system}\newline
 A $C^*$-dynamical system is a triple $(A,G,\alpha)$ consisting of a $C^*$-algebra $A$, a locally compact group $G$ and a continuous homomorphism $\alpha:G\rightarrow \text{Aut}A$. 
\end{definition}
The ideas of $C^*$-dynamical systems will occur frequently in the sections that follow.
\begin{definition}\textsl{($C(X)$, $C_0(X)$ and $C_c(X)$)}\newline
 If $X$ is a locally compact Hausdorff space, then $C(X)$, $C_0(X)$ and $C_c(X)$ denote, respectively, the algebra of all continuous complex-valued functions on $X$, the subalgebra of all bounded complex valued functions vanishing at infinity and the subalgebra of $C_0(X)$ consisting of functions with compact support.
\end{definition}

\begin{definition}\textsl{(Covariant Representation)}\newline
 Let $(A,G,\alpha)$ be a $C^*$-dynamical system.  Then a covariant representation of $(A,G,\alpha)$ is a pair $(\pi,U)$ consisting of a representation $\pi:A\rightarrow B(H)$ and a unitary representation $U:G\rightarrow U(H)$ on the same Hilbert space $H$, such that for any $a\in A$
\begin{equation*}
 \pi\left(\alpha_s(a)\right)=U_s\pi(a)U^*_s.
\end{equation*}
Here we define the notation $U(s):=U_s$ and similarly $\alpha(s)(a):=\alpha_s(a)$.
\end{definition}
%% kan miskien hier n voorbeeld van so n covariant representation insit om net die idee te illustreer %%
We state the next three lemmas without proof (The interested reader can be find details regarding the proofs in \cite[Lemma 1.87, Proposition 2.23, Lemma 2.27]{Williams}).
\begin{lemma}\label{induct}
 Suppose that $D_0$ is a dense subset of a Banach space $D$.  Then 
\begin{equation*}
 C_c(G)\odot D_0:=\text{span}\left\{z\otimes a: z\in C_c(G), a\in D_0 \right\}
\end{equation*}
is dense in $C_c(G,D)$ in the inductive limit topology, and therefore for the topology on $C_c(G,D)$ induced by the $L^1$-norm.
\end{lemma}

\begin{lemma}
Suppose that $(\pi,U)$ is a covariant representation of $(A,G,\alpha)$ on $H$.  Then
\begin{equation*}
 \pi\rtimes U(f):=\int_G\pi(f(s))U_sd\mu(s)
\end{equation*}
 defines a *-representation of $C_c(G,A)$ on $H$ called the integrated form of $(\pi,U)$.  We call $\pi\rtimes U$ the integrated form of $(\pi,U)$.
\end{lemma}

\begin{lemma}\label{cp}
 Suppose that $(A,G,\alpha)$ is a dynamical system and that for each $f\in C_c(G,A)$ we define
\begin{equation*}
 \|f\|_u:=\text{sup}\left\{\|\pi\rtimes U(f)\|:(\pi,U) \text{ is a covariant representation of } (A,G,\alpha) \right\}.
\end{equation*}
Then $\|\cdot\|_u$ is a norm on $C_c(G,A)$ called the universal norm.  The completion of $C_c(G,A)$ with respect to $\|\cdot\|_u$ is a $C^*$-algebra called the crossed product of $A$ by $G$ and denoted by $A\rtimes_{\alpha}G$.
\end{lemma}

It is possible to describe the $C^*$-algebra structure on $C_c\left(G,C_0(X)\right)$ in terms of functions on $G\times X$.  We make the identification
\begin{equation*}
 \psi_f:G\times X\rightarrow \mathbb{C}:(s,x)\mapsto f(s)(x)
\end{equation*}
where $f\in C_c(G,C_0(X))$.  Then clearly $\psi_f(s,\cdot)=f(s)\in C_0(X)$.  It is clear that we have the following inclusions
\begin{equation*}
 C_c(G\times X)\subset C_c(G,C_c(X))\subset C_c(G,C_0(X)).
\end{equation*}
Furthermore since point evaluation is a *-homomorphism from $C_0(X)$ to $\mathbb{C}$ if $f\in C_c(G,C_0(X))$ , we have
\begin{equation*}
 \int_G f(s)d\mu(s)(x)=\int_Gf(s)(x)d\mu(s)
\end{equation*}
with $\mu$ the Haar measure of the group $G$.  Note that if the group is not abelian, there will be both a left and a right Haar measure.  For our puroposes the left Haar measure will be in view. 

\begin{lemma}\cite[Lemma 1.61]{Williams}
Let $\mu$ be the Haar measure on a locally compact group $G$.  Then there is a continuous homomorphism $\Delta:G\rightarrow \mathbb{R}^+$ such that
\begin{equation*}
 \Delta(r)\int_Gf(sr)d\mu(s)=\int_Gf(s)d\mu(s)
\end{equation*}
 for all $f\in C_c(G)$.  The function $\Delta$ is independent of choice of Haar measure and is called the modular function on $G$.
\end{lemma}

It is clear that for any $f$ and $g$ in $C_c(G,A)$ and $(r,s)\in G\times G$ we have $f(r)\alpha_r(g(r^{-1}s))\in A$.  Since both $f$ and $g$ are in $C_c(G,A)$ it is easy to see that the mapping $(s,r)\mapsto f(r)\alpha_r(g(r^{-1}s))$ is in $C_c(G\times G,A)$.  We also observe that
\begin{equation*}
 f*g(s):=\int_G f(r)\alpha_r(g(r^{-1}s))d\mu(r)
\end{equation*}
defines ans element of $C_c(G,A)$.  We call the above star product the \textsl{convolution} of $f$ and $g$.  It can be shown using \cite[Proposition 1.105, Lemma 1.92]{Williams} that for all $f,g,h\in C_c(G,A)$ this convolution is associative
\begin{equation*}
 (f*g)*h=f*(g*h).
\end{equation*}

\begin{lemma}
 Define the mapping $*:C_c(G,A)\rightarrow C_c(G,A)$ by
\begin{equation*}
 f^*(s):=\Delta(s^{-1})\alpha_s(f(s^{-1})^*).
\end{equation*}
The *-mapping is an involution on $C_c(G,A)$.
\end{lemma}
\begin{proof}
The following calculation proves the lemma:
\begin{eqnarray*}
 (f^*)^*(s)&=&\Delta(s^{-1})\alpha_s((f^*)(s^{-1})^*)\\
&=&\Delta(s^{-1})\alpha_s\left(\left( \Delta(s)\alpha_{s^{-1}}(f(s)^*) \right)^* \right)\\
&=&\Delta(s^{-1})\Delta(s)\alpha_s\left( \alpha_{s^{-1}}(f(s)) \right)\\
&=&f(s).
\end{eqnarray*}
\end{proof}
The convolution together with the involution defined in the previous lemma shows that $C_c(G,A)$ is a *-algebra.
%% Moet hier nog mooi uitskryf wat presies ek wil doen, dit gaan maar so biejtie lukraak hierso, nie tevrede daarmee nie.

\begin{remark}
 We know that $C_c(\mathbb{Z},C_c(\mathbb{T}))$ has a dense *-subalgebra $C_c(\mathbb{Z}\times \mathbb{T})$ where the convolution product is given by the finite sum
\begin{equation*}
 f*g(n,z):=\sum^{\infty}_{m=-\infty}f(m,z)f(n-m,e^{-2\pi im\theta}z),
\end{equation*}
and the involution by
\begin{equation*}
 f^*(n,z)=\overline{f(-n,e^{-2\pi in\theta}z)}.
\end{equation*}
If $\varphi\in C(\mathbb{T})$ and $h\in C_c(\mathbb{Z})$, then we define $\varphi\otimes h$ as an element of $C_c(\mathbb{Z}\times\mathbb{T})$ by writing
\begin{equation*}
 \varphi\otimes h(n,z)=\varphi(z)h(n).
\end{equation*}
Let $\delta_n$ be the function on $\mathbb{Z}$ which is equal to 1 at $n$ and zero elsewhere.
\end{remark}

\begin{lemma}
 The identity element of $C_c(\mathbb{Z}\times \mathbb{T})$ is $1\otimes\delta_0$.  Here $1$ denotes the identity function in $C_c(\mathbb{T})$.
\end{lemma}
\begin{proof}
 Let $\varphi\in C(\mathbb{T})$ and $h\in C_c(\mathbb{Z})$ and consider
\begin{eqnarray*}
 (\varphi\otimes h)*(1\otimes\delta_0)(n,z)&=&\sum^{\infty}_{m=-\infty}\varphi\otimes h(m,z)1\otimes\delta_0(n-m,e^{-2\pi im \theta}z)\\
&=&\sum^{\infty}_{m=-\infty}\varphi(z)h(m)\delta_0(n-m)\\
&=&\varphi(z)h(n)\\
&=&\varphi\otimes h(n,z).
\end{eqnarray*}
A similar calculation shows that $(1\otimes\delta_0)*(\varphi\otimes h)=\varphi\otimes h$.  We also compute the involution of $1\otimes\delta_0$
\begin{eqnarray*}
 (1\otimes\delta_0)^*(n,z)&=&\overline{(1\otimes\delta_0)(-n,e^{-2\pi in\theta}z)}\\
&=&\delta_0(-n)\\
&=&1(z)\delta_0(n)\\
&=&a\otimes\delta_0(n,z).
\end{eqnarray*}
Clearly $1\otimes \delta_0$ is the unit element in $C_c(\mathbb{Z}\times \mathbb{T})$.
\end{proof}

\begin{lemma}
$u:=1\otimes\delta_1$ defines a unitary in $C_c(\mathbb{Z}\times \mathbb{T})$, where $1$ denotes the identity function in $C_c(\mathbb{T})$.
\end{lemma}
\begin{proof}
We perform the following calculation
 \begin{eqnarray*}
  u^**u(n,z)&=&\sum^{\infty}_{m=-\infty}u^*(m,z)u\left(n-m,e^{-2\pi in\theta}z\right)\\
&=&\sum^{\infty}_{m=-\infty}\overline{1\otimes\delta_1(-m,e^{-2\pi im\theta}z)}1\otimes \delta_1\left(n-m,e^{-2\pi im\theta}z\right)\\
&=&\sum^{\infty}_{m=-\infty}\delta_1(-m)\delta_1(n-m)\\
&=&\delta_1(n+1)\\
&=&\delta_0(n)\\
&=&1\otimes\delta_0(n,z).
 \end{eqnarray*}
With a similar calculation we can show that $u*u^*=1\otimes\delta_0$.  This shows that $u$ is indeed unitary.
\end{proof}

\begin{lemma}
 Define $\iota_{\mathbb{T}}(z):=z$ for all $z\in\mathbb{T}$.  The function defined by $v:=\iota_{\mathbb{T}}\otimes\delta_0$ is a unitary element in $C_c(\mathbb{Z}\times\mathbb{T})$.
\end{lemma}
\begin{proof}
 The lemma follows from the following calculation
\begin{eqnarray*}
 v^**v(n,z)&=&\sum^{\infty}_{m=-\infty}v^*(m,z)v\left(n-m,e^{-2\pi in\theta}z\right)\\
&=&\sum^{\infty}_{m=-\infty}\overline{\iota_{\mathbb{T}}\otimes\delta_0(-m,e^{-2\pi im\theta}z)}\iota_{\mathbb{T}}\otimes\delta_0\left(n-m,e^{-2\pi im\theta}z\right)\\
&=&\sum^{\infty}_{m=-\infty}\overline{\iota_{\mathbb{T}}\left(e^{-2\pi im\theta}z \right)}\delta_0(-m)\iota_{\mathbb{T}}\left(e^{-2\pi im\theta}z \right)\delta_0(n-m)\\
&=&\sum^{\infty}_{m=-\infty}\delta_0(-m)\delta_0(n-m)\\
&=&\delta_0(n)\\
&=&1\otimes\delta_0(n,z).
\end{eqnarray*}
With a similar calculation we can show that $v*v^*=1\otimes\delta_0$.  This shows that $v$ is indeed unitary.
\end{proof}

\begin{lemma}\label{cm}
 The unitaries $u$ and $v$ as defined in the previous two lemmas satisfy the commutation relation
\begin{equation*}
 uv=e^{2\pi i\theta}vu.
\end{equation*}
\end{lemma}
\begin{proof}
We will calculate the left hand side first,
 \begin{eqnarray*}
  u*v(n,z)&=&\sum^{\infty}_{m=-\infty}u(m,z)v\left(n-m,e^{-2\pi im\theta}z\right)\\
&=&\sum^{\infty}_{m=-\infty}\delta_1(m)\iota_{\mathbb{T}}\left(e^{-2\pi im\theta}z \right)\delta_0(n-m)\\
&=&\iota_{\mathbb{T}}\left(e^{-2\pi i\theta}z \right)\delta_0(n-1)\\
&=&e^{-2\pi i\theta}z\delta_1(n).
 \end{eqnarray*}
Similarly we calculate the right hand side,
\begin{eqnarray*}
 v*u(n,z)&=&\sum^{\infty}_{m=-\infty}v(m,z)u\left(n-m,e^{-2\pi im\theta}z\right)\\
&=&\sum^{\infty}_{m=-\infty}\iota_{\mathbb{T}}(z)\delta_0(m)\delta_1(n-m)\\
&=&z\delta_1(n).
\end{eqnarray*}
Hence we have
\begin{equation*}
 uv=e^{2\pi i\theta}vu.
\end{equation*}

\end{proof}

\begin{lemma}
Let $u=1\otimes\delta_1$ as defined previously. Let $i_{C(\mathbb{T})}(\varphi):=\varphi\otimes\delta_0$ for all $\varphi\in C(\mathbb{T})$.  Then
\begin{equation*}
 i_{C(\mathbb{T})}(\varphi)*u^n=\varphi\otimes\delta_n.
\end{equation*}
\end{lemma}
\begin{proof}
 Let us first consider the case where $n=1$.
\begin{eqnarray*}
 i_{C(\mathbb{T})}(\varphi)*u&=&\sum^{\infty}_{m=-\infty}\varphi\otimes\delta_0(m,z)1\otimes\delta_1\left(n-m,e^{-2\pi im\theta} \right)\\
&=&\sum^{\infty}_{m=-\infty}\varphi(z)\delta_0(m)\delta_1(n-m)\\
&=&\varphi(z)\delta_1(n)\\
&=&\varphi\otimes\delta_1(n,z)
\end{eqnarray*}
Suppose that
\begin{equation*}
 i_{C(\mathbb{T})}(\varphi)*u^k=\varphi\otimes\delta_k.
\end{equation*}
holds for any $k\in\mathbb{Z}$.  We want to show that it is also true for $k+1$.
\begin{eqnarray*}
 i_{C(\mathbb{T})}(\varphi)*u^{k+1}(n,z)&=&i_{C(\mathbb{T})}(\varphi)*u^k*u(n,z)\\
&=&\left(\varphi\otimes\delta_k\right)*u(n,z)\\
&=&\sum^{\infty}_{m=-\infty}\varphi\otimes\delta_k(m,z)u\left(n-m,e^{-2\pi im\theta}z\right)\\
&=&\sum^{\infty}_{m=-\infty}\varphi(z)\delta_k(m)\delta_1(n-m)\\
&=&\varphi(z)\delta_1(n-k)\\
&=&\varphi(z)\delta_{k+1}(n)\\
&=&\varphi\otimes\delta_{k+1}(n,z)
\end{eqnarray*}
Hence, the statement of the lemma follows from mathematical induction.
\end{proof}

\begin{lemma}
 Let $(C(\mathbb{T}),\mathbb{Z},\alpha)$ be a dynamical system and let $u$ and $i_{C(\mathbb{T})}$ be as in the previous lemma.  Then
\begin{equation*}
 u*i_{C(\mathbb{T})}(\varphi)*u^*=i_{C(\mathbb{T})}\left(\alpha_1(\varphi)\right).
\end{equation*}
\end{lemma}
\begin{proof}
 Let $\psi=i_{C(\mathbb{T})}(\varphi)*u^*$.  Then
\begin{eqnarray*}
 \psi(n,z)&=&\sum^{\infty}_{m=-\infty}i_{C(\mathbb{T})}(\varphi)(m,z)u^*\left(n-m,e^{-2\pi im\theta}z\right)\\
&=&\sum^{\infty}_{m=-\infty}\varphi(z)\delta_0(m)\delta_1(m-n)\\
&=&\varphi\otimes\delta_{-1}(n,z).
\end{eqnarray*}
We compute
\begin{eqnarray*}
 u*\psi(n,z)&=&\sum^{\infty}_{m=-\infty}u(m,z)\psi\left(n-m,e^{-2\pi im\theta}z\right)\\
&=&\sum^{\infty}_{m=-\infty}\delta_1(m)\varphi\otimes\delta_{-1}\left(n-m,e^{-2\pi im\theta}z\right)\\
&=&\varphi\otimes\delta_{-1}\left(n-1,e^{-2\pi i\theta}z\right)\\
&=&\varphi\left(e^{-2\pi i\theta}z\right)\delta_0(n).
\end{eqnarray*}
Finally, let us consider the right hand side:
\begin{eqnarray*}
 i_{C(\mathbb{T})}\left(\alpha_1(\varphi)\right)(n,z)&=&\alpha_1(\varphi)\otimes\delta_0(n,z)\\
&=&\alpha_1(\varphi)(z)\delta_0(n)\\
&=&\varphi\left(e^{-2\pi i\theta}z\right)\delta_0(n).
\end{eqnarray*}
This concludes the proof.
\end{proof}

\begin{remark}\label{span}
By Lemma \ref{induct} we see that
\begin{equation*}
 \text{span}\left\{i_{C(\mathbb{T})}(\varphi)*u^n=\varphi\otimes\delta_n:\varphi\in C(\mathbb{T}),n\in\mathbb{Z} \right\}
\end{equation*}
 is dense in $C_c(\mathbb{Z},C(\mathbb{T}))$, and hence dense in $C_c(\mathbb{Z}\times\mathbb{T})$.  Furthermore, since $\iota_{\mathbb{T}}$ separates the points of $\mathbb{T}$, which is clearly compact, and the subalgebra $\left\{\iota_{\mathbb{T}}:\iota_{\mathbb{T}}(z)=z \text{ for all } z\in \mathbb{T} \right\}$ clearly contains all the constant functions in $C(\mathbb{T})$ we conclude by the Stone-Weierstrass Theorem that $\iota_{\mathbb{T}}$ generates $C(\mathbb{T})$ as a $C^*$-algebra.  In particular we see that 
\begin{equation*}
 \text{span}\left\{v^n:n\in\mathbb{Z} \right\}
\end{equation*}
is a dense subalgebra of $i_{C(\mathbb{T})}(C(\mathbb{T}))$. 
\end{remark}

\begin{theorem}\label{crossed_product}
The quantum torus $A_{\theta}$ can be written as the crossed product $C(\mathbb{T})\rtimes_{\alpha} \mathbb{Z}$. 
\end{theorem}
\begin{proof}
 From Lemma \ref{cp} we know that $C(\mathbb{T})\rtimes_{\alpha} \mathbb{Z}$ is the completion of $C_c(\mathbb{Z},C(\mathbb{T}))$ in the universal norm.  By Remark \ref{span} it follows that the unitary operators $u=1\otimes\delta_1$ and $v=\iota_{\mathbb{T}}\otimes\delta_0$ generate $C(\mathbb{T})\rtimes_{\alpha}\mathbb{Z}$.  By Lemma \ref{cm} we see that these unitary operators satisfy the commutation relations
\begin{equation*}
 uv=e^{2\pi i\theta}vu
\end{equation*}
with $\theta$ irrational.  Hence from the definition of the quantum torus (see Definition \ref{QT}) we conclude that
\begin{equation*}
 A_{\theta}=C^*(u,v)=C(\mathbb{T})\rtimes_{\alpha}\mathbb{Z}.
\end{equation*}
\end{proof}

We state the next lemma without proof.  Details can be found in \cite[p. 96]{Williams}.
\begin{lemma}\label{dT}
 Suppose that $\mathbb{Z}$ acts on $\mathbb{T}$ by an irrational rotation
\begin{equation*}
 n\cdot z:=e^{2\pi i\theta}z.
\end{equation*}
Then for every $z\in\mathbb{T}$, the orbit $\mathbb{Z}\cdot z$ is dense in $\mathbb{T}$.
\end{lemma}

\begin{lemma}\label{murph}
Let $U$ and $V$ be any two unitary operator with commutation relation
\begin{equation*}
 UV=e^{e\pi i\theta}VU.
\end{equation*}
Then the spectrum of the unitary operators, $\sigma(U)=\mathbb{T}=\sigma(V)$. 
\end{lemma}
\begin{proof}
Note that since $U$ and $V$ are unitary, their spectra $\sigma(U)$ and $\sigma(V)$ are subsets of $\mathbb{T}$.  Let $\I$ denote the identity operator in $C^*(U,V)=A_{\theta}$.  Note that
\begin{eqnarray*} 
\lambda\in\sigma(V)&\Longleftrightarrow& V-\lambda\I \text{ is not invertible}\\
&\Longleftrightarrow& U^n(V-\lambda\I)\text{ is not invertible}\\
&\Longleftrightarrow& \left(e^{2\pi in\theta}V-\lambda\I\right)U^n \text{ is not invertible}\\
&\Longleftrightarrow& V-e^{-2\pi in\theta}\lambda\I \text{ is not invertible}\\
&\Longleftrightarrow& e^{-2\pi in\theta}\lambda\in \sigma(V).
\end{eqnarray*}
We could have switched the roles of $U$ and $V$ to find the relation for $\sigma(U)$.  Now, since $\sigma(V)$ must be nonempty, by Lemma \ref{dT} $\sigma(V)$ has to contain a dense subset of $\mathbb{T}$.  But since $\sigma(V)$ is closed the lemma follows.  
\end{proof}
Using the previous lemmas and propositions we can prove an important theorem in the theory of quantum tori known as the \textsl{universal property of the quantum torus}.

\begin{theorem}(Universal Property of the Quantum Torus for $\theta\in\mathbb{R}/\mathbb{Q}$)\label{universal}\newline
Let $B$ be a $C^*$-algebra with unitaries $u,v$ such that $vu=e^{2\pi i\theta}uv$.  Let $A_{\theta}$ be as in Definition \ref{QT} but with $\theta$ irrational.  Then there exists a *-isomorphism
\begin{equation*}
 \varphi:A_{\theta}\rightarrow C^*(u,v)\subset B
\end{equation*}
sending $U\mapsto u$ and $V\mapsto v$.
\end{theorem}
%%  DIT KAN OOK MISKIEN N GOEIE IDEE WEES OM DIE BEWYS VAN HIERDIE STELLING IN TE SLUIT IN DIE VERHANDELING %%
\begin{proof}
 Since any $C^*$-algebra can be seen as a *-subalgebra of $B(H)$ for some Hilbert space $H$ we only have to consider unitary operators $U$ and $V$ in $B(H)$ satisfying the commutation relation
\begin{equation*}
 UV=e^{2\pi i\theta}VU.
\end{equation*}
By Lemma \ref{murph} and \cite[Theorem 2.1.13]{murphy} we have a *-isomorphism
\begin{equation*}
 \pi:C(\sigma(V))=C(\mathbb{T})\rightarrow C^*(V)\subset C^*(U,V)
\end{equation*}
which maps $\iota_{\mathbb{T}}$ to $V$.  With $H$ a Hilbert space let $W:\mathbb{Z}\rightarrow U(H)$ be given by
\begin{equation*}
 W_n:=U^n.
\end{equation*}
Since $UV=e^{2\pi i\theta}VU$ we have
\begin{eqnarray*}
 W_n\pi(\iota_{\mathbb{T}})W^*_n&=&U^n\pi(\iota_{\mathbb{T}})U^{-n}\\
&=&U^nVU^{-n}\\
&=&e^{-2\pi in\theta}V\\
&=&\pi\left(\alpha_n(\iota_{\mathbb{T}})\right),
\end{eqnarray*}
where $\alpha_n$ came from the dynamical system $(C(\mathbb{T}),\mathbb{Z},\alpha)$.  We clearly see that $(\pi,W)$ is a covariant representation of $(C(\mathbb{T}),\mathbb{Z},\alpha)$.  Furthermore, let $u=1\otimes\delta_1$.  As in the previous lemmas, we then have the following
\begin{eqnarray*}
 \pi\rtimes W(u)&=&\sum^{\infty}_{m=-\infty}\pi\left(u(m,\cdot) \right)W_m\\
&=&\sum^{\infty}_{m=-\infty}\pi\left(\delta_{1}(m) \right)W_m\\
&=&\pi(\I_{C(\mathbb{T})})W_1\\
&=&U.
\end{eqnarray*}
Similarly, let $v=\iota_{\mathbb{T}}\otimes\delta_0$.  As in the previous lemmas, we then have the following
\begin{eqnarray*}
 \pi\rtimes W(v)&=&\sum^{\infty}_{m=-\infty}\pi\left(v(m,\cdot) \right)W_m\\
&=&\sum^{\infty}_{m=-\infty}\pi\left(\iota_{\mathbb{T}}(\cdot)\delta_{0}(m) \right)W_m\\
&=&\pi(\iota_{\mathbb{T}})\\
&=&V.
\end{eqnarray*}
Then clearly $L:=\pi\rtimes W$ is a *-isomorphism of $A_{\theta}$ into $C^*(U,V)$.
\end{proof}

The following result will be of importance in Chapter 3 where we are going to study Morita equivalence of different quantum tori.
\begin{lemma}\label{A_iso}
 $A_{\theta}$ is isomorphic to $A_{\theta+c}$ where $c$ is some integer.
\end{lemma}
\begin{proof}
 $A_{\theta}$ is generated by unitary operators $U$ and $V$ satisfying the commutation relation
\begin{equation*}
 UV=e^{2\pi i\theta}VU.
\end{equation*}
Similarly the quantum torus $A_{\theta+c}$ is generated by the unitary operators $\tilde{U}$ and $\tilde{V}$ which satisfy the commutation relation
\begin{eqnarray*}
 \tilde{U}\tilde{V}&=&e^{2\pi i\left(\theta+c\right)}\tilde{V}\tilde{U}\\
&=&\left(e^{2\pi i c}\right)e^{2\pi i\theta}\tilde{V}\tilde{U}\\
&=&e^{2\pi i\theta}\tilde{V}\tilde{U}
\end{eqnarray*}
So by that universal property of quantum tori we can now conclude the $A_{\theta}\cong A_{\theta +c}$.
\end{proof}
%%K-toerie van kwantum torus was eers hier gewees

\section{Trace of the Quantum Torus}\label{trace_int}\noindent
Classically we my define an example of a state on $C(\mathbb{T}^2)$ by
\begin{eqnarray*}
 \omega:C(\mathbb{T}^2)\rightarrow \mathbb{C}:
 f\mapsto\int^{1}_{0}\int^{1}_{0}f(x,y) dxdy.
\end{eqnarray*}
We would like to extend this definition to the noncommutative torus $A_{\theta}$.  Let $\Omega=\I\in L^2(\mathbb{T}^2)$ be the identity.  Then for any $f\in L^2(\mathbb{T}^2)$ we can write
\begin{eqnarray*}
 \langle \Omega,f\Omega\rangle&=&\int^{1}_{0}\int^{1}_{0}\overline{\Omega}f\Omega(x,y) dxdy\\
&=&\int^{1}_{0}\int^{1}_{0}f(x,y) dxdy.
\end{eqnarray*}
The inner product above motivates our definition of the linear functional $\tau$ given below.  For any $A\in A_{\theta}$
\begin{eqnarray}
 \tau(A)&:=&\langle \Omega, A\Omega\rangle\label{linfunc}\\
&=& \int^{2\pi}_{0}\int^{2\pi}_{0}\overline{\Omega(x,y)}A\Omega(x,y)  dxdy\nn\\
&=& \int^{2\pi}_{0}\int^{2\pi}_{0}A\Omega(x,y) dxdy\nn.
\end{eqnarray}
Now, recall that $A_{\theta}$ is the $C^*$-algebra generated by unitary operators $U$ and $V$ with commutation relation
\begin{equation*}
 UV=e^{2\pi i\theta}VU.
\end{equation*}
Substituting $U^mV^n$ into equation (\ref{linfunc}) we obtain
\begin{eqnarray*}
 \tau(A)&=&\tau\left(U^mV^n\right)\\
&=&\int^{2\pi}_{0}\int^{2\pi}_{0}U^mV^n\Omega(x,y) dxdy\\
&=&\int^{2\pi}_{0}\int^{2\pi}_{0}e^{2\pi i(mx+ny)}e^{\pi i\theta mn}\Omega\left(x-\frac{n\theta}{2},y+\frac{m\theta}{2}\right) dxdy\\
&=&\int^{2\pi}_{0}\int^{2\pi}_{0}e^{2\pi i(mx+ny)}e^{\pi i\theta mn} dxdy\\
&=&0.
\end{eqnarray*}
On the other hand, when $m=n=0$ we find
\begin{eqnarray*}
 \tau(A) = \tau\left(U^mV^n\right)=1
\end{eqnarray*}

\begin{theorem}\label{quantum_trace}
 $\tau$ as defined above, defines a trace on $A_{\theta}$.
\end{theorem}
\begin{proof}
In order to prove this theorem we need to show that $\tau$ is a tracial positive linear functional.  Let $A$ and $B$ be elements in $A_{\theta}$.  Linearity is trivial due to the inner product.  If $A$ is a positive element in $A_{\theta}$ we can write $A=S^*S$ for some $S\in A_{\theta}$.  Substituting this into the definition we find
\begin{eqnarray*}
 \tau(S^*S)&=&\langle \Omega,S^*S\Omega\rangle\\
&=&\langle S\Omega,S\Omega\rangle\\
&=&\|S\Omega\|^2\\
&\geq&0.
\end{eqnarray*}
Let $B$ be the *-algebra generated by $\{U,V\}$.  We want to show that $\tau(ab)=\tau(ba)$ for any $a=U^mV^n,b=U^jV^k\in B$,  
\begin{eqnarray*}
 \tau(U^mV^nU^jV^k)&=&\int^{2\pi}_{0}\int^{2\pi}_{0}U^mV^nU^jV^k\Omega(x,y) dxdy\\
&=&\int^{2\pi}_{0}\int^{2\pi}_{0}U^mV^ne^{2\pi i(jx+ky)}e^{\pi i\theta jk}\Omega\left(x-\frac{k\theta}{2},y+\frac{j\theta}{2}\right) dxdy\\
&=&\int^{2\pi}_{0}\int^{2\pi}_{0}e^{2\pi i\left((m+j)x+(n+k)y\right)}e^{\pi i\theta (mn+jk)}dx dy\\
&=&\int^{2\pi}_{0}\int^{2\pi}_{0}U^jV^ke^{2\pi i(mx+ny)}e^{\pi i\theta mn}\Omega\left(x-\frac{n\theta}{2},y+\frac{m\theta}{2}\right) dxdy\\
&=&\int^{2\pi}_{0}\int^{2\pi}_{0}U^mV^nU^jV^k\Omega(x,y) dxdy\\
&=&\tau(U^jV^kU^mV^n).
\end{eqnarray*}
Recall that the elements of $B$ consist of finite linear combinations of elements of the form $U^mV^n$ and that the above result holds for any $a,b\in B$.  Since $B$ is a dense *-subalgebra of the $C^*$-algebra $A_{\theta}$ we can extend the linear functional (\ref{linfunc}) to the whole quantum torus.
\end{proof}

\begin{lemma}\label{Ad}
Let $A$ be a $C^*$-algebra. The mapping defined by
\begin{equation*}
\text{Ad}u:A\rightarrow A:a\mapsto uau^*
\end{equation*}
with $u$ a unitary in $A$, is a *-isomorphism.
\end{lemma}
\begin{proof}
\begin{eqnarray*}
\left(\text{Ad}u\right)(a+b)&=&u(a+b)u^*\\
&=&uau^*+ubu^*\\
&=&\left(\text{Ad}u\right)(a)+\left(\text{Ad}u\right)(b)
\end{eqnarray*} 
\begin{eqnarray*}
\left(\text{Ad}u\right)(ab)&=&u(ab)u^*\\
&=&(uau^*)(ubu^*)\\
&=&\left(\text{Ad}u\right)(a)\left(\text{Ad}u\right)(b)
\end{eqnarray*}
\begin{eqnarray*}
\left(\text{Ad}u\right)(a^*)&=&ua^*u^*=(uau^*)^*=\left(\text{Ad}u\right)(a)^*
\end{eqnarray*}
So $\text{Ad}u$ is a *-homomorphism.  Let $x\in \text{ker}(\text{Ad}u)$
\begin{equation*}
\left(\text{Ad}u\right)(x)=uxu^*=0
\end{equation*}
This implies that $x=0$ and $\text{Ad}u$ is injective.  Let $a\in A$ be arbitrary
\begin{equation*}
\left(\text{Ad}u\right)(u^*au)=u(u^*au)u^*=a
\end{equation*}
So $\text{Ad}u$ is surjective.  We conclude that $\text{Ad}u$ is a bijective *-homomorphism and hence a *-isomorphism.
\end{proof}

\begin{remark}
The ideas and methods of dynamical systems constantly come into play throughout this thesis.  In the following lemma we will use a dynamical system approach to show that the canonical trace $\tau$ on the quantum torus is unique.  Consider the two *-automorphisms
\begin{eqnarray*}
 \alpha(a):=UaU^*,\quad
 \beta(a):=VaV^*.
\end{eqnarray*}
We observe that for any trace $\omega$ on $A_{\theta}$ we have
\begin{equation*}
 \omega\left(\alpha(a)\right)=\omega\left(UaU^*\right)=\omega(a)
\end{equation*}
 and a similar result for $\beta$.  Hence any trace on $A_{\theta}$ is invariant with respect to the above *-automorphisms.
\end{remark}

\begin{lemma}
Any state on $A_{\theta}$ that is invariant under both $\alpha$ and $\beta$ as defined above has to be the canonical trace $\tau$ on $A_{\theta}$. 
\end{lemma}
\begin{proof}
 Since all elements of $A_{\theta}$ can be expressed as finite products of elements of the form $U^mV^n$, where $U$ and $V$ are the two unitary operators generating the quantum torus, we only have to consider the action of the two *-automorphisms on $U^mV^n$
\begin{eqnarray*}
 \alpha\left( U^mV^n\right)&=&U^{m+1}V^nU^*\\
&=&e^{2\pi i\theta (m+1)n}V^nU^m
\end{eqnarray*}
and
\begin{eqnarray*}
 \beta\left( U^mV^n\right)&=&VU^{m}V^nV^*\\
&=&e^{-2\pi i\theta m}U^mV^n.
\end{eqnarray*}
Let $\omega$ be an arbitrary state on $A_{\theta}$ that is invariant under $\alpha$ and $\beta$.  We then have
\begin{eqnarray*}
 \omega\left( U^mV^n \right) &=& \omega\left[ \alpha \left( U^mV^n \right)\right]\\
&=&e^{2\pi i\theta (m+1)n}\omega\left( U^mV^n \right)
\end{eqnarray*}
and
\begin{eqnarray*}
 \omega\left( U^mV^n \right) &=& \omega\left[ \beta \left( U^mV^n \right)\right]\\
&=&e^{-2\pi i\theta m}\omega\left( U^mV^n \right).
\end{eqnarray*}
These solutions can only be realised when either both $m$ and $n$ are zero or $\omega\left(U^mV^n\right)=0$.  Comparing these to 
\begin{equation*}
 \tau\left(U^mV^n\right)=\left\{
\begin{array}{cc}
 0 & \text{if }m\neq n\\
 1 & \text{if }m=0=n
\end{array}
\right.
\end{equation*}
concludes the proof.
\end{proof}

%% hierdie afdeling was geskuif van voor na hier

\section{The Koopman Construction and the \\ natural action of $\mathbb{R}^2$ on $A_{\theta}$}
\noindent
In this section we study the natural action of $\mathbb{R}^2$ on the quantum torus in preparation for the following section in which we study the smooth elements of the quantum torus.  Ideas from dynamical systems such as those used in section \ref{trace_int} will once again play a major role in the present section, as well as those that follow.

We are familiar with the time evolution of operators in quantum mechanics, if $A$ is some operator on $B(H)$, we can describe the time evolution by
\begin{equation*}
 A(t)=U_tA(0)U_t^*
\end{equation*}
where $A(0)$ denotes the original operator and $A(t)$ the same operator at some later time.  $U$ is a unitary group which in quantum mechanics is normally written as
\begin{equation*}
 t\mapsto U_t=e^{-iHt}
\end{equation*}
with $H=H^*$ the Hamiltonian and $t$ the time.  In the next section we will use this point of view to describe dynamical systems on the quantum torus.

\begin{lemma}\label{koop}
 Let $L^2(\mu):=\mathcal{H}$ and define the operator
\begin{equation*}
 W:\mathcal{H}\rightarrow \mathcal{H}, f\mapsto f\circ T.
\end{equation*}
Let $\mu$ be a probability measure on the measure space $(X,\Sigma)$.  Define the operator $T$ in the following way:
\begin{enumerate}
 \item $T:X\rightarrow X$,
 \item $T^{-1}(Y)\in \Sigma$ and $\mu(T^{-1}(Y))=\mu(Y)$ for every $Y\in\Sigma$, and
 \item $T$ is invertible
\end{enumerate}
then $W$ is unitary.
\end{lemma}
\begin{proof}
We can define the operator
\begin{equation*}
 V:\mathcal{H}\rightarrow \mathcal{H}, f\mapsto f\circ T^{-1}.
\end{equation*}
Now clearly we have
\begin{eqnarray*}
 V(Wf)&=&V(f\circ T)=f\circ T^{-1}\circ T=f\\
 W(Vf)&=&W(f\circ T^{-1})=f\circ T^{-1}\circ T=f.
\end{eqnarray*}
So $W$ has an inverse.  Let $f,g\in \mathcal{H}$, the inner product is defined by
\begin{equation*} 
 \langle f,g \rangle=\int f g^*d\mu.
\end{equation*}
Now we consider
\begin{eqnarray*}
 \langle Wf,Wg\rangle&=&\langle f\circ T,g\circ T\rangle\\
&=&\int (f\circ T)(g\circ T)^*d\mu\\
&=&\int fg^*d\mu(T^{-1})\\
&=&\int fg^*d\mu\\
&=&\langle f,g \rangle,
\end{eqnarray*}
hence $W^*W=\I$ and $W$ is unitary.
\end{proof}

We would like to consider the action of the classical torus, $\mathbb{T}^2=\mathbb{R}^2/ 2\pi\mathbb{Z}^2$ on the quantum torus, however it is enough to consider only the action of $\mathbb{R}^2$ on the quantum torus since the contribution from $\mathbb{Z}^2$ does not play any part.  For convenience we state the generating operators of the quantum torus below:
\begin{eqnarray*}
 Uf(x,y)=e^{2\pi ix}f\left(x,y+\frac{\theta}{2}\right),\quad
Vf(x,y)=e^{2\pi iy}f\left(x-\frac{\theta}{2},y\right).
\end{eqnarray*}
We define the mapping
\begin{equation}\label{groupmap}
 \varphi_{r,s}:\mathbb{T}^2\rightarrow \mathbb{T}^2:(x,y)\mapsto (x+r,y+s).
\end{equation}
We are interested in square integrable functions on $\mathbb{T}^2$.  However any $f\in L(\mathbb{T}^2)$ is an element of the class of measurable functions $\left\{f:\int |f|^2 d\mu<\infty \right\}$ with respect to the normalized Haar measure on $\mathbb{T}^2$.  In our case we may identiry the Haar measure with the Lebesgue measure on $[0,1)\times [0,1)$ restricted to the Borel $\sigma$-algebra.  Now suppose we have $[f]=[g]$ in $L^2(\mathbb{T}^2)$.  This implies that $f=g$ almost everywhere except maybe on some set of zero measure.  Hence the same holds for the composition
$f\circ \varphi_{r,s}=g\circ \varphi_{r,s}$.  This enables us to well define the operator
\begin{equation*}
 W_{r,s}:\mathcal{H}\rightarrow\mathcal{H}:f\mapsto f\circ\varphi_{r,s}.
\end{equation*}
Since $\varphi_{-r,-s}=\varphi^{-1}_{r,s}$ we observe that $\varphi_{r,s}$ is invertible.  Furthermore, let $Y=[a,b)\times[c,d)$ be an element of $\Sigma$.  Then according to our construction we have
\begin{equation*}
 \mu(Y)=(b-a)(d-c).
\end{equation*}
After applying $\varphi_{-r,-s}$ to $Y$, we find
\begin{eqnarray*}
 \mu\left(\varphi_{-r,-s}(Y)\right)&=&\mu\left([a-r,b-r)\times[c-s,d-s)\right)\\
&=&(b-r-a+r)(d-s-c+s)\\
&=&(b-a)(d-c).
\end{eqnarray*}
Hence $\varphi^{-1}_{r,s}(Y)\in \Sigma$ for every $Y\in\Sigma$ and $\mu\left(\varphi^{-1}_{r,s}(Y)\right)=\mu(Y)$.  So $\varphi_{r,s}$ satisfies all the conditions of Lemma \ref{koop} and we can conclude that $W_{r,s}$ is unitary.  
\begin{lemma}
$(r,s)\mapsto W_{r,s}$ is a unitary group. 
\end{lemma}
\begin{proof}
Let $f\in L^2(\mathbb{T}^2)$ and consider the following equalities
 \begin{eqnarray*}
  W_{r+a,s+b}f&=&f\circ \varphi_{r+a,s+b}\\
&=&\left(f\circ\varphi_{a,b}\right)\circ\varphi_{r,s}\\
&=&W_{r,s}\left(f\circ\varphi_{a,b}\right)\\
&=&W_{r,s}W_{a,b}f.
 \end{eqnarray*}
We also have
\begin{eqnarray*}
W_{0,0}\left(W_{r,s}f\right)&=&W_{0,0}\left(f\circ\varphi_{r,s}\right)\\
&=&f\circ\varphi_{r,s}\circ\varphi_{0,0}\\
&=&f\circ\varphi_{r,s}\\
&=&W_{r,s}f.
\end{eqnarray*}
Similarly we find that $W_{r,s}W_{0,0}=W_{r,s}$.  Since we know that $W_{r,s}$ is unitary the lemma follows.
\end{proof}
\noindent
This is an example of the \emph{Koopman construction}.  Since $W_{r,s}$ is unitary we can use it to determine the time evolution of a dynamical system, similar to how time evolution arises in elementary quantum mechanics.  Now we define a mapping on $A_{\theta}$ which describes the evolution
\begin{equation*}
 \alpha_{r,s}:A_{\theta}\rightarrow A_{\theta}:a\mapsto W_{r,s}aW^*_{r,s}.
\end{equation*}
Suppose $a$ and $b$ are two distinct elements of $A_{\theta}$ and consider
\begin{eqnarray*}
 a \neq b &\Longleftrightarrow& W_{r,s}a\neq W_{r,s}b\\
&\Longleftrightarrow&W_{r,s}aW^*_{r,s}\neq W_{r,s}bW^*_{r,s}\\
 &\Longleftrightarrow&\alpha_{r,s}(a)\neq\alpha_{r,s}(b).
\end{eqnarray*}
So $\alpha_{r,s}$ is well defined.  Now we are able to consider the ``time'' evolution of the generating operators of the quantum torus.  We see that
\begin{eqnarray*}
 \left(\alpha_{r,s}(U)f\right)(x,y)&=&\left(W_{r,s}UW^*_{r,s} f\right)(x,y)\\
&=&(W_{r,s}h)(x,y)
\end{eqnarray*}
where we write
\begin{eqnarray*}
 h(x,y)&:=&\left(UW^*_{r,s}f\right)(x,y)\\
&=&Ug(x,y)
\end{eqnarray*}
with
\begin{eqnarray*}
 g(x,y)&:=&\left(W^*_{r,s}f\right)(x,y)\\
&=&f(x-r,y-s).
\end{eqnarray*}
Substituting this result back into the previous constructions we obtain
\begin{eqnarray*}
 h(x,y)&=&\left(Ug\right)(x,y)\\
&=&e^{2\pi i x}g(x,y+\frac{\theta}{2})\\
&=&e^{2\pi i x}f(x-r,y-s+\frac{\theta}{2})
\end{eqnarray*}
and finally
\begin{eqnarray*}
 (W_{r,s}h)(x,y)&=&h(x+r,y+r)\\
&=&e^{2\pi i r}e^{2\pi i x}f(x,y+\frac{\theta}{2})\\
&=&e^{2\pi i r}uf(x,y).
\end{eqnarray*}
We can perform a similar calculation to obtain the ``time'' evolution of $V$.  The final results are
\begin{eqnarray*}
 \alpha_{r,s}(U)&=&e^{2\pi i r}U\\
 \alpha_{r,s}(V)&=&e^{2\pi i s}V.
\end{eqnarray*}
The previous construction now enables us to construct a mapping from $\mathbb{R}^2$ to $A_{\theta}$ in a natural way, namely
\begin{eqnarray*}
 \alpha:\mathbb{R}^2\rightarrow A_{\theta}:(r,s)\mapsto \alpha_{r,s}(a)
\end{eqnarray*}
where $a\in A_{\theta}$.

\begin{proposition}
 $(A_{\theta},\mathbb{R}^2,\alpha)$ is a $C^*$-dynamical system.
\end{proposition}
\begin{proof}
 Let ${r_l}$ and ${s_l}$ be sequences in $\mathbb{T}$ respectively converging to $r$ and $s$.  Let $B$ be the *-algebra generated by $\{U,V\}$, the generating operators of the quantum torus.  Every element of $B$ is a finite linear combination of terms of the form $U^mV^n$, so
\begin{eqnarray*}
 \lim_{l\rightarrow\infty}\alpha(r_l,s_l)(U^mV^n)
&=&\lim_{l\rightarrow\infty}e^{2m\pi i r_l}e^{2n\pi i s_l}U^mV^n\\
&=&e^{2\pi i (mr+ns)}U^mV^n\\
&=&\alpha(r,s)(U^mV^n).
\end{eqnarray*}
By the uniqueness of the limit $\alpha$ is continuous on $B$.  We know that $B$ is dense in $A_{\theta}$ and we can extend the continuity to the whole of $A_{\theta}$.  Let $(r_n)$ and $(s_n)$ be two sequences in $\mathbb{T}$ respectively converging to $r$ and $s$.  Let $a\in A_{\theta}$.  Then we can find an upper bound for $\alpha_{r,s}$
\begin{eqnarray*}
 \|\alpha_{r,s}\|&=&\sup_{\|a\|=1}\|\alpha_{r,s}(a)\|\\
&=&\sup_{\|a\|=1}\|W_{r,s}aW^*_{r,s}\|\\
&\leq&1.
\end{eqnarray*}
Given $b\in B$ and $\epsilon>0$, then for some $n$ large enough, we have  
\begin{equation*}
 \|\alpha_{r_n,s_n}(b)-\alpha_{r,s}(b)\|<\frac{\epsilon}{3}.
\end{equation*}
Since $B$ is dense in $A_{\theta}$ we have for some $b\in B$ and $a\in A_{\theta}$ that
\begin{equation*}
 \|b-a\|<\frac{\epsilon}{3}.
\end{equation*}
Now we show that the continuity can be extended to the whole of $A_{\theta}$:
\begin{eqnarray*}
 &&\|\alpha_{r_n,s_n}(a)-\alpha_{r,s}(a)\|\\
 &=&\|\alpha_{r_n,s_n}(a)-\alpha_{r_n,s_n}(b)+\alpha_{r_n,s_n}(b)-\alpha_{r,s}(b)+\alpha_{r,s}(b)-\alpha_{r,s}(a)\|\\
 &\leq&\|\alpha_{r_n,s_n}(a)-\alpha_{r_n,s_n}(b) \|+\|\alpha_{r_n,s_n}(b)-\alpha_{r,s}(b) \|+\|\alpha_{r,s}(b)-\alpha_{r,s}(a) \|\\
 &\leq&\|\alpha_{r_n,s_n}\|\|a-b\|+\|\alpha_{r_n,s_n}(b)-\alpha_{r,s}(b) \|+\|\alpha_{r,s}\|\|a-b\|\\
 &<&\frac{\epsilon}{3}+\frac{\epsilon}{3}+\frac{\epsilon}{3}\\
 &=&\epsilon.
\end{eqnarray*}
To conclude, we have a $C^*$-algebra $A_{\theta}$, a group $\mathbb{R}^2$ and a continuous homomorphism $(r,s)\mapsto \alpha_{r,s}(a)$ for some $s\in A_{\theta}$ and according to Definition \ref{dynamical_system} $\left(A_{\theta},\mathbb{R}^2,\alpha\right)$ is a $C^*$-dynamical system.
\end{proof}

\section{Derivations and the smooth algebra $A^{\infty}_{\theta}$}\label{der}

\begin{definition}(Derivations)\cite[p. 153]{Sakai}\newline
 Let $A$ be a $*$-algebra and let $\delta$ be a linear mapping of $A$ into itself.  $\delta$ is called a $*$-\emph{derivation} if
\begin{eqnarray*}
 \delta(xy)=\delta(x)y+x\delta(y),\quad
 \delta(x^*)=\delta(x)^*
\end{eqnarray*}
for all $x,y\in A$.  
\end{definition}

\noindent
\begin{remark}\label{time_evo}
Let $A(0)$ be an operator in $B(\mathcal{H})$ with $\mathcal{H}$ our Hillbert space.  When we consider the Heisenberg picture of quantum mechanics we can determine the time evolution of $A(0)$ in the following familiar fashion:
\begin{eqnarray*}
 \frac{d}{dt}A(t)&=&\frac{d}{dt}\left(e^{-iHt}A(0)e^{iHt}\right)\\
&=&-iHe^{-iHt}A(0)e^{iHt}+ie^{-iHt}A(0)He^{iHt}.
\end{eqnarray*}
We would like to focus on a particular point of interest which will suppress the exponential functions.  This happens when $t=0$.
\begin{eqnarray*}
 \left.\frac{d}{dt}A(t)\right|_{t=0}&=&-iHA(0)+iA(0)H\\
&=&\frac{1}{i}\left[H,A(0)\right].
\end{eqnarray*}
If we define 
\begin{eqnarray*}
 \delta(\cdot):=\frac{1}{i}\left[H,\cdot\right]
\end{eqnarray*}
and take two operators $A$ and $B$ on $B(\mathcal{H})$ where $\mathcal{H}$ is the Hilbert space in question, we calculate
\begin{eqnarray*}
 \delta(AB)&=&\frac{1}{i}\left[H,AB\right]\\
&=&\frac{1}{i}\left([H,A]B+A[H,B]\right)\\
&=&\frac{1}{i}[H,A]B+\frac{1}{i}A[H,B]\\
&=&\delta(A)B+A\delta(B)
\end{eqnarray*}
and
\begin{eqnarray*}
 \delta(A^*)&=&\frac{1}{i}\left[H,A^*\right]\\
&=&\frac{1}{i}\left(HA^*-A^*H\right)\\
&=&\frac{1}{i}\left(AH-HA\right)^*\\
&=&\left(\frac{1}{i}\left[H,A\right]\right)^*\\
&=&\delta(A)^*.
\end{eqnarray*}
This formal calculation suggests that indeed $\delta$ is a derivation.  Now using a similar approach to the one used in quantum mechanics we will define derivations on the irrational rotation algebra which reduce to normal partial derivatives in the classical scenario.
\end{remark}

We require a definition for derivatives of operator valued functions.  Let $A$ be a normed algebra and consider a mapping
\begin{equation*}
 f:\mathbb{R}\rightarrow A.
\end{equation*}
If there exists an $f'(x)\in A$ such that
\begin{equation}\label{derivative}
 0=\lim_{h\rightarrow 0}\frac{\|f(x+h)-f(x)-f'(x)h \|}{\|h\|}
\end{equation}
then we say $f'(x)$ is the derivative of $f$ at $x$.  Note that sometimes we make use of the notation $\frac{d}{dx}f(x)$ to denote the derivative of $f$ at $x$ and sometimes switch between the two notations.  Furthermore note that the existence of $f'(x)$ implies the continuity of $f$ at $x$.
\begin{proposition}\label{product_rule}
 The ``product rule'', ``sum rule'' as well as the rule for scalar multiplication holds for the above definition of the derivative of an operator valued function.  Furthermore we also have
\begin{equation*}
 \frac{d}{dx}\left(f(x)^*\right)=\left(\frac{d}{dx}f(x)\right)^*.
\end{equation*}
\end{proposition}
\begin{proof}
 Let $A$ be a normed algebra and consider two mappings $f,g:\mathbb{R}\rightarrow A$.  We would like to show that 
\begin{equation*}
 \left(f g \right)'(x)=f'(x)g(x)+f(x)g'(x).
\end{equation*}
Consider the definition
\begin{eqnarray*}
 &&\frac{\|f(x+h)g(x+h)-f(x)g(x)-f'(x)g(x)h-f(x)g'(x)h \|}{\|h\|}\\
 &=&\frac{1}{\|h\|}\left\|\left(f(x+h)-f(x)\right)g(x+h)\right.\\
&&\left.+f(x)\left(g(x+h)-g(x)\right)-f'(x)g(x)h-f(x)g'(x)h \right\|\\
 &\leq&\frac{1}{\|h\|}\left(\|f(x+h)-f(x)-f'(x)h\|\|g(x+h)\|\right.\\
&&\left.+\|g(x+h)-g(x)-g'(x)h\|\|f(x)\|+\|f'(x)\left(g(x+h)-g(x)\right)h\|\right)\\
&\leq&\frac{1}{\|h\|}\|f'(x)\left(g(x+h)-g(x)\right)\|\|h\|\\
&=&\|f'(x)\left(g(x+h)-g(x)\right)\|.
\end{eqnarray*}
Now if we take the limit as $h$ goes to zero.  We find that
\begin{equation*}
 \lim_{h\rightarrow 0}\frac{\|f(x+h)g(x+h)-f(x)g(x)-f'(x)g(x)h-f(x)g'(x)h \|}{\|h\|}=0.
\end{equation*}
Hence by the uniqueness of the limit we can conclude that the product rule holds for the derivative of operator valued functions.  Now suppose that both $f'(x)$ and $g'(x)$ exist, we would like to show that
\begin{equation*}
 (f+g)'(x)=f'(x)+g'(x).
\end{equation*}
Consider the following:
\begin{eqnarray*}
 &&\frac{\|(f+g)(x+h)-(f+g)(x)-(f'(x)+g'(x))h\|}{|h|}\\
&=&\frac{\|\left(f(x+h)-f(x)-f'(x)h\right)+\left(g(x+h)-g(x)-g'(x)h\right)\|}{|h|}\\
&\leq&\frac{\|f(x+h)-f(x)-f'(x)h\|}{|h|}+\frac{\|g(x+h)-g(x)-g'(x)h\|}{|h|}.
\end{eqnarray*}
Taking the limit as $h$ goes to zero we find that the ``sum rule'' holds.  The scalar multiplication result follows trivially.  

Finally, suppose that both $\frac{d}{dx}\left(f(x)^*\right)$ and $\left[\frac{d}{dx}\left(f(x)\right)\right]^*$ exist.  Hence we have
\begin{eqnarray*}
 0&=&\lim_{h\rightarrow 0}\frac{\left\|f(x+h)^*-f(x)^*-\frac{d}{dx}\left(f(x)^*\right)h \right\|}{\|h\|}\\
 &=&\lim_{h\rightarrow 0}\frac{\left\|f(x+h)^*-f(x)^*-\left[\frac{d}{dx}\left(f(x)\right)\right]^*h \right\|}{\|h\|}.
\end{eqnarray*}
By the uniqueness of the limit we conclude that $\frac{d}{dx}\left(f(x)^*\right)=\left[\frac{d}{dx}\left(f(x)\right)\right]^*$.
\end{proof}
\begin{remark}
 Since we are considering functions from $\mathbb{R}$ to some noncommutative algebra we can simplify (\ref{derivative}) to
\begin{equation*}
 0=\lim_{h\rightarrow 0}\left\|\frac{f(x+h)-f(x)}{h}-f'(x) \right\|.
\end{equation*}
If we write the derivative as above the connection with the normal definition of the derivative becomes clear.
\end{remark}
\begin{definition}(Smooth, $C^{\infty}$)\newline
 Let $A$ be a $C^*$-algebra.  Then we say a function $f:\mathbb{R}^n\rightarrow A$ is smooth or of $C^{\infty}$ class if all the partial derivatives
\begin{eqnarray*}
 \partial_{x_{j_1}}\cdots\partial_{x_{j_m}}f(x_1,\cdots x_n)
\end{eqnarray*}
exist for any $j_{1},\cdots, j_{m}\in\{1,\cdots,n\}$ and any $m\in \{0,1,2,\cdots\}$.  

If there exists a $\partial_{x_{j}}f(x_1,\cdots,x_n)\in A$ such that
\begin{equation*}
 0=\lim_{h\rightarrow 0}\frac{\|f(x_1,\cdots,x_j+h,\cdots,x_n)-f(x_1,\cdots,x_j,\cdots,x_n)-\partial_{x_j}f(x_1,\cdots,x_n)h\|}{\|h\|}
\end{equation*}
then we say that $\partial_{x_j}f(x_1,\cdots, x_n)$ is the partial derivative of $f(x_1,\cdots,x_n)$ at $x_j$.
%% Hierdie definisie moet steeds verbeter word %%
\end{definition}

\begin{remark}
For ease of calculation it will be useful to write these partial derivatives in terms of the accent notation.  For this purpose we define
\begin{equation*}
 \partial_{x_j}f(x_1,\cdots,x_n):=g'(x_j)
\end{equation*}
with
\begin{equation*}
 g(x_j)=f(x_1,\cdots,x_j,\cdots,x_n),
\end{equation*}
where all the $x_i$'s are fixed except for $x_j$.
\end{remark}
Similar to Connes \cite{Connes1980} we define the smooth quantum torus as follows:
\begin{definition}(\textsl{Smooth Quantum Torus})\newline
 Let $(A,\mathbb{R}^n,\alpha)$ be a $C^*$ dynamical system.  We shall say that $x \in A$ is of $C^{\infty}$ class if and only if the map $g \mapsto \alpha_g(x)$ from $\mathbb{R}^n$ to the normed space $A$ is $C^{\infty}$.  The \textsl{smooth quantum torus} is defined to be
\begin{equation*}
 A^{\infty}_{\theta}:=\left\{a\in A_{\theta}:a\text{ is of }C^{\infty}\right\}.
\end{equation*}
\end{definition}
\begin{proposition}
 The space $A^{\infty}_{\theta}$ is a *-algebra.
\end{proposition}
\begin{proof}
 Consider the $C^*$-dynamical system $(A_{\theta},\mathbb{R}^2,\alpha_{r,s})$.  Let $a,b\in A^{\infty}_{\theta}$, since $\alpha_{r,s}$ is a *-automorphism linearity follows
\begin{eqnarray*}
 (r,s)\mapsto\alpha_{r,s}(a+b)&=&\alpha_{r,s}(a)+\alpha_{r,s}(b).
\end{eqnarray*}
But, both the mappings 
\begin{eqnarray*}
 (r,s)\mapsto\alpha_{r,s}(a), \quad
 (r,s)\mapsto\alpha_{r,s}(b)
\end{eqnarray*}
are of $C^{\infty}$ class, hence so is the mapping 
\begin{equation*}
(r,s)\mapsto\alpha_{r,s}(a+b).
\end{equation*}
Let $k\in\mathbb{C}$ and $a\in A^{\infty}_{\theta}$ and consider 
\begin{eqnarray*}
 (r,s)\mapsto\alpha_{r,s}(ka)=k\alpha_{r,s}(a).
\end{eqnarray*}
But $(r,s)\mapsto\alpha_{r,s}(a)$ is of class $C^{\infty}$, so clearly scalar multiplication also holds.  Similarly using the fact that $\alpha_{r,s}$ is a *-automorphism we can consider the mapping
\begin{eqnarray*}
 (r,s)\mapsto\alpha_{r,s}(ab)
=\alpha_{r,s}(a)\alpha_{r,s}(b).
\end{eqnarray*}
By Proposition \ref{product_rule} we know that the ``product rule'' holds for operator valued functions of $\mathbb{R}$ and we can conclude that the mapping
\begin{equation*}
 (r,s)\mapsto\alpha_{r,s}(ab)
\end{equation*}
is of $C^{\infty}$ class.  $A^{\infty}_{\theta}$ inherits its involution directly from $A_{\theta}$.  We can consider the mapping
\begin{eqnarray*}
 (r,s)&\mapsto&\alpha_{r,s}(a^*)=\alpha_{r,s}(a)^*
\end{eqnarray*}
which implies that $a^*$ is of $C^{\infty}$ class.  This shows that $A^{\infty}_{\theta}$ is indeed a *-algebra.
\end{proof}
The next proposition shows that the smooth irrational rotation algebra is not empty.

\begin{lemma}\label{u_and_v}
Let $B$ be the *-algebra generated by the generating unitaries $U$ and $V$ of the quantum torus.  Then $B$ is contained in $A^{\infty}_{\theta}$.
\end{lemma}
\begin{proof}
 We can consider the mapping
\begin{eqnarray*}
 (r,s)&\mapsto&\alpha_{r,s}(U^mV^n)\\
&=&\alpha_{r,s}(U)^m\alpha_{r,s}(V)^n\\
&=&e^{2\pi i(rm+ns)}U^mV^n
\end{eqnarray*}
We substitute the above result into the definition for the derivative of an operator valued function to obtain
\begin{eqnarray*}
 &&\lim_{h\rightarrow 0}\frac{1}{\|h\|}\|\alpha_{r+h,s}(U^mV^n)-\alpha_{r,s}(U^mV^n)-2\pi ime^{2\pi i(rm+ns)}U^mV^nh\|\\
&=&\lim_{h\rightarrow 0}\frac{1}{\|h\|}\|e^{2\pi i(rm+ns)}\left(e^{2\pi ihm}U^mV^n-U^mV^n-2\pi imU^mV^nh\right)\|\\
&\leq&\lim_{h\rightarrow 0}\frac{1}{\|h\|}\|e^{2\pi ihm}-1-2\pi imh\|\|U^mV^n\|\\
&=&0.
\end{eqnarray*}
Performing similar calculations we can find the $s$-derivative of $U^mV^n$ and we can extend the results to partial derivatives of any order.  By the uniqueness of the limit we can conclude that $U^mV^n\in A^{\infty}_{\theta}$.  Hence $A^{\infty}_{\theta}$ contains all finite linear combinations of elements of the form $U^mV^n$ and hence contains $B$.
\end{proof}

\begin{proposition}
The smooth irrational rotation algebra $A^{\infty}_{\theta}$ is dense in $A_{\theta}$. 
\end{proposition}
\begin{proof}
 Let $B$ be the *-algebra generated by the unitaries $U$ and $V$ that generate the quantum torus.  By Lemma \ref{u_and_v} we know that $B$ is contained in $A^{\infty}_{\theta}$.  Furthermore we know that $B$ is dense in $A_{\theta}$, hence $A^{\infty}_{\theta}$ is also dense in $A_{\theta}$.
\end{proof}
We will now define derivations on $A_{\theta}$ in accordance with Remark \ref{time_evo}.
\begin{equation}\label{define_derivations}
 \delta_1(a):=\left.\frac{d}{dr}\alpha_{r,0}(a)\right|_{r=0}, \quad
 \delta_2(a):=\left.\frac{d}{ds}\alpha_{0,s}(a)\right|_{s=0}
\end{equation}

\begin{proposition}
 $\delta_i:A^{\infty}_{\theta} \rightarrow A^{\infty}_{\theta}$ as defined in equation (\ref{define_derivations}) is a *-derivation for $i=1,2$.
\end{proposition}

\begin{proof}
Let $a,b\in A_{\theta}$.  We can write
\begin{eqnarray*}
 \delta_1(ab)&=&\left.\frac{d}{dr}\alpha_{r,0}(ab)\right|_{r=0}\\
&=&\left.\frac{d}{dr}\left(\alpha_{r,0}(a)\alpha_{r,0}(b)\right)\right|_{r=0}\\
&=&\left.\left(\frac{d}{dr}\alpha_{r,0}(a)\right)\alpha_{r,0}(b)\right|_{r=0}+\left.\alpha_{r,0}(a)\left(\frac{d}{dr}\alpha_{r,0}(b)\right)\right|_{r=0}\\
&=&\delta_1(a)b+a\delta_1(b).
\end{eqnarray*}
A similar arguments holds for $\delta_2$.  Linearity follows trivially from 
\begin{eqnarray*}
 \delta_1(a+b)&=&\left.\frac{d}{dr}\alpha_{r,0}(a+b)\right|_{r=0}\\
&=&\left.\frac{d}{dr}\left[\alpha_{r,0}(a)+\alpha_{r,s}(b)\right]\right|_{r=0}\\
&=&\delta_1(a)+\delta_1(b).
\end{eqnarray*}
Let $\lambda \in\mathbb{C}$ and consider
\begin{eqnarray*}
 \delta_1(\lambda a)&=&\partial_t\alpha_{t,0}(\lambda a)|_{t=0}\\
&=&\lambda\partial_t\alpha_{t,0}(a)\\
&=&\lambda\delta_1(a).
\end{eqnarray*}
Similar calculations holds for $\delta_2$.  Hence we have $\delta_j(\lambda a)=\lambda\delta_j(a)$.  Furthermore we have
\begin{eqnarray*}
 \delta_1(a^*)&=&\partial_t\alpha_{r,0}(a^*)\\
&=&\partial_t\alpha_{t,0}(a)^*\\
&=&\delta_1(a)^*
\end{eqnarray*}
and a similar calculation holds for $\delta_2$.  Hence we have $\delta_j(a^*)=\delta_j(a)^*$.  Lastly we need to show that the mapping
\begin{equation*}
 (r,s)\mapsto\alpha_{r,s}\left(\delta_i(a)\right)
\end{equation*}
is of $C^{\infty}$ class.  We will first consider the case of $\delta_1(a)$
\begin{eqnarray*}
 &&\left\|\frac{\alpha_{r+h,s}(\delta_1(a))-\alpha_{r,s}(\delta_1(a))}{h}-\partial_r\alpha_{r,s}(\delta_1(a))\right\|\\
&=&\left\|\frac{\alpha_{r+h,s}(\partial_t\alpha_{t,0}(a)|_{t=0})-\alpha_{r,s}(\partial_t\alpha_{t,0}(a)|_{t=0}}{h}-\partial_r\alpha_{r,s}(\partial_t\alpha_{t,0}(a)|_{t=0})\right\|
\end{eqnarray*}
Here we have a problem, it is not known if the partial derivatives $\partial_r$ and $\partial_t$ commute.  The problem then is with the term of the form
\begin{eqnarray*}
 \partial_r\alpha_{r,0}\left(\partial_{t}\alpha_{t,s}(a)\right).
\end{eqnarray*}
We know that $a\in A^{\infty}_{\theta}$ and hence is of class $C^{\infty}$.  Furthermore since $\alpha_{t,s}$ is a *-isomorphism from $A^{\infty}_{\theta}$ to itself, $\alpha_{t,s}(a)$ also is of class $C^{\infty}$ and the same holds for $\partial_t\alpha_{t,s}(a)$.  Hence the mapping
\begin{equation*}
 (p,q)\mapsto\alpha_{p,q}\left(\partial_{t}\alpha_{t,s}(a)\right)
\end{equation*}
is of class $C^{\infty}$ and there exists an operator $\partial_p \alpha_{p,q} \left(\partial_t \alpha_{t,s}(a)\right)$ such that
\begin{eqnarray*}
 0=\lim_{h\rightarrow \infty}\left\|\frac{\alpha_{p+h,q}\left(\partial_{t}\alpha_{t,s}(a)\right)-\alpha_{p,q}\left(\partial_{t}\alpha_{t,s}(a)\right)}{h} -\partial_p \alpha_{p,q} \left(\partial_t \alpha_{t,s}(a)\right) \right\|.
\end{eqnarray*}
Then in particular we can choose $p=r$ and $q=0$, then
\begin{equation*}
 (r,0)\mapsto\alpha_{r,0}\left(\partial_{t}\alpha_{t,s}(a)\right)
\end{equation*}
is of class $C^{\infty}$ and $\partial_r \alpha_{r,0} \left(\partial_t \alpha_{t,s}(a)\right)$ exists.  Hence, we see that
\begin{equation*}
 \lim_{h\rightarrow 0}\left\|\frac{\alpha_{r+h,s}(\delta_1(a))-\alpha_{r,s}(\delta_1(a))}{h}-\partial_r \alpha_{r,s}(\delta_1(a))\right\|=0.
\end{equation*}
The same can be shown in the case of $\delta_2(a)$.  And a similar procedure follows for the higher order derivatives.
\end{proof}
\begin{remark}
 According to the previous theorem we can write the derivations in the following form:
\begin{eqnarray}\label{derivations_def}
\begin{array}{ccccccc}
 \delta_1(u) &=& 2\pi iu &,& \delta_2(u) &=& 0 \\
 \delta_1(v) &=& 0       &,& \delta_2(v) &=& 2\pi iv.
\end{array}
\end{eqnarray}
\end{remark}

\subsection{Classical Limit}\label{clas}
We may ask if the above constructions reduce to their respective classical counterparts in the case where $\theta$ is zero and $A_{\theta}\cong C(\mathbb{T}^2)$ as in Proposition \ref{kontinu}.  For this let us first consider the time evolution in the classical limit.  Let $h=U^*_{r,s}$ and $l=gh$.  Suppose $g\in C(\mathbb{T}^2)$ and consider the following:
\begin{eqnarray*}
 \left[\tau_{r,s}(g)\right]f(x,y)&=&\left[U_{r,s}gU^*_{r,s}f\right](x,y)\\
&=&\left[U_{r,s}gh\right](x,y)\\
&=&\left[U_{r,s}l\right](x,y)\\
&=&l(x+r,y+s)\\
&=&g(x+r,y+s)h(x+r,y+s)\\
&=&g(x+r,y+s)f(x,y).
\end{eqnarray*}
Clearly we observe that
\begin{equation*}
 \tau_{r,s}(g)=g\circ\varphi_{r,s}
\end{equation*}
which is exactly the evolution we would expect to find classically.  Similarly we can show that in the classical limit the derivations as defined above reduce to the corresponding partial derivatives.  This result is also mentioned in \cite{Rosenberg2008}.  Consider the derivation $\delta_1$ in the classical limit:
\begin{eqnarray*}
\left[\delta_1(g)\right](x,y)&=&\left[\left.\frac{\partial}{\partial r}\tau_{r,0}(g)\right|_{r=0}\right](x,y)\\
&=&\left.\frac{\partial}{\partial r}g(x+r,y)\right|_{r=0}.
\end{eqnarray*}
Let $z=x+r$.  Then according to the chain rule we can write 
\begin{equation*}
 \frac{\partial g}{\partial r}=\frac{\partial g}{\partial z}\frac{\partial z}{\partial r}=\frac{\partial g}{\partial z}
\end{equation*}
and substitute this back into the previous result to find
\begin{eqnarray*}
 \left[\delta_1(g)\right](x,y)&=&\left.\frac{\partial}{\partial z}g(z,y)\right|_{z=x}\\
&=&\frac{\partial}{\partial x}g(x,y).
\end{eqnarray*}
So we see that
\begin{equation*}
 \delta_1=\frac{\partial}{\partial x}
\end{equation*}
in the classical limit.  Similarly we can show that we also have
\begin{equation*}
 \delta_2=\frac{\partial}{\partial y}
\end{equation*}
in the classical limit.

%%miskien moet ek hier n subafdeling insit
\begin{lemma}
 The unique trace $\tau$ on the quantum torus is preserved under the action of the classical torus.
\end{lemma}
\begin{proof}
Let $B$ be the *-algebra generated by the generating unitaries $U$ and $V$ of the quantum torus.  Since every element of $B$ is a finite linear combination of terms of the form $Y^mV^n$ we only need to concern ourselves with these terms.  Observe that
\begin{eqnarray*}
 \tau\left(\alpha_{r,s}\left(U^mV^n\right)\right)&=&\tau\left(e^{2\pi i(rm+sn)}U^mV^n\right)\\
&=&e^{2\pi i(rm+sn)}\tau\left(U^mV^n\right)\\
&=&\left\{\begin{array}{cc}
           0 & \text{if } m\neq n\\
	   1 & \text{if } m=n=0
          \end{array}\right.
\end{eqnarray*}
Since $B$ is dense in $A_{\theta}$ we can extend this result to the whole of $A_{\theta}$ and the result follows.
\end{proof}\noindent
We have found an analog of partial derivatives on the quantum torus.  We can consider these mappings as the infinitesimal generators of the action of $\mathbb{T}^2$ on the quantum torus.
\begin{lemma}\label{useful}
 Let $\tau$ be the unique trace on $A_{\theta}$ and let $\delta_j$ with $j=1,2$ be the derivations defined above.  Then $\tau\circ \delta_j=0$.
\end{lemma}
\begin{proof}
As usual let $B$ denote the *-algebra generated by the generating unitaries $U$ and $V$ of the quantum torus.  Since all elements of $B$ are finite linear combinations of terms of the form $U^mV^n$ we only need to concern ourselves with terms of this form.  Let us first consider the derivation $\delta_1$.  Observe that
\begin{eqnarray*}
\delta_1\left(U^mV^n\right)&=&\left.\frac{d}{dr}\alpha_{r,0}\left(U^mV^n \right)\right|_{r=0}\\
&=&\left.\frac{d}{dr}\left(e^{2\pi i(rm+sn)}U^mV^n \right)\right|_{r=0}\\
&=&\left.\left(2\pi i m e^{2\pi i(rm+sn)}U^mV^n \right)\right|_{r=0}\\
&=&2\pi i m e^{2\pi i sn}U^mV^n.
\end{eqnarray*}
Now substitute this result into the unique trace of the quantum torus to get
\begin{eqnarray*}
 \tau\left( \delta_1(U^mV^n)\right)&=&\tau\left( 2\pi i m e^{2\pi i sn}U^mV^n\right)\\
&=&2\pi i m e^{2\pi i sn}\tau\left(U^mV^n\right)\\
&=&\left\{\begin{array}{cc}
    0 & \text{if } m\neq n\\
    0 & \text{if } m=n=0
   \end{array}\right.\\
&=&0.
\end{eqnarray*}
A similar calculation can be performed for $\delta_2$.  Since $B$ is dense in $A_{\theta}$ we can extend the result to the whole of $A_{\theta}$.  This then concludes the proof.
\end{proof}
The next lemma can be seen as the noncommutative generalization of integration by parts to irrational rotation algebras
\begin{lemma}\cite{Rosenberg2008}\label{integration_by_parts}\newline
 If $a,b\in A^{\infty}_{\theta}$, then $\tau(\delta_j(a)b)=-\tau(a\delta_j(b))$ for $j=1,2$.
\end{lemma}
\begin{proof}
 From Lemma \ref{useful} we have
\begin{equation*}
 0=(\tau\circ\delta_j)(ab)=\tau(\delta_j(a)b)+\tau(a\delta_j(b)).
\end{equation*}
The result clearly follows from the above calculation.
\end{proof}

\section{Finite Dimensional Representation}\label{finite_representation}\noindent  
In this section we construct a matrix model generated by two unitary matrices which satisfy the same commutation relation as the two unitary operators that generate the quantum torus.  The noncommutative behaviour of the matrices is then naturally encoded into the model and it serves as a good example of how we can think about noncommutative actions.  We follow the outline of \cite{Madore1999}.

We can choose an orthonormal basis $|j\rangle_1$ with $0\leq j\leq n-1$ of $\mathbb{C}^n$ and set $|n\rangle_1=|0\rangle_1$.  Introduce the unitary matrices $u,v$ in $M_n(\mathbb{C})$ by their respective action on these basis vectors by setting
\begin{eqnarray*}
 u|j\rangle_1=q^j|j\rangle_1,\quad
 v|j\rangle_1=|j+1\rangle_1.
\end{eqnarray*}
Here $q$ is some complex number such that $|q|=1$ and we assume $q^n=1$.  They satisfy 
\begin{equation}\label{com_rel}
 uv=qvu, \quad u^n=1, \quad v^n=1.
\end{equation} 
We can choose another orthonormal basis $|j\rangle_2$ in which $v$ is diagonal.  These two bases can be related by the Fourier transform
\begin{eqnarray*}
 |j\rangle_1=\frac{1}{\sqrt{n}}\sum^{n-1}_{l=0}q^{jl}|l\rangle_2, \quad
 |l\rangle_2=\frac{1}{\sqrt{n}}\sum^{n-1}_{j=0}q^{-jl}|j\rangle_1.
\end{eqnarray*} 
Where $k$ and $r$ are non-negative real numbers, introduce hermitian matrices $x$ and $y$ by
\begin{equation}\label{x_and_y}
 x|j\rangle_1=\frac{k}{r}j|j\rangle_1, \quad
 y|l\rangle_2=\frac{k}{r}l|l\rangle_2
\end{equation}
It can be shown that for $q=e^{ik/r^2}$ we have
\begin{equation}\label{u_v_q}
 u=e^{ix/r},\quad
 v=e^{iy/r}
\end{equation}
where the parameters $r$ and $k$ must be related by 
\begin{equation*}
n=\frac{2\pi r^2}{k}=\frac{\left(2\pi r\right)^2}{2\pi k}. 
\end{equation*}
%ek kan miskien n bewys van hierdie insluit, vind maar net uit of ek dit ook in die dokument moet sit.
\begin{proposition}\label{uv_generate}
The matrices $u$ and $v$, generate the entire $M_n(\mathbb{C})$.
\end{proposition}
\begin{proof}
Since $A:=M_n(\mathbb{C})$ is a von Neumann algebra we must have $A=A''$, where $A''$ is the bicommutant of $A$ \cite[p,115]{murphy}.  Let $C^*(u,v)$ be the $C^*$-algebra generated by the unitary matrices $u$ and $v$.  We will show that the commutant of $C^*(u,v)$ consists only of scalar multiples of the $n\times n$ identity matrix.  This in turn implies that $C^*(u,v)''$ must be all of $A$.

Consider the $C^*$-subalgebra $C^*(u)\subset C^*(u,v)$.  Since $q$ is the an $n$-th root of unity, which we can also assume to be
\begin{equation*}
 q=e^{\frac{2\pi i}{n}},
\end{equation*}
$C^*(u)$ contains all the diagonal matrices.  The easiest way to see this is to note that the diagonal entries $1,q,\cdots,q^{n-1}$ of $u$ are all distinct.  Hencegiven a polynomial $p$ for which $p(q^j)=\delta_{ij}$ for some fixed $0\leq i\leq n-1$, we then have that $E_{ii}=p(u)\in C^*(u)$.  When a matrix commutes with all diagonal matrices it also has to be a diagonal matrix.  Furthermore, diagonal matrices that commute with $v$ have to be scalar multiples of the $n\times n$ identity matrix.  This implies that the commutant of $C^*(u,v)$ is the scalars.  This completes the proof.
\end{proof}
\begin{lemma}
 $M_n(\mathbb{C})$ is a simple $C^*$-algebra, in other words the only closed ideals it contains is $0$ and the whole of $M_n(\mathbb{C})$.
\end{lemma}

We can determine the commutation relations
\begin{eqnarray*}
 \left[x,v\right]=\frac{k}{r}v\left(1-nP\right),\quad
 \left[y,u\right]=-\frac{k}{r}u\left(1-nQ\right).
\end{eqnarray*}
Where we define the projections
\begin{eqnarray*}
 P=\left|n-1\rangle_1 \langle n-1\right|,\quad
 Q=\left|0\rangle_2 \langle 0\right|.
\end{eqnarray*}
Define the derivations 
\begin{equation}\label{deri}
 \delta_1=-\frac{1}{ik}\text{ad}y, \quad
 \delta_2=-\frac{1}{ik}\text{ad}x
\end{equation}
where $\text{ad}y(u):=[y,u]$.  The action of the derivations on the matrices $u$ and $v$ are
\begin{eqnarray*}
 \delta_1(u)=-ir^{-1}u(1-nQ),\quad
 \delta_1(v)=0\\
 \delta_2(v)=-ir^{-1}v(1-nP),\quad
 \delta_2(u)=0.
\end{eqnarray*}
Due to equation (\ref{com_rel}) we can think of $M_n(\mathbb{C})$ as a finite dimensional representation of the quantum torus.

\chapter{$\sigma$-Models}\label{chapter_03}
%Binne hierdie hoofstuk moet ek die polyakov aksie herformuleer vir die klassieke torus en dit mooi in alle besonderhede deurwerk
%Kan ook die Euler-Lagrange vergelykings en alles aangaandie die aksies ook in hierdie hoofstuk sit
\begin{center}
 \begin{quote}
  \textsl{``The mathematical problems that have been solved or techniques that have arisen out of physics in the past have been the lifeblood of mathematics.''\newline - Sir Michael Atiyah}
 \end{quote}
\end{center}
This chapter follows very closely form the work of Mathai Vargese and Johnathan
Rosenberg \cite{MR}.  Here we generalize the Polyakov action to the case where
parameter space and world-time are both replaced by different noncommutative
tori and show that in the classical limit the noncommutative action reduces to
its classical counterpart.  This chapter plays a key role, as it shows how the
mathematics and physics come together to form a noncommutative sigma model.  To
explore these results we consider an example not explored by the previous two
authors.  Here we create an example of such a noncommutative action in the case
of the space of $2\times2$ matrices with complex entries, denoted by
$M_2(\mathbb{C})$.  In this special case we construct a specific partition
function and calculate noncommutative path integrals. 

\section{Noncommutative Action}\noindent
\noindent
Recall that in classical string theory we can describe the string dynamics using the Polyakov action \cite[p. 583]{zwiebach} given by
\begin{equation}\label{polyakov}
 S=\frac{-1}{4\pi\alpha^{\prime}}\int \sqrt{-h}h^{\alpha\beta}\partial_{\alpha}X^{\mu}\partial_{\beta}X^{\nu}\eta_{\mu\nu}d\sigma_1 d\sigma_2,
\end{equation}
where $\sigma_1$ and $\sigma_2$ play the roles of $\sigma$ and $\tau$
respectively.  We use Einsteinian notation, where repeated indices imply
summation.  For example using the Euclidean metric we can write
\begin{equation*}
 \sum^{n}_{\mu=1}X^{\mu}X^{\mu}=X_{\mu}X^{\mu}.
\end{equation*}
Let us simplify equation (\ref{polyakov}) by setting all the constants equal to one and looking specifically at the Euclidean metric together with the mapping
\begin{equation*}
 g:\Sigma\rightarrow \mathbb{R}^n, ~ g(\sigma_1,\sigma_2)=\left(X^1(\sigma_1,\sigma_2),\cdots,X^n(\sigma_1,\sigma_2)\right),
\end{equation*}
where the $X^j(\sigma,\tau)$ are called the mapping functions and map $\mathbb{R}$ into itself.  Here $\Sigma$ plays the role of the two dimensional parameter space and $\mathbb{R}^n$ that of world-time.  Using the framework just described we write the Polyakov action as
\begin{equation*}
 S=\int\partial_{\alpha}X^{\mu}\partial_{\alpha}X^{\mu}d\sigma_1 d\sigma_2.
\end{equation*}
In order to generalize the action to the noncommutative realm we have to work in
an algebraic picture.  For this purpose, let $f\in C(\mathbb{R}^n)$.  Then we
define
\begin{equation}\label{dephine}
 \varphi:C(\mathbb{R}^n)\rightarrow C(\Sigma):f\mapsto f\circ g.
\end{equation}
Consider the following calculation
\begin{eqnarray}
 &&\left[\partial_{\alpha}e^{iX^1\circ g}\right]^*\partial_{\alpha}e^{iX^1\circ g}+
\cdots+\left[\partial_{\alpha}x^{iX^n\circ g}\right]^*\partial_{\alpha}e^{iX^n\circ g}\nn\\
 &=&i\sum^{n}_{j=1} \left[ i\left( \partial_{\alpha}X^j\circ g\right)e^{iX^j\circ g}\right]^*\left(\partial_{\alpha} X^j\circ g \right) e^{iX^j\circ g} \nn\\
 &=&\sum^{n}_{j=1} \left[\partial_{\alpha}X^j\circ g\right]^*\left(\partial_{\alpha} X^j\circ g\right) e^{i\left((X^j)^*-X^j\right)\circ g}\nn\\
 &=&\left(\partial_{\alpha}X^{\mu}\right)^*\partial_{\alpha} X^{\mu} \label{calc_1}
\end{eqnarray}
where in the final step we employ the shorthand notation $X^{\mu}:=X^{\mu}\circ g$.  Since we can write
\begin{equation*}
 e^{iX^{\mu}\circ g}=e^{iX^{\mu}}\circ g
\end{equation*}
we see that in the context of *-algebras we can regard $e^{iX^{\mu}}$ as a unitary element.  This leads us to define the unitaries
\begin{equation*}
 U^{\mu}:=e^{iX^{\mu}}.
\end{equation*}
The calculation that led to equation (\ref{calc_1}) then enables us to write the Polyakov action as
\begin{equation}\label{polya}
 S=\int\left[\partial_{\alpha}\varphi(U^{\mu})\right]^*\partial_{\alpha}\varphi(U^{\mu})d\sigma_1 d\sigma_2.
\end{equation}
where $\varphi$ is as defined in equation (\ref{dephine}).

We would like to generalize the Polyakov action to the case where our parameter space and world-time both become noncommutative $C^*$-algebras and the natural way to proceed is to make use of equation (\ref{polya}).  In order to justify the noncommutative version of the Polyakov action let us first consider how this action would behave on the classical torus.  

We have seen in Proposition \ref{kontinu} that when $\theta=0$ we have
$A_{0}\cong C(\mathbb{T}^2)$.  This will be the case we are interested in to
show that the noncommutative version of the Polyakov action reduces to it's
classical form in the $\theta=0$ limit.  Let us now suppose that parameter space
as well as world-time can both be described by $\mathbb{T}^2$.  The mapping
$g:\Sigma\rightarrow X$ then becomes
\begin{equation*}
 g:\mathbb{T}^2\rightarrow \mathbb{T}^2
\end{equation*}
and as in equation (\ref{dephine}) we have
\begin{equation*}
 \varphi:C\left(\mathbb{T}^2\right)\rightarrow C\left(\mathbb{T}^2\right): f\mapsto f\circ g.
\end{equation*}
Recall that in the classical limit we can write the generating unitaries of the
quantum torus, $U$ and $V$; as
\begin{eqnarray*}
 U&:=&U(x,y)=e^{ix}\\
 V&:=&V(x,y)=e^{iy}.
\end{eqnarray*}
In deriving equation (\ref{polya}) we obtained $n$ unitaries $e^{iX^{\mu}}$, however now we are interested in $\mathbb{T}^{2}=\mathbb{R}^2/2\pi\mathbb{Z}^2$ and so only two unitaries are needed and they are $U$ and $V$, the generating unitaries of the quantum torus in the classical limit.  We saw in subsection \ref{clas} that in the classical limit where $\theta=0$, the derivations we constructed in section \ref{der} reduce to normal partial derivatives.  In our present case $\partial_{\sigma_1}=\delta_1$ and $\partial_{\sigma_2}=\delta_2$.  Now we can write equation (\ref{polya}) as
\begin{eqnarray*}
 S&=&\int\left[
\delta_1\left(\varphi(U)\right)^*\delta_1(\varphi(U))+\delta_2\left(\varphi(U)
\right)^*\ delta_2(\varphi(U))+\right.\\
 && \delta_1\left(\varphi(V)\right)^*\delta_1(\varphi(V))+\delta_2\left(\varphi(V)\right)^*\delta_2(\varphi(V))\left. \right]d\sigma_1 d\sigma_2
\end{eqnarray*}
From Theorem \ref{quantum_trace} we know that we can write the above expression for the Polyakov action as in the following claim:
\begin{claim}
The natural generalization of the Polyakov action to noncommutative $C^*$-algebra, $A_{\theta}$ is  
\begin{eqnarray}\label{noncom_action}
  S(\varphi)&=&\tau\left[\delta_1(\varphi(U))^*\delta_1(\varphi(U))+\delta_2(\varphi(U))^*\delta_2(\varphi(U))\right.\nn\\
&&\left.+\delta_1(\varphi(V))^*\delta_1(\varphi(V))+\delta_2(\varphi(V))^*\delta_2(\varphi(V))
 \right].\label{new_polyakov}
\end{eqnarray}
This can also be written as
\begin{equation}\label{for_calc}
 S\left(\varphi\right)=\sum^{2}_{k=1}\sum^{2}_{l=1}\tau\left[\delta_k\left(\varphi(U_l)\right)^*\delta_k\left(\varphi(U_l)\right)\right]
\end{equation}
where we denote
\begin{equation*}
 U_1:=U, \quad U_2:=V.
\end{equation*}
\end{claim}
\noindent
When we consider the manipulations done on the Polyakov action in the classical limit we can clearly see that equation (\ref{noncom_action}) reduces to the normal Polyakov action.  In what follows we will use equation (\ref{noncom_action}) when considering the action on mappings between different noncommutative tori.  As we have mentioned earlier, the existence of *-homomorphism $A_{\Theta}\rightarrow A_{\theta}$ is not obvious and so far we have merely assumed the existence of such maps.  In chapter \ref{chapter_5} we will study the existence of such mappings.

\begin{proposition}(Simple case)\newline
 Let $\Theta=\theta$ and $\varphi$ be the identity map.  Then we find that
$S(\varphi)=8\pi^2$.
\end{proposition}
\begin{proof}
According to equation (\ref{derivations_def}) we have
\begin{eqnarray*}
 S(\varphi)&=&\tau\left[\delta_1(u)^*\delta_1(u)+\delta_2(u)^*\delta_2(u)
+\delta_1(v)^*\delta_1(v)+\delta_2(v)^*\delta_2(v)\right]\\
&=&\tau\left[\delta_1(u)^*\delta_1(u)+0+0+\delta_2(v)^*\delta_2(v)\right]\\
&=&\tau\left[4\pi^2+4\pi^2\right]\\
&=&8\pi^2.
\end{eqnarray*}
This completes the proof.
\end{proof}

\begin{proposition}(More general, but still $\Theta=\theta$)\newline
 Let $U$ and $V$ be the generating unitaries of the quantum torus.  Consider the mapping defined by
\begin{equation*}
 \varphi_{A}:U\mapsto U^pV^q,~ V\mapsto U^rV^s
\end{equation*}
with 
$A=\left(
\begin{array}{cc}
  p&q\\
  r&s
 \end{array}
\right)\in$SL$(2,\mathbb{Z})$.  $\varphi$ is a well defined *-isomorphism and $S(\varphi)=4\pi^2(p^2+q^2+r^2+s^2)$.
\end{proposition}
\begin{proof}
From the definition of $\varphi_A$ we have
\begin{eqnarray*}
 &&\varphi_A(U)\varphi_A(V)=e^{2\pi i\theta}\varphi_A(V)\varphi_A(U)\\
&\Longleftrightarrow&U^pV^qU^rV^s=e^{2\pi i\theta} U^rV^sU^pV^q\\
&\Longleftrightarrow&U^p\left(e^{-2\pi i\theta qr} U^rV^q\right)V^s=e^{2\pi
i\theta}U^r\left(e^{-2\pi i\theta ps} U^pV^s\right)V^q\\
&\Longleftrightarrow&U^{p+r}V^{q+s}=e^{2\pi i\theta}e^{2\pi
i\theta(qr-ps)}U^{p+r}V^{q+s}.
\end{eqnarray*}
But since $A\in SL(2,\mathbb{C})$ we have
\begin{equation*}
 ps-qr=1
\end{equation*}
which implies that $\varphi_A$ is well defined. That $\varphi_A$ is a
*-isomorphism follows from the universal property of the quantum torus (Theorem
\ref{universal}).

Let us now determine the noncommutative action.  We see that
\begin{eqnarray*}
 S(\varphi_{A})&=&\tau\left[\delta_1(u^pv^q)^*\delta_1(u^pv^q)+\delta_2(u^pv^q)^*\delta_2(u^rv^s)\right.\\
&&\left.+\delta_1(u^rv^s)^*\delta_1(u^rv^s)+\delta_2(u^rv^s)^*\delta_2(u^rv^s)\right]\\
&=&\tau\left[(2\pi i p u^pv^q)^*(2\pi i pu^pv^q)+(2\pi iq u^pv^q)^*(2\pi i qu^pv^q)\right.\\
&&\left.+(2\pi i ru^rv^s)^*(2\pi i r u^rv^s)+(2\pi i su^rv^s)^*(2\pi i su^rv^s)\right]\\
&=&\tau\left[4\pi^2(p^2+q^2+r^2+s^2)\right]\\
&=&4\pi^2(p^2+q^2+r^2+s^2).
\end{eqnarray*}
\end{proof}

%\begin{remark}
% It has been shown by Hanfeng Li \cite{Li2009} that this action is a minimum.
%\end{remark}

\section{The Two Dimensional Representation}\noindent
In this section we look at an example not included in \cite{MR}.  We would like
to construct an example which shows what can be done in the simplest case which
exhibits clear noncommutative behaviour.  The natural choice would be to
consider the space of two by two matrices with complex entries.  We considered
the general $n\times n$ case in section \ref{finite_representation} and in this
section we will continue with the notation introduced for the general case.  Let
$q=e^{i\frac{k}{r^2}}=e^{i\pi}=-1$.  Then for the two by two case we can write
these matrices explicitly as
\begin{eqnarray*}
 u=\left(
\begin{array}{cc}
 1 & 0 \\
 0 & -1
\end{array}
\right),\quad
v=\left(
\begin{array}{cc}
 0 & 1 \\
 1 & 0
\end{array}
\right),
\end{eqnarray*}
where they satisfy the form of equation (\ref{u_v_q}).  These matrices satisfy the commutation relation
\begin{equation*}
 uv=-uv.
\end{equation*}
From Proposition \ref{uv_generate} we know that $M_2(\mathbb{C})$ is generated by these two unitary matrices $u$ and $v$ and for the remainder of the section we will concern ourselves with the case $M_2(\mathbb{C})$ which we regard as the finite dimensional representation of the quantum torus.  Similarly we can write down expressions for the matrices $x$ and $y$ as in equation (\ref{x_and_y})
\begin{eqnarray*}
 x=\frac{k}{r} \left(
\begin{array}{cc}
 0 & 0 \\
 0 & 1
\end{array}
\right),\quad
y=\frac{k}{2 r} \left(
\begin{array}{cc}
 1 & -1 \\
 -1 & 1
\end{array}
\right).
\end{eqnarray*}
Likewise, the derivations (\ref{deri}) can be written as
\begin{eqnarray*}
 \delta_1(\cdot)&=&-\frac{1}{ik}\left[y,\cdot\right]=-\frac{1}{2ir}\left[\left(
\begin{array}{cc}
 1 & -1 \\
 -1 & 1
\end{array}
\right),\cdot \right]\\
\delta_2(\cdot)&=&\frac{1}{ik}\left[x,\cdot\right]=\frac{1}{ir}\left[\left(
\begin{array}{cc}
 0 & 0 \\
 0 & 1
\end{array}
\right),\cdot \right].
\end{eqnarray*}
Using these formulas for the matrices together with the method used by Mathai and Rosenberg \cite{MR} and equation (\ref{noncom_action}) we are able to determine a specific action for the finite dimensional representation.

\begin{proposition}\label{why_auto}
 If $\varphi:M_n(\mathbb{C})\rightarrow M_n(\mathbb{C})$ is a unital *-homomorphism then $\varphi$ necessarily has to be a *-automorphism.
\end{proposition}
\begin{proof}
Since we are dealing with finite dimensional representations we have
\begin{equation*}
 M_n(\mathbb{C})\cong B(H)=K(H)
\end{equation*}
for the Hilbert space $H$, consisting of $n$-dimensional column vectors.  From \cite[example 3.2.2]{murphy} we know that $M_n(\mathbb{C})$ is a simple $C^*$-algebra, in other words the only closed ideals contained in $M_n(\mathbb{C})$ is $0$ and the whole of $M_n(\mathbb{C})$.  The kernel of $\varphi$,
\begin{equation*}
 \ker(\varphi)=\left\{a\in M_n(\mathbb{C}): \varphi(a)=0 \right\}
\end{equation*}
is a closed, two sided ideal of $M_n(\mathbb{C})$.  But since $M_n(\mathbb{C})$
is simple $\ker(\varphi)$ must be either $0$ or $M_n(\mathbb{C})$.  If
$\ker(\varphi)=M_n(\mathbb{C})$ then $\varphi$ has to be identically zero.  But
this is not possible since $\varphi$ is unital, and hence $\ker(\varphi)=0$. 
Suppose $\varphi(a)=\varphi(b)$ for some $a,b\in M_n(\mathbb{C})$.  Then
$\varphi(a-b)=0$.  But this can only occur for $a=b$.  Hence $\varphi$ is a
bijective *-homomorphism and therefore a *-isomorphism, which in this case
implies it is in fact a *-automorphism.
\end{proof}

We use equation (\ref{for_calc}) to perform computations and observe for
conveniece that
\begin{equation*}
 S\left(\varphi\right)=\sum^{2}_{k=1}\sum^{2}_{l=1}\text{Tr}\left[\delta_k\left(\varphi(u_l)\right)^*\delta_k\left(\varphi(u_l)\right)\right]
\end{equation*}
where we denote
\begin{equation*}
 u_1:=U, ~ u_2:=V.
\end{equation*}
From Proposition \ref{why_auto} we only need to work with the case where the
mapping $\varphi$ is a *-automorphism of $M_2\left(\mathbb{C}\right)$.  Since
$M_2\left(\mathbb{C}\right)$ is finite dimensional it is isomorphic to $B(H)$
for some finite dimensional Hilbert space $H$.  Furthermore, the finite
dimensionality implies that $B(H)=K(H)$.  It is known that every automorphisms
of $K(H)$ has the form
\begin{equation*}
 \text{Ad}W:A\mapsto WAW^*
\end{equation*}
where $W$ is some unitary operator in $K(H)$ \cite{Raeburn}.  Therefore these types of mappings are all that we have to consider when calculating the actions of the *-automorphisms.  The action can then be rewritten in the form
\begin{equation*}
 S\left(W\right)=\sum^{2}_{k=1}\sum^{2}_{l=1}\text{Tr}\left[\delta_k\left(Wu_lW^*\right)^*\delta_k\left(Wu_lW^*\right)\right].
\end{equation*}
If we are now able to parametrize all the unitary two by two matrices with complex entries we will be able to determine the action for any *-automorphism of $M_2\left(\mathbb{C}\right)$.

\section{Parametrization of $SU(2)$}\noindent
There has been some literature regarding the parametrization of $SU(N)$ \cite{Tilma2002}, but for now we are only interested in the $SU(2)$ case.  This case has been studied in \cite{Inomata}, however for completeness we give a brief recollection of the main constructions.  By definition $SU(2)$ consists of unitary two by two matrices $g$ with complex entries which adhere to the constraint
\begin{equation*}
 \det g=1.
\end{equation*}
We can write $g\in SU(2)$ as
\begin{equation*}
 g=\left(
\begin{array}{cc}
 a&b\\
 c&d
\end{array}
\right).
\end{equation*}
Since the row (and column) vectors of unitary matrices are orthonormal we can
always choose the row vectors such that $d=\overline{a}$ and $c=-\overline{b}$. 
Then
\begin{eqnarray*}
 g^*g&=&\left(
\begin{array}{cc}
 \overline{a}&-b\\
 \overline{b}&a
\end{array}
\right)
\left(
\begin{array}{cc}
 a&b\\
 -\overline{b}&\overline{a}
\end{array}
\right)\\
&=&\left(
\begin{array}{cc}
 \overline{a}a+\overline{b}b&0\\
 0&\overline{a}a+\overline{b}b
\end{array}
\right)
\end{eqnarray*}
and the $\det g=1$ condition becomes
\begin{equation*}
 |a|^2+|b|^2=1.
\end{equation*}
Now we write 
\begin{eqnarray*}
a=u_1+iu_2,\quad
b=-u_3+iu_4
\end{eqnarray*}
with $u_1,\dots,u_4\in\mathbb{R}$ such that
\begin{equation*}
u^2_1+u^2_2+u^2_3+u^2_4=1.
\end{equation*}
This is just the unit sphere $S^3$ centered at the origin in four dimensional real Euclidean space.  The matrix $g$ can now be rewritten as
\begin{equation*}
 g=\left(
\begin{array}{cc}
 u_1+iu_2&-u_3+iu_4\\
 u_3+iu_4&u_1-iu_2
\end{array}
\right).
\end{equation*}
Using Euler angles we determine the parametrization of the unit sphere, $S^3$ as in \cite{Inomata}
\begin{eqnarray*}
u_1=\cos\frac{\theta}{2} \cos\frac{\phi+\psi}{2},\quad
u_2=\cos\frac{\theta}{2} \sin\frac{\phi+\psi}{2}\\
u_3=\sin\frac{\theta}{2} \sin\frac{\phi-\psi}{2},\quad
u_4=\sin\frac{\theta}{2} \cos\frac{\phi-\psi}{2}
\end{eqnarray*}
where the ranges of the angles are
\begin{eqnarray*}
 0\leq \theta \leq \pi,\quad
 0\leq \phi \leq 2\pi,\quad
0\leq \psi \leq 4\pi.
\end{eqnarray*}
This implies that we can write any unitary matrix $g\in SU(2)$ as a function of the Euler angles as follows:
\begin{equation*}
 g(\phi,\theta,\psi)=\left(
\begin{array}{cc}                          
\cos\frac{\theta}{2}e^{i\frac{\phi+\psi}{2}}&i\sin\frac{\theta}{2}e^{i\frac{\phi-\psi}{2}}\\
i\sin\frac{\theta}{2}e^{-i\frac{\phi-\psi}{2}}&\cos\frac{\theta}{2}e^{-i\frac{\phi+\psi}{2}}
\end{array}
\right).
\end{equation*}
We can also find the invariant volume element or measure as in \cite{Inomata}
\begin{equation*}
 d\mu(g)=\frac{1}{16\pi^2}\sin\theta d\theta d\phi d\psi
\end{equation*}
which is normalized such that
\begin{equation*}
 \int_{SU(2)}d\mu(g)=1.
\end{equation*}
This parametrization gives us all the elements of $SU(2)$ exactly once.  We can now use this to express the action of any *-automorphism of $M_2(\mathbb{C})$ as a function of $\phi,\theta$ and $\psi$.

\section{Path Integrals}\noindent
In string theory the path integral is an integral over the space of all world-sheets of the string.  In the noncommutative case, specifically $M_n(\mathbb{C})$, the world-sheet is the space of all unital *-homomorphisms 
\begin{equation*}
 \varphi:M_n(\mathbb{C})\rightarrow M_n(\mathbb{C}),
\end{equation*}
described by
\begin{equation*}
 \varphi(\cdot)=U(\cdot) U^*
\end{equation*}
where $U$ is in $SU(n)$.  We only need to consider $SU(n)$ since we can take any matrix in $U(n)$ and multiply it by a certain factor in such a way that its determinant becomes 1.  In \cite{MR} the authors try to calculate noncommutative path integrals but have to make numerous simplifying assumption.  In the end they find a semi classical approximation by summing over the critical points.  In this section we calculate explicit noncommutative path integrals in the case of $M_2(\mathbb{C})$.  Since we are now able to determine the action for any *-automorphism of $M_2(\mathbb{C})$ we can proceed to study the noncommutative path integral.  After the parametrization we can write the action of any *-automorphism $\varphi$ of $M_2\left(\mathbb{C}\right)$ as
\begin{eqnarray*}
 S_{P}(g)&=&\sum^2_{k=1}\sum^2_{l=1}\text{Tr}\left[\delta_k\left(gu_lg^*\right)^*\delta_k\left(gu_lg^*\right)\right]\\
&=&\frac{1}{16 r^2}e^{-2 i (\phi +\psi )} \left[-4 \left(-1+e^{4 i \phi }\right) \left(-1+e^{4 i \psi }\right) \cos(\theta )-\right.\\
&&\left.\left(1+e^{2 i \phi }\right)^2 \left(1+e^{2 i \psi }\right)^2 \cos(2 \theta )+\right.\\
&&\left.4 e^{2 i (\phi +\psi )} (21+\cos(2 \phi ) (1-3 \cos(2 \psi ))+\cos(2 \psi ))\right].
\end{eqnarray*}
Mathematica was used to perform the calculations which yielded the result of the second line.  We are able to determine the relationship between the minima (Figure \ref{minima}) and maxima (Figure \ref{maxima}) of the Polyakov action with respect to $r$.  The minima and maxima can respectively be described by 
\begin{equation*}
S^{\text{min}}_{P}=4r^{-2}, \quad S^{\text{max}}_{P}=6r^{-2}. 
\end{equation*}
We can also consider the path integral associated with this particular Polyakov action.  We can write the partition function as the following path integral
\begin{eqnarray*}
 Z\left(g\right)&=&\int_{SU(2)}e^{-S(\varphi)}d\mu(g)\\
&=&\frac{1}{16\pi^2}\int^{4\pi}_{0}\int^{2\pi}_{0}\int^{\pi}_{0}e^{-S(\varphi)}\sin\theta d\theta d\phi d\psi.
\end{eqnarray*}
A graphical representation of the partition function, $Z(g)$ as a function of $r$ is given below (Figure \ref{partition_function}).
\setcounter{figure}{2}
\begin{figure}[h!]
 \centering
 \includegraphics[scale=0.85]{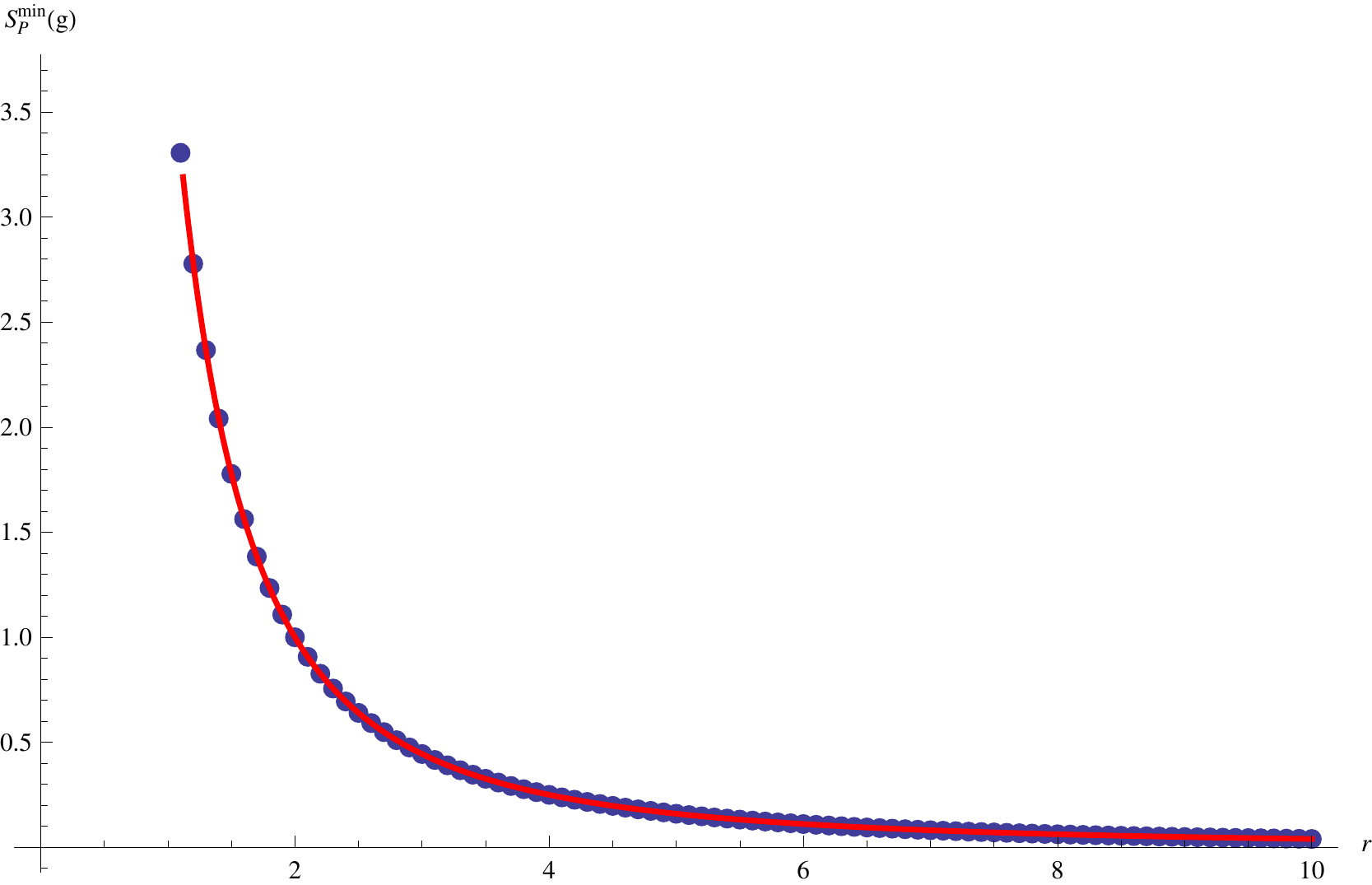}
 % munimum.eps: 0x0 pixel, 300dpi, 0.00x0.00 cm, bb=
 \caption{Minimum of the partition function as a function of $r$.  We can write an exact formula $S_{\text{min}}=4r^{-2}$.}
\label{minima}
\end{figure}
\begin{figure}[h!]
 \centering
 \includegraphics[scale=0.85]{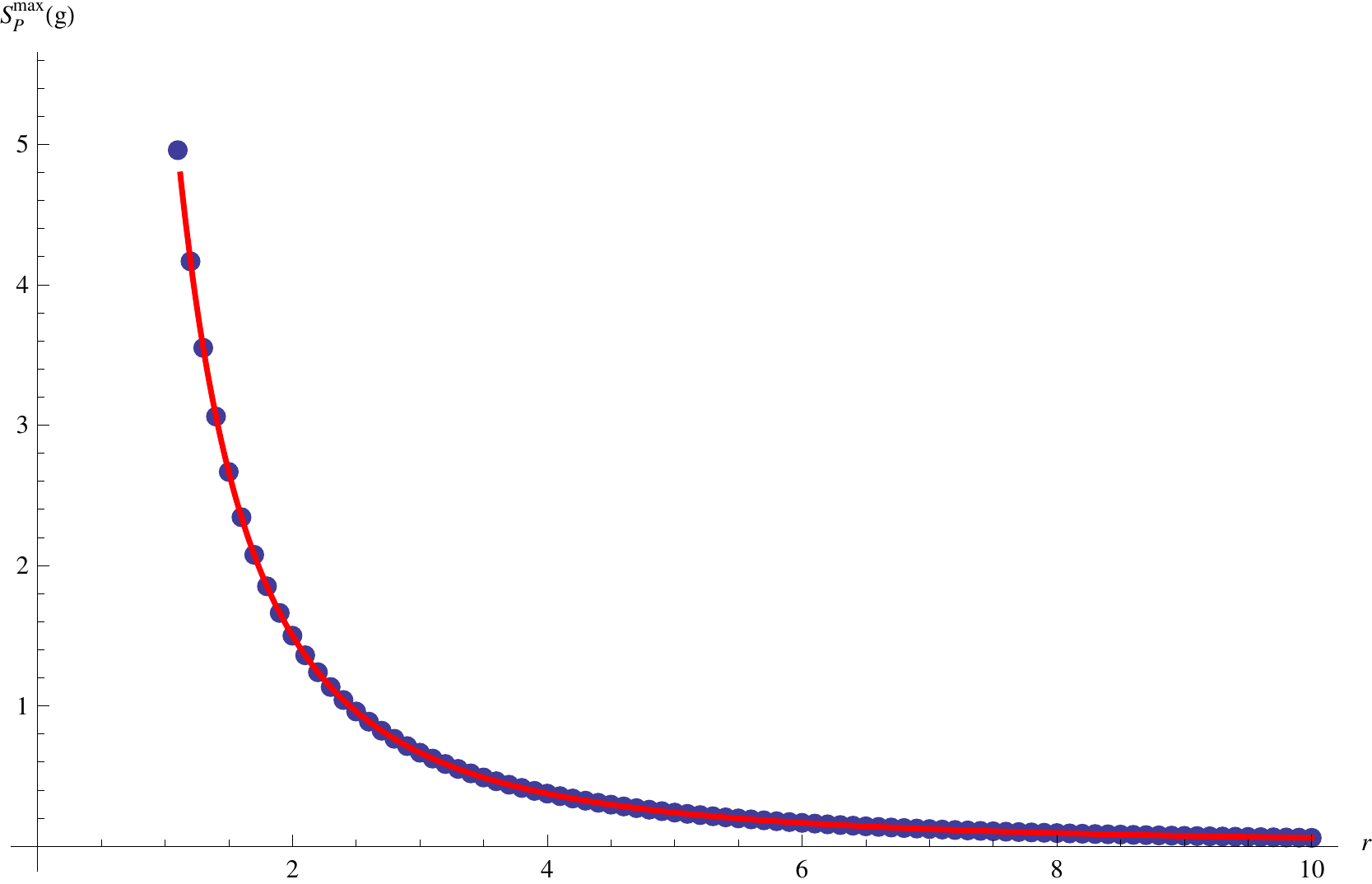}
 % maximum.eps: 0x0 pixel, 300dpi, 0.00x0.00 cm, bb=
 \caption{Maximum of the partition function as a function of $r$.  We can write an exact formula $S_{\text{max}}=6r^{-2}$.}
\label{maxima}
\end{figure}

\begin{figure}[h!]
 \centering
 \includegraphics[scale=0.8]{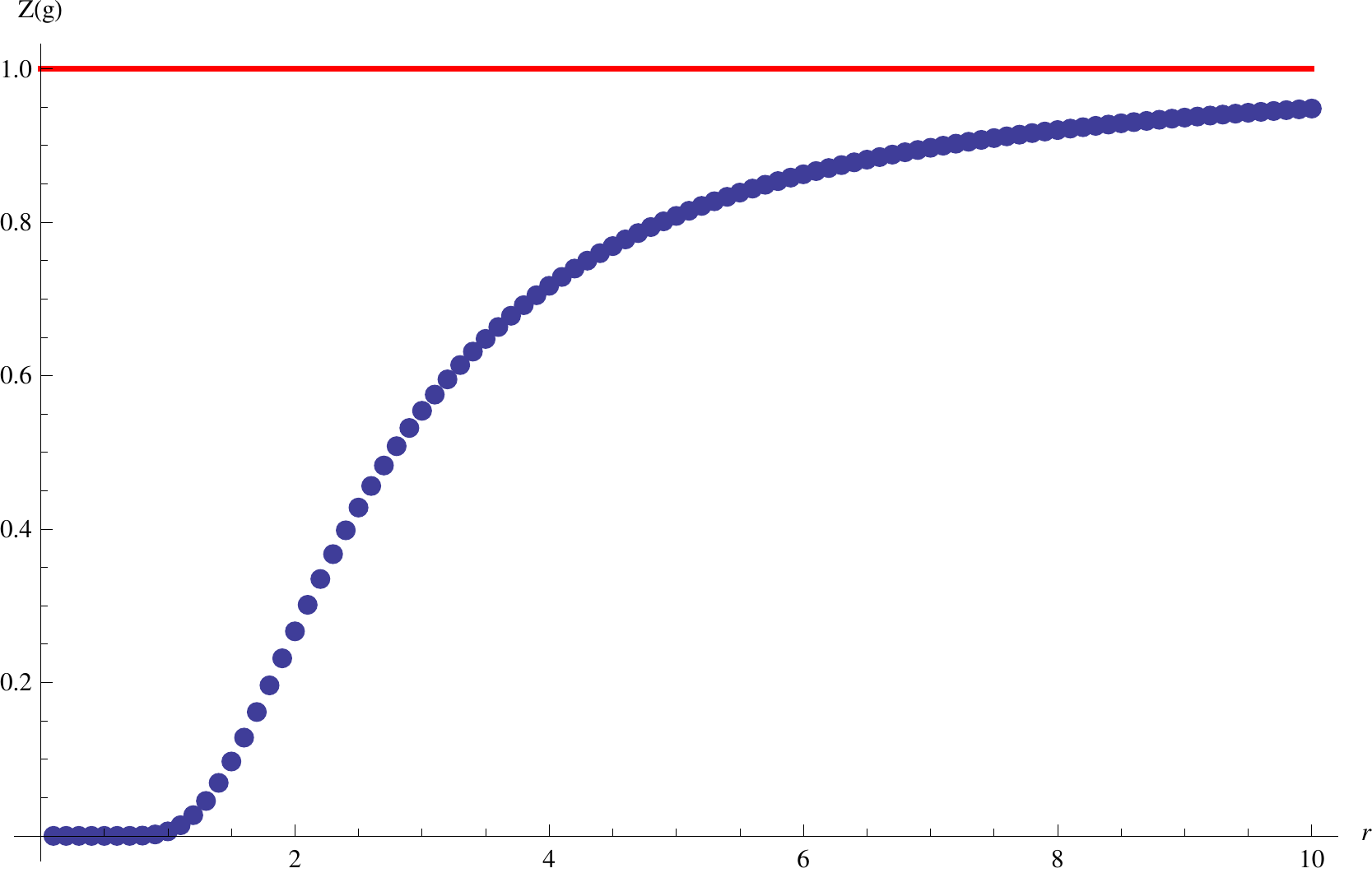}
 % partition_function.eps: 0x0 pixel, 300dpi, 0.00x0.00 cm, bb=
 \caption{Partition function as a function of $r$}
\label{partition_function}
\end{figure}

Classically we would be very interested in the critical points of the action since they should satisfy the classical Euler-Lagrange equations and describe the physical trajectories of particles.  A strange phenomena that we see in the two by two matrix scenario is that these maxima and minima occur infinitely many times as can be seen by the contour plots in figures \ref{contour01} and \ref{contour02}.

\begin{figure}[h!]
 \centering
 \includegraphics[scale=0.6]{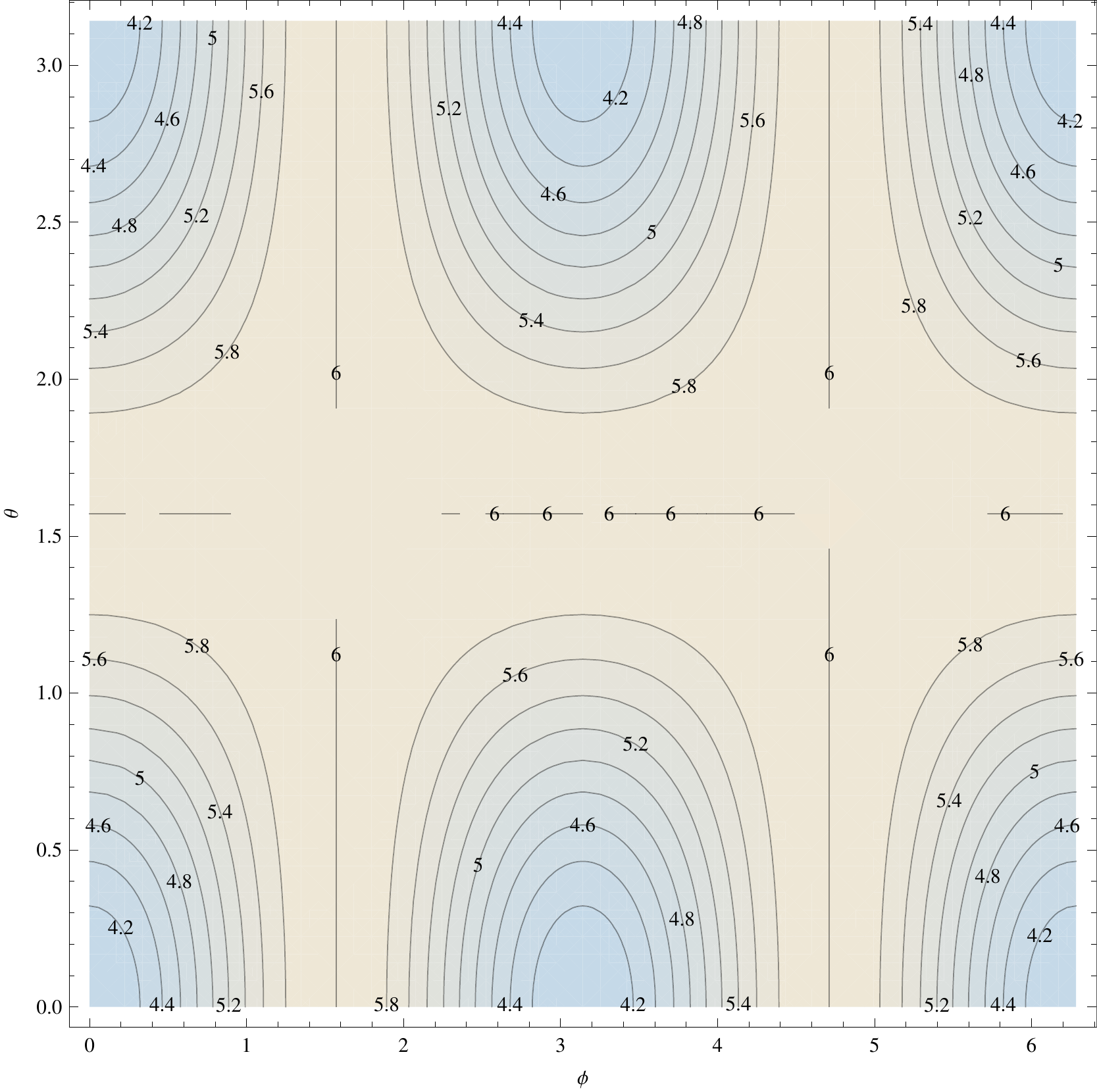}
 % contour.eps: 0x0 pixel, 300dpi, 0.00x0.00 cm, bb=
\caption{Contour plot of the action for $\psi=0$ and $r=1$.}
\label{contour01}
\end{figure}

\begin{figure}[h!]
 \centering
 \includegraphics[scale=0.6]{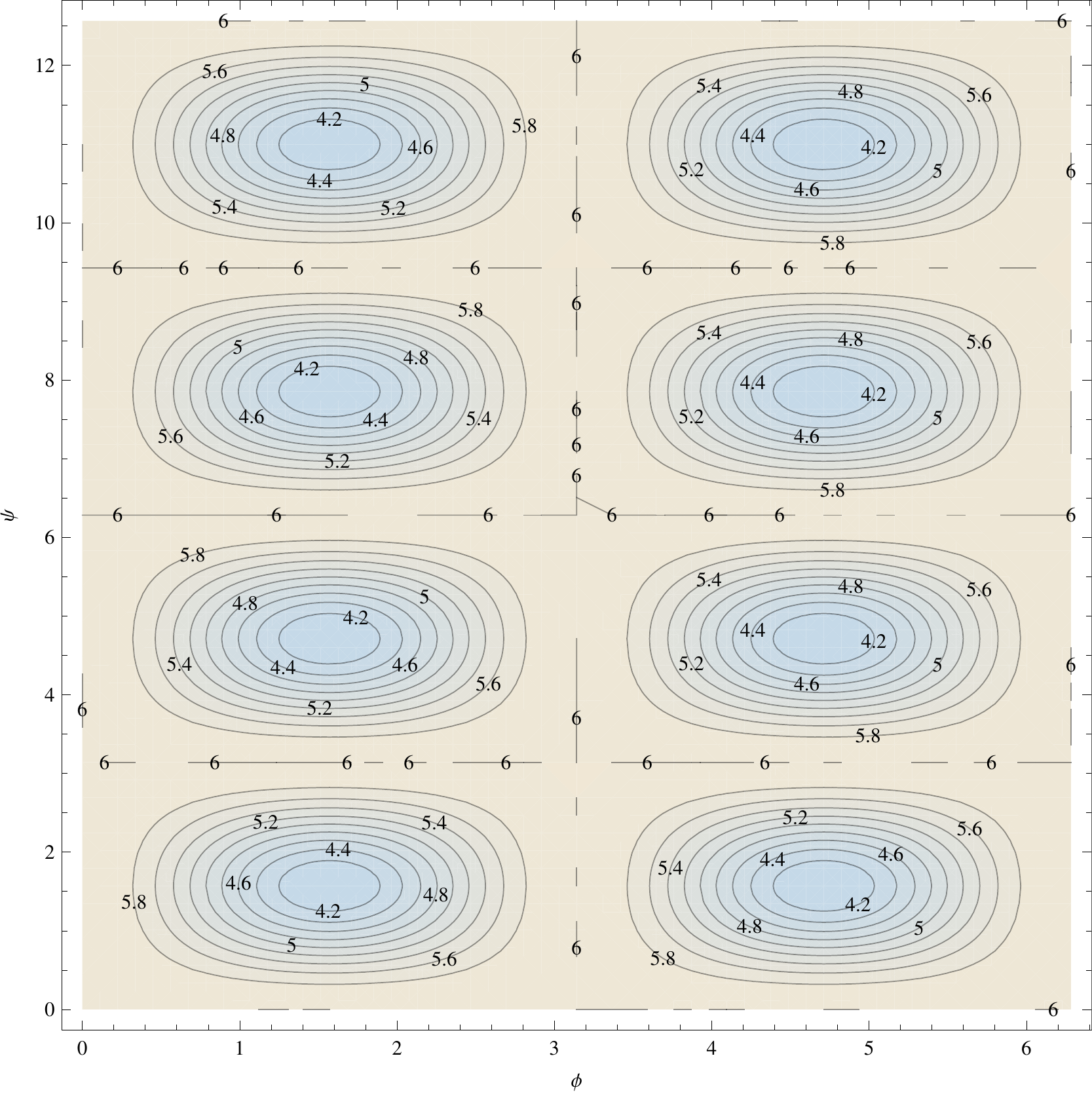}
 % contour2.pdf: 0x0 pixel, 0dpi, 0.00x0.00 cm, bb=
\caption{Contour plot of the action for $\theta=\frac{\pi}{2}$ and $r=1$.}
\label{contour02}
\end{figure}

In normal statistical mechanics we can calculate the expectation values of the energy and the variance of the energy by
\begin{equation}
 \langle E \rangle = -\frac{\partial \ln Z}{\partial \beta}, \quad \langle \Delta(E)^2\rangle = \langle \left(E-\langle E \rangle \right)^2 \rangle = \frac{1}{Z^2}\frac{\partial^2 Z}{\partial \beta^2}
\end{equation}
respectively, where $\beta=\frac{1}{k_BT}$ is the inverse temperature.  The only
free parameter we have at our disposal is $r$, which was introduced in equation
(\ref{x_and_y}).  From a dimensional analysis point of view $r$ will have the
dimensions of length.  Furthermore, in special relativity time also has the
dimensions of length, so heuristically at least, the role of time can then also
be played by $r$ in the above context.  It is a familiar result from
statistical mechanics that we can write $\beta=it$, the ``length'' of the system
in imaginary time.  Then, at least heuristically we can treat $r$ as the inverse
temperature $\beta$.  This enables us to determine the ``expectation value of
the energy of our string'' of which we give a graphical representation in figure
\ref{expectH}.

\begin{figure}[h!]
 \includegraphics[bb=0 0 480 307,scale=0.8]{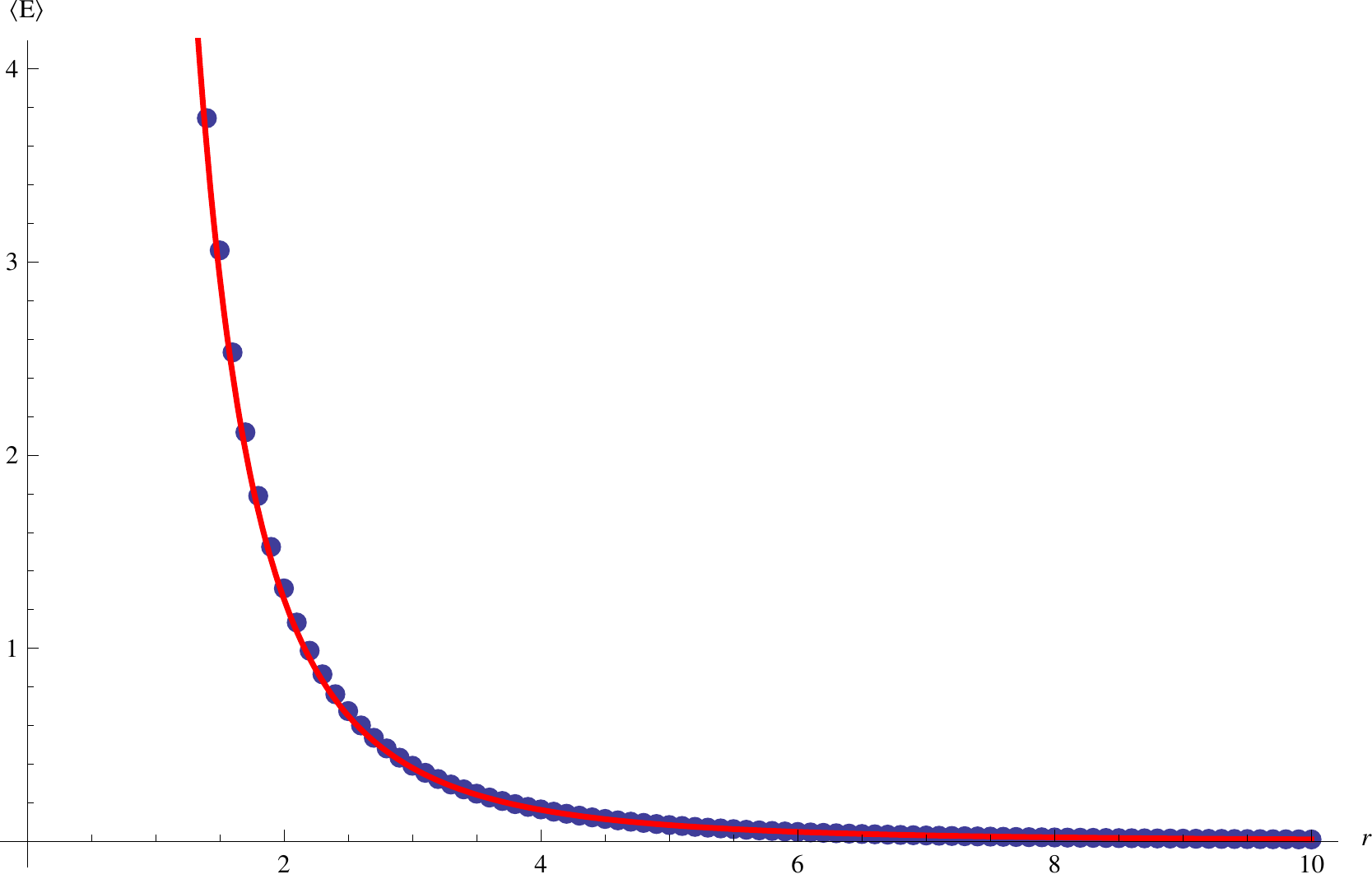}
 \caption{$\langle H \rangle$ as a function of the ``inverse temperature'' $r$.  The red line is a near perfect fit and can be described by the equation $\langle H \rangle=9.63 r^{-2.94}$.}
  \label{expectH}
\end{figure}

\begin{figure}[h!]
 \includegraphics[bb=0 0 480 310,scale=0.8]{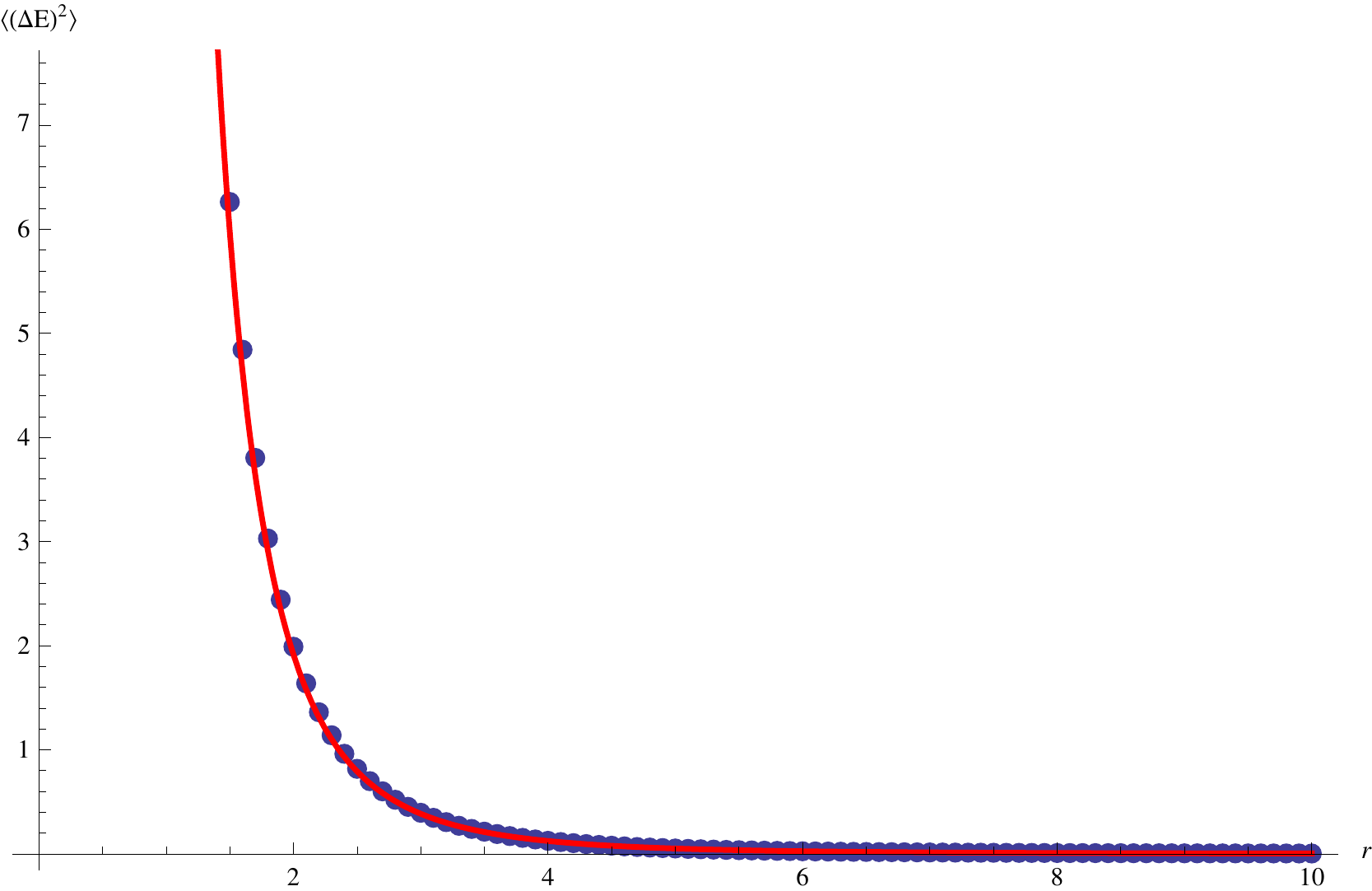}
\caption{$\langle \left(\Delta E\right)^2\rangle$ as a function of $r$.  The red line is a near perfect fit and can be described by the equation $\langle \left(\Delta E\right)^2\rangle=29.71r^{-3.957}$.}
\label{expect_variance}
\end{figure} 

In Principle we are now able to calculate various other thermodynamic quantities
and study their respective behaviour as a function of $r$, our ``inverse
temperature''.  If we are so bold as to assume that $r$ is indeed the inverse
temperature $r=\beta=\frac{1}{k_BT}$, then we are able to compute the specific
heat $C_v$ and entropy $S$ of our system 
\begin{equation}
C_v=\frac{\partial \langle E \rangle}{\partial T} = \frac{1}{k_BT^2}\langle \left(\Delta(E) \right)^2 \rangle, \quad  S=k_B\left(\ln Z+\beta\langle E\rangle\right)
\end{equation} 
as functions of temperature.  For ease of computation the calculations were
performed in units of $k_B$.  Figures \ref{Cv} and \ref{entopry} show the
respective graphical representations.

\begin{figure}[h!]
 \includegraphics[bb=0 0 480 299,scale=0.8]{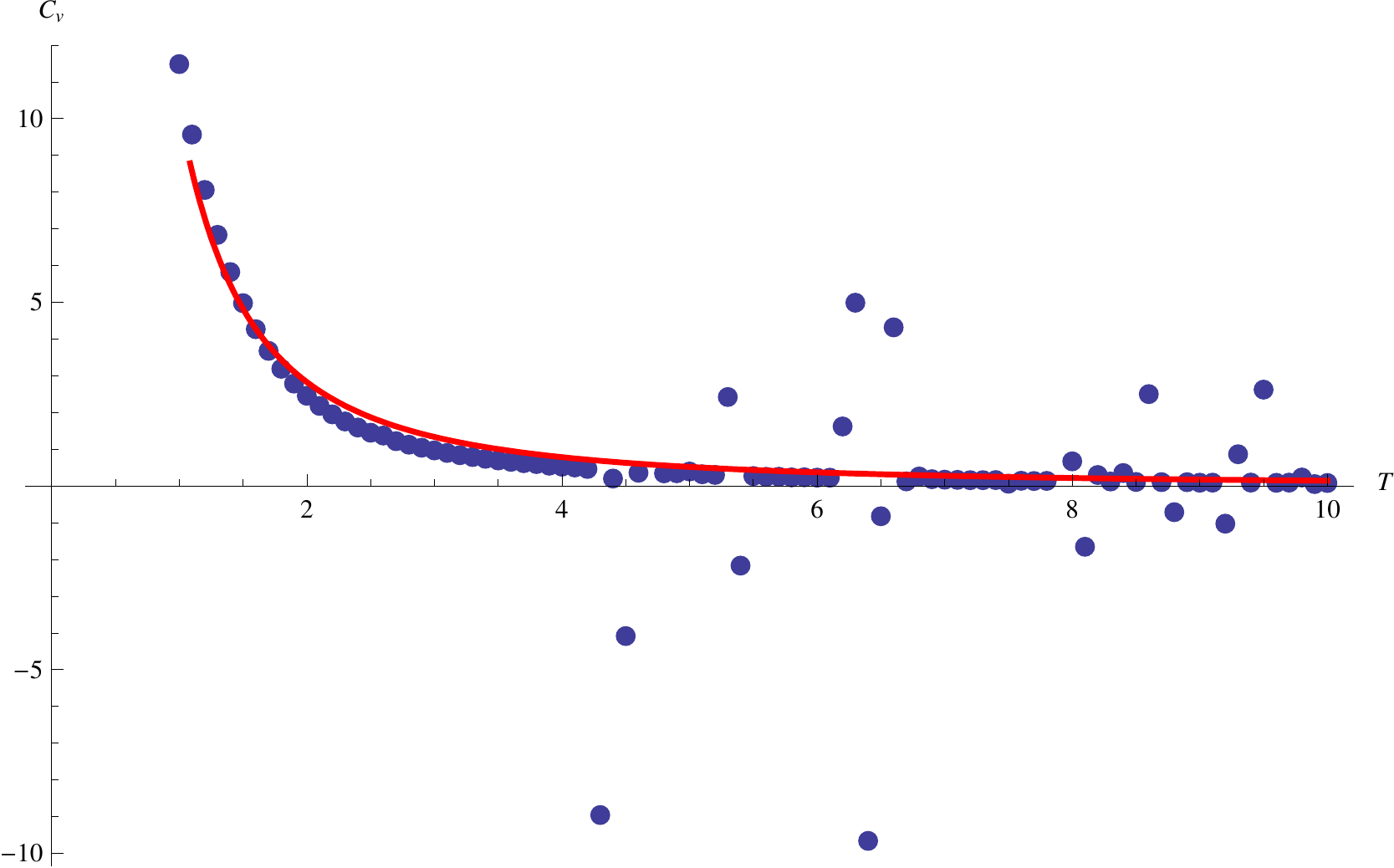}
 \caption{The specific heat of the finite dimensional representation as functions of ``temperature'', $T=\frac{1}{r}$.  The oscillating terms appearing form $T=4$ and onward occur due to convergence problems, however when studying the $T\rightarrow\infty$ limit we observe that the specific heat tends to zero.}
 \label{Cv}
\end{figure}

\begin{figure}[h!]
 \includegraphics[bb=0 0 480 299,scale=0.8]{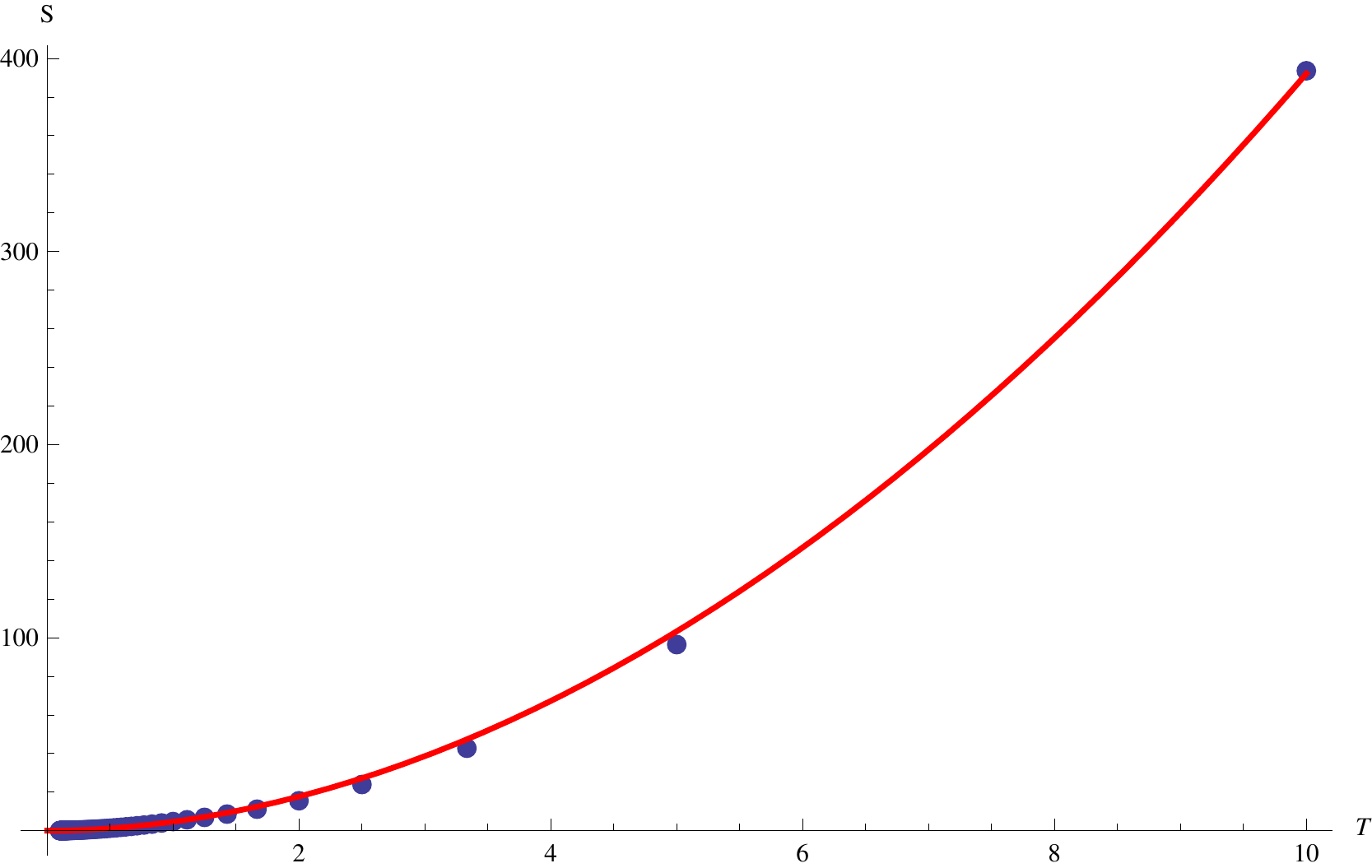}
\caption{The entropy as a function of ``temperature'' $T$.}
\label{entopry}
\end{figure} 

\newpage
The study of noncommutative path integrals is in its infancy and the
fundamentals of the theory are being developed from a $C^*$-algebraic point of
view.  The Mathai-Rosenberg $\sigma$-model \cite{MR} is an interesting
noncommutative $\sigma$-model to study since we can create a finite dimensional
representation which highlights the basic ideas of the model and enables us to
calculate specific thermodynamic quantities.  Specifically, if we make the
assumption that $r$, which was introduced in equation (\ref{x_and_y}), plays the
role of the inverse temperature then we can determine the specific heat as well
as the entropy of the finite dimensional representation as a function of the
``temperature''.  Even when making this bold assumption we do not observe clear
phase transitions in the system.  This can be ascribed to the fact that
interactions have not yet been included into the model.  In order to study
interactions from the point of view of $C^*$-algebras we have to study the
differential geometry of 
noncommutative spaces, the foundations of the most widely accepted approach
having been laid by Connes in his ground breaking book \cite{Connes}.  The
introduction of interactions in the Mathai-Rosenberg $\sigma$-model might be a
suitable project for future work. 

When considering higher dimensional representations the thermodynamic quantities
have the same general shape as in the $M_2(\mathbb{C})$ case, which leads us to
make the conjecture that when taking the ``thermodynamic'' limit, which in this
case can be seen as $n\rightarrow \infty$, the finite dimensional representation
approximates the behaviour of the case in Section \ref{qt_section}.  In the
$C^*$-algebraic framework there is still room available for improvement and the
finite dimensional representation might be able to shed some light on the
matter.

\chapter{K-Theory and Morita Equivalence}\label{KM}\noindent
\begin{center}
\begin{quotation}
 \textsl{``The most useful piece of advice I would give to a mathematics student is always to suspect an impressive sounding Theorem if it does not have a special case which is both simple and non-trivial.''\newline
-Michael Atiyah}
\end{quotation}
\end{center}
In order to study noncommutative $\sigma$-models there has to exist mappings
between the target space and world sheet.  Without such mappings we are not able
to construct an action and all dynamics will be lost.  In this chapter we
summarize the basic ideas regarding K-theory and Morita equivalence that will be
used in proving the existence theorems that appear in the following chapter. 
For completeness sake we show the construction of the $K_0$ and $K_1$ groups. 
This will give the reader some background information regarding the origin of
the abelian cancellative groups used in K-theory.  
\section{K-Theory}\noindent
We follow the same procedure as \cite[Chapter 7]{murphy} when dealing with the
K-theory of $C^*$-algebras.  Letting $A$ be any unital *-algebra, we define the
following set of projections
\begin{equation}\label{P[A]}
 P[A]:=\bigcup^{\infty}_{n=1}\{p\in M_n(A):p^2=p=p^*\}
\end{equation}
\begin{definition}(Equivalent Projections)\cite[p. 218]{murphy}\newline
Two projections $p,q\in P[A]$ are said to be \emph{equivalent} and denoted by $p\sim q$ if there is a rectangular matrix $u$ with entries in $A$ such that $p=u^*u$ and $q=uu^*$. 
\end{definition}
\begin{definition}(Stable Equivalence)\cite[p. 219]{murphy}\label{stably_equivalent}\newline
Let $A$ be a unital *-algebra.  We say that two projections $p,q\in P[A]$ are \emph{stably equivalent} and write $p\approx q$ if there is a positive integer $n$ such that $\I_n\oplus p\sim \I_n\oplus q$.  Here $\I_n$ denotes the $n\times n$ identity matrix.
\end{definition}
Now, for any $p\in P[A]$ let $[p]$ denote its stable equivalence class and let $K_0(A)^+$ denote the set of all these equivalence classes.  For any $[p],[q]\in K_0(A)^+$ define
\begin{equation}\label{productK0}
 [p]+[q]:=[p\oplus q]
\end{equation}
By Theorem 7.1.2 of \cite{murphy} we know that $K_0(A)^+$ is a cancellative
abelian semigroup with zero element $[0]$.  We will now define the enveloping or
Grothendieck group $G(N)$ of an abelian cancellative group $N$ by construction.
First of all define an equivalence relation $\asymp$ on $N\times N$ by declaring
$(x,y)\asymp(z,t)$ if, and only if
\begin{equation*}
 x+t=y+z.
\end{equation*}
Now let $[x,y]$ denote the equivalence classes of $(x,y)$ and let $G(N)$ be the set of all such equivalence classes.  $G(N)$ is an additive group with operation defined by
\begin{equation*}
 [x,y]+[z,t]:=[x+z,y+t]
\end{equation*}
From this we can easily see that the inverse of $[x,y]$ is $[y,x]$.
\begin{equation*}
 [x,y]+[y,x]=[x+y,y+x]
\end{equation*}
But $(x+y,y+x)\asymp(z,t)$ if $x+y+t=z+x+y$ which implies that $z=t$ because the cancellation property holds.

We can easily see that the map
\begin{equation*}
 \varphi:N\rightarrow G(N):x\mapsto [x,0]
\end{equation*}
is a homomorphism, since for $x,y\in N$
\begin{equation*}
 \varphi(x+y)=[x+y,0]=[x,0]+[y+0]=\varphi(x)+\varphi(y).
\end{equation*}
It is also clear form the definition of $\asymp$ that $\varphi$ is injective.  It is now natural to identify $N$ as a subsemigroup of $G(N)$ by identifying $x$ with $[x,0]$.  This can be visualised by thinking of elements of $G(N)$ to be differences of elements in $N$.  Mathematically speaking we can write
\begin{equation*}
 G(N)=\{x-y:x,y\in N\}.
\end{equation*}
Now we have enough information to define the K-theory of a *-algebra.
\begin{definition}($K_0(A)$)\cite[p. 220]{murphy}\newline
Let $A$ be a unital *-algebra.  We define $K_0(A)$ to be the Grothendieck group of $K_0(A)^+$.
\end{definition}

We can also define maps between different $K_0$ groups using *-homomorphisms.  Consider the *-homomorphism 
\begin{equation*}
 \varphi:A\rightarrow B
\end{equation*}
between *-algebras $A$ and $B$.  If $a:=(a_{ij})$ is an $m\times n$ matrix with entries in $A$ we can now extend $\varphi$ to the corresponding matrix algebra. 
\begin{equation*}
 \varphi(a)=(\varphi(a_{ij}))
\end{equation*}
where $(\varphi(a_{ij}))$ is an $m\times n$ matrix with entries in $B$.  Consider now another $n\times p$ matrix, $b$ with entries in $A$, clearly $ab$ is a $m\times p$ matrix with entries in $A$.
\begin{eqnarray*}
 \varphi(ab)&=&
\varphi\left[\left(\begin{array}{lll}
a_{11} & \dots & a_{1n}\\
 \vdots&\ddots  & \vdots\\
a_{m1} & \dots & a_{mn} 
\end{array}\right)
\left(\begin{array}{lll}
b_{11} & \dots & b_{1p}\\
 \vdots&\ddots  & \vdots\\
b_{n1} & \dots & b_{np} 
\end{array}\right)\right]\\
&=&\varphi\left(
\begin{array}{lll}
\sum^{n}_{i=1}a_{1i}b_{i1} & \dots & \sum^{n}_{i=1}a_{1i}b_{ip}\\
 \vdots&\ddots  & \vdots\\
\sum^{n}_{i=1}a_{mi}b_{i1} & \dots & \sum^n_{i=1}a_{mi}b_{in} 
\end{array}
\right)\\
&=&\left(
\begin{array}{lll}
\sum^{n}_{i=1}\varphi(a_{1i})\varphi(b_{i1}) & \dots & \sum^{n}_{i=1}\varphi(a_{1i})\varphi(b_{ip})\\
 \vdots&\ddots  & \vdots\\
\sum^{n}_{i=1}\varphi(a_{mi})\varphi(b_{i1}) & \dots & \sum^n_{i=1}\varphi(a_{mi})\varphi(b_{in}) 
\end{array}
\right)\\
&=&\left(
\begin{array}{lll}
\varphi(a_{11}) & \dots & \varphi(a_{1n})\\
 \vdots&\ddots  & \vdots\\
\varphi(a_{m1}) & \dots & \varphi(a_{mn}) 
\end{array}\right)
\left(\begin{array}{lll}
\varphi(b_{11}) & \dots & \varphi(b_{1p})\\
 \vdots&\ddots  & \vdots\\
\varphi(b_{n1}) & \dots & \varphi(b_{np}) 
\end{array}\right)\\
&=&\varphi(a)\varphi(b)
\end{eqnarray*}
In a similar way we can also show that in the case where $m=n$, we have $\varphi(a^*)=\varphi(a)^*$, so $\varphi:M_n(A)\rightarrow M_n(B)$ is a *-homomorphism.
\begin{lemma}\cite[p. 220]{murphy}\label{lem1}
 Let $A$ and $B$ be *-algebras and let $\varphi:A\rightarrow B$ be a *-homomorphism between them.  If $p\sim q$ then $\varphi(p)\sim\varphi(q)$.
\end{lemma}
The proof follows trivially from the argument just before the lemma.
\begin{lemma}\cite[p. 220]{murphy}
Let $\varphi:A\rightarrow B$ be a unital *-homomorphism between unital *-algebras.  If $p\approx q$ then $\varphi(p)\approx\varphi(q)$. 
\end{lemma}
The proof follows from standard matrix manipulation and Lemma \ref{lem1}.  

\begin{lemma}\label{welldefinedhomomorphism}
Let $A$ and $B$ be unital $C^*$-algebras, then
\begin{equation*}
 \varphi_*:K_0(A)\rightarrow K_0(B):[p]\mapsto[\varphi(p)]
\end{equation*}
is a well defined group homomorphism.
\end{lemma}
\begin{proof}
The above two results guarantee that 
\begin{equation*}
\varphi_*:K_0(A)^+\rightarrow K_0(B)^+ 
\end{equation*}
is well defined by setting
\begin{equation*}
 \varphi_*([p])=[\varphi(p)].
\end{equation*}
From matrix calculations we can show that
\begin{equation*}
\varphi(p\oplus q)=\varphi(p)\oplus\varphi(q) 
\end{equation*}
which then implies that
\begin{equation*}
 \varphi_*([p]+[q])=\varphi_*([p])+\varphi_*([q]).
\end{equation*}
Hence $\varphi_*$ is a homomorphism.  These results can then be extended to show that
\begin{equation*}
 \varphi_*:K_0(A)^+\rightarrow K_0(B)
\end{equation*}
is a well defined homomorphism.  Finally, the properties of the Grothendieck group $G(N)$ then imply that the map
\begin{equation*}
 \varphi_*:K_0(A)\rightarrow K_0(B):[p]\mapsto[\varphi(p)]
\end{equation*}
is a well defined group homomorphism.
\end{proof}
%put the definition of the non unital case here%

\begin{remark}\label{nonunitalK}
Let $A$ be a non-unital $C^*$-algebra.  We denote its unitization by
$\widetilde{A}:=A\oplus\mathbb{C}$.  If
$\tau:\widetilde{A}\rightarrow\mathbb{C}$ is the canonical *-homomorphism
\cite[p. 208]{murphy}, we set $\widetilde{K}_0(A):=\ker(\tau_*)$.  Hence
$\widetilde{K}_0(A)$ is a subgroup of $K_0(\widetilde{A})$.  If
$\varphi:A\rightarrow B$ is a *-homomorphism of $C^*$-algebras and
$\widetilde{\varphi}:\widetilde{A}\rightarrow \widetilde{B}$ is the unique
unital *-homomorphism extending $\varphi$, then
$\widetilde{\varphi}_{*}\left(\widetilde{K}_0(A) \right)\subset
\widetilde{K}_0(B)$.  Hence we get a homomorphism
$\varphi_*:\widetilde{K}_0(A)\rightarrow \widetilde{K}_0(B)$ restricting
$\widetilde{\varphi}_*$ \cite[p. 229]{murphy}.  Now let $\I$ be the unit of
$\widetilde{A}$ and consider any $x\in \widetilde{K}_0(A)$.  According to the
construction of the Grothendieck group $G(N)$ we can write $x=[r]-[q]$ for two
projections $r$ and $q$, both of which we may suppose to be elements of
$M_n(\widetilde{A})
$ for some $n$.  Then we can write
\begin{eqnarray*}
 x=[r]+[\I_n -q] -[\I_n]&=&[r\oplus(\I_n -q)]-[\I_n]\\
&:=&[p]-[\I_n]
\end{eqnarray*}
where $p\in P[\widetilde{A}]$.  This enables us to use the theory developed
earlier to treat the nonunital case by extending any nonunital $C^*$-algebra to
its unitization.
\end{remark}

\begin{definition}\label{cone_suspension}(Cone, Suspension of $C^*$-Algebra)\cite[p. 246]{murphy}
Let $A$ be a $C^*$-algebra.  The \emph{cone} of $A$ is defined by
\begin{equation*}
 \mathcal{C}(A):=\{f\in A[0,1]:f(1)=0\}
\end{equation*}
The \emph{suspension} of $A$ is defined by
\begin{equation*}
 S(A):=\{f\in \mathcal{C}(A): f(0)=0\}.
\end{equation*}
Here $A[0,1]$ is a shorthand notation for $C([0,1],A)$.
\end{definition}
Note that $C([0,1],A)$ is a $C^*$-algebra equipped with pointwise operations,
sup norm and for any $f\in C([0,1],A)$ the involution given by
$f^*(x)=\left(f(x)\right)^*$.
\begin{lemma}
Let $A$ be a $C^*$-algebra.  The cone, $\mathcal{C}(A)$ is a $C^*$-algebra. 
\end{lemma}
\begin{proof}
It is clear that $\mathcal{C}(A)$ is a *-algebra.  It only remains to check that $\mathcal{C}(A)$ is complete and satisfies the $C^*$-algebra norm.  Let $\{f_n\}$ be a Cauchy sequence in $\mathcal{C}(A)$.  Given $\epsilon>0$, then there is an $N$ such that 
\begin{eqnarray*}
 \|f_n-f_m\|=\sup_{x\in [0,1]}\|f_n(x)-f_m(x)\|<\frac{\epsilon}{3}
\end{eqnarray*}
whenever $n,m>N$.  Fix $n>N$ and select some $\delta>0$ for which we have
\begin{equation*}
 \|f_n(x)-f_m(y)\|<\frac{\epsilon}{3}\text{ whenever }\|x-y\|<\delta.
\end{equation*}
Letting $m\rightarrow \infty$ we get $\|fn-f\|\leq\frac{\epsilon}{3}$.  Then we
can write
\begin{eqnarray*}
\|f(x)-f(y)\|&=&\|f(x)-f_n(x)+f_n(x)-f_n(y)+f_n(y)-f(y)\|\\
 &\leq&\|f(x)-f_n(x)\|+\|f_n(x)-f_n(y)\|+\|f_n(y)-f(y)\|\\
&<&\epsilon
\end{eqnarray*}
which shows that $f$ is continuous.  Next note that
\begin{eqnarray*}
 \|f_n-f\|&=&\sup_{x\in[1,0]}\left\|f_n(x)-f(x) \right\|\\
&=&\sup_{x\in[1,0]}\left\|f_n(x)-\lim_{m\rightarrow\infty}f_m(x) \right\|\\
&=&\sup_{x\in[1,0]}\lim_{m\rightarrow\infty}\left\|f_n(x)-f_m(x) \right\|\\
&<&\epsilon.
\end{eqnarray*}
Hence $f_n$ converges to $f$ in $\mathcal{C}(A)$.  Lastly we need to verify that
$f$ is in $\mathcal{C}(A)$.  Clearly
\begin{equation*}
 0=\lim_{n\rightarrow \infty}f_n(1)=f(1).
\end{equation*}
So $\mathcal{C}(A)$ is complete and hence a Banach algebra.  It remains to
verify that the norm is a $C^*$-algebra norm.  Let $f\in\mathcal{C}(A)$
\begin{eqnarray*}
 \|f\|^2&=&\sup_{x\in[0,1]} \|f(x) \|^2\\
&=& \sup_{x\in[0,1]} \| \left(f(x)\right)^*f(x)\|\\
&=& \sup_{x\in[0,1]} \| \left(f^*f\right)(x)\|\\
&=& \|f^*f\|
\end{eqnarray*}
This concludes the proof.
\end{proof}
\begin{lemma}
Let $A$ be a $C^*$-algebra.  The suspension, $S(A)$ is a closed ideal in $\mathcal{C}(A)$ and therefore a $C^*$-algebra.
\end{lemma}
\begin{proof}
We will show that $S(A)$ is a closed, two sided ideal in $\mathcal{C}(A)$.  Involution on $S(A)$ is inherited from $\mathcal{C}(A)$.  For any $f\in S(A)$, $g\in\mathcal{C}(A)$ and $x\in[0,1]$ we have
\begin{eqnarray*}
(fg)(x)&=&f(x)g(x)\\
(fg)(1)=f(1)g(1)&=&0=g(1)f(1)=(gf)(1)\\
(fg)(0)=f(0)g(0)&=&0=g(0)f(0)=(gf)(0), 
\end{eqnarray*}
hence $fg,gf\in S(A)$ and $S(A)$ is an ideal of $\mathcal{C}(A)$.  Since both
$f$ and $g$ are continuous it is clear that $fg$ and $fg$ are also continuous. 
It remains to show that $S(A)$ is closed.  Consider a sequence, $\{f_n\}$ in
$S(A)$ converging to $f\in\mathcal{C}(A)$; so we already have $f(1)=0$.  Since
these functions are continuous we can write
\begin{equation*}
 \lim_{n\rightarrow\infty}f_n(x)=f(x)
\end{equation*}
and in particular
\begin{equation*}
 0=\lim_{n\rightarrow\infty}f_n(0)=f(0). 
\end{equation*}
This shows that $S(A)$ is a closed ideal of the $C^*$-algebra $\mathcal{C}(A)$ and hence is a $C^*$-algebra in its own right.
\end{proof}
Notice that $S(A)$ is a nonunital $C^*$-algebra and when we want to calculate its K-theory we have to make use of Remark \ref{nonunitalK}.
\begin{definition}($K_1(A)$)\newline
 $K_1(A)$ is defined as the $K_0$ group of the suspension of $A$
\begin{equation*}
 K_1(A):=K_0(S(A)).
\end{equation*}
\end{definition}
\begin{definition}(Stably Isomorphic)\newline
 Two $C^*$-algebras $A$ and $B$ are said to be \emph{stably isomorphic} if $A\otimes K$ is isomorphic to $B\otimes K$ where $K$ is the space of compact operators on some separable infinite dimensional Hilbert space $H$.
\end{definition}
\begin{proposition}\label{isoK}
 Stably isomorphic $C^*$-algebras have isomorphic $K_0$-groups.
\end{proposition}
\begin{proof}
The proof follows from \cite[Theorem 7.4.3]{murphy}.
\end{proof}
A fundamental result from K-theory is the six term exact sequence, also known as the PV-sequence, which will be used in determining the K-theory of the quantum torus.  The results are given below without proof.
\begin{theorem}\label{PV}(Pimsner and Voiculescu Short Exact
Sequence)\cite[Theorem 2.4]{PV1980}\newline
 Suppose $\alpha$ is a *-automorphism of the $C^*$-algebra $A$.  Then there is a six term exact sequence
\begin{eqnarray*}
 \begin{diagram}
  K_0(A)&\rTo^{\text{Id}_*-\alpha_*}
&K_0(A)& \rTo^{i_*}&K_0(A\rtimes_{\alpha}\mathbb{Z})\\
  \uTo&&&&\dTo\\
K_1(A\rtimes_{\alpha}\mathbb{Z})&\lTo^{i_*}&K_1(A)&\lTo^{\text{Id}_*-\alpha_*}
&K_1(A)
 \end{diagram}
\end{eqnarray*}
\end{theorem}
When determining the K-theory of the quantum torus we refer to Theorem \ref{PV}
as the PV short exact sequence.  Regarding the notation used in Theorem
\ref{PV}, $i:A\rightarrow A\rtimes_{\alpha}\mathbb{Z}$ is the inclusion map, Id
the identity map and $\alpha$ any *-automorphism of the $C^*$-algebra $A$. 
Since $\alpha$ is a *-automorphism, $\alpha_*$ is a group-homomorphism, which
in turn ensutes that Id$_*-\alpha_*$ is the difference of two
group-homomorphisms and hence a group-homomorphism in its own right.

\subsection{K-Theory of the Quantum Torus}\noindent
In order to fully understand the $C^*$-algebra which is the quantum torus we need to determine its K-theory.  Determining the K-theory of the quantum torus was a difficult task, however Rieffel, Pimsner and Voiculescu succeeded in calculating it in their articles \cite{PV1980b,PV1980,Rieffel1981,Rieffel1983}.  The main tool which we will be using to find the K-theory of the quantum torus is the PV-sequence which we introduced in Theorem \ref{PV}.  
\begin{proposition}\label{K0groupoftorus}
Let $\theta$ be some irrational number.  The $K_0$ group of $A_{\theta}$ is
$\mathbb{Z}\oplus\mathbb{Z}$.
\end{proposition}
Let us consider the PV-short exact sequence for the $C^*$-algebra $C(\mathbb{T})$.  Using Theorem \ref{PV} we find the following six term exact sequence for $\alpha:C(\mathbb{T})\rightarrow C(\mathbb{T}):f\mapsto R_{\theta}$ where $R_{\theta}$ is rotation by $\theta$
\begin{equation*}
 \begin{diagram} 
K_0(C(\mathbb{T}))&\rTo^{\text{Id}_*-\alpha_*}&K_0(C(\mathbb{T}))&\rTo^{i_*}
&K_0(C(\mathbb{T})\rtimes_{\alpha}\mathbb{Z})\\
  \uTo&&&&\dTo\\
K_1(C(\mathbb{T})\rtimes_{\alpha}\mathbb{Z})&\lTo^{i_*}&K_1(C(\mathbb{T}))&
\lTo^{\text{Id}_*-\alpha_*}&K_1(C(\mathbb{T}))
 \end{diagram}
\end{equation*}
By Theorem \ref{crossed_product} we know that we can write
$A_{\theta}=C(\mathbb{T})\rtimes_{\alpha}\mathbb{Z}$.  Let us now substitute
this result into the above expression, then we find
\begin{equation}\label{torus_K}
 \begin{diagram}
  K_0(C(\mathbb{T}))&\rTo^{\text{Id}_*-\alpha_*}&K_0(C(\mathbb{T}))&\rTo^{i_*}&K_0(A_{\theta})\\
  \uTo&&&&\dTo\\
  K_1(A_{\theta})&\lTo^{i_*}&K_1(C(\mathbb{T}))&\lTo{\text{Id}_*-\alpha_*}&K_1(C(\mathbb{T}))
 \end{diagram}
\end{equation}
So now we are left with finding the K-theory of $C(\mathbb{T})$.  Let 
\begin{equation*}
\epsilon:C(\mathbb{T})\rightarrow \mathbb{C}:f\mapsto f(1)
\end{equation*}
and let 
\begin{equation*}
j:S\rightarrow C(\mathbb{T})
\end{equation*}
 be the inclusion, where $S$ is the suspension (see Definition \ref{cone_suspension}) of $C(\mathbb{T})$.  We follow \cite[Example 7.5.1]{murphy} to find the K-theory of $C(\mathbb{T})$.  Clearly $\text{im}(j)=\text{ker}(\epsilon)$ and we can define a *-homomorphism 
\begin{equation*}
\epsilon':\mathbb{C}\rightarrow C(\mathbb{T}):z\mapsto f_z
\end{equation*} 
where $f_z(x)=z$ for any $z\in \mathbb{C}$ and $x\in\mathbb{T}$.  In other words $f_z$ is the constant function identically equal to $z$.  This implies that we can write
\begin{equation*}
 \epsilon'(z+w)=f_{z+w}=z+w=f_z+f_w=\epsilon'(z)+\epsilon'(w).
\end{equation*}
So, $\epsilon'$ is a *-homomorphism and $\epsilon\epsilon'=\text{Id}$.  This implies that the diagram
\begin{equation*}
\begin{diagram}
 0&\rTo &S& \rTo^{j} &C(\mathbb{T})& \rTo^{\epsilon}&\mathbb{C}&\rTo 0
\end{diagram}
\end{equation*}
is a split short exact sequence of $C^*$-algebras.  By \cite[Theorem 6.5.2]{murphy} we know that for any arbitrary $C^*$-algebra $A$  
\begin{equation*}
\begin{diagram}
 0&\rTo & S\otimes_{*}A & \rTo^{j\otimes_*\text{Id}_A} & C(\mathbb{T})\otimes_{*}A & \rTo^{\epsilon\otimes_*\text{Id}_A} & \mathbb{C}\otimes_{*}A&\rTo & 0
\end{diagram}
\end{equation*}
is a split short exact sequence.  By \cite[Remarks 7.5.3, 7.5.4]{murphy} we know that for $i=0,1$ 
\begin{equation*}
\begin{diagram}
 0 & \rTo & K_i\left(S\otimes_{*}A\right) & \rTo^{j\otimes_*\text{Id}_A} & K_i\left(C(\mathbb{T})\otimes_{*}A\right) & \rTo^{\epsilon\otimes_*\text{Id}_A} & K_i\left(\mathbb{C}\otimes_{*}A\right) & \rTo & 0
\end{diagram}
\end{equation*}
is a split short exact sequence of K-groups.  Thus, by Lemma \ref{split_short_exact} we have
\begin{eqnarray*}
 K_i\left(C(\mathbb{T})\otimes_{*}A\right)&\cong&K_i\left(S\otimes_{*}A\right)\oplus K_i\left(\mathbb{C}\otimes_{*}A\right)\\
&\cong&K_i\left(S(A)\right)\oplus K_i\left(A\right)\\
&\cong&K_{i-1}\left(A\right)\oplus K_i\left(A\right)\\
&\cong&K_0\left(A\right)\oplus K_1\left(A\right).
\end{eqnarray*}
Now in particular, if we replace the $C^*$-algebra $A$ by $\mathbb{C}$ we have
\begin{eqnarray*}
 K_i\left(C(\mathbb{T})\otimes_{*}\mathbb{C}\right)&\cong&K_i\left(C(\mathbb{T})\right)\\
&\cong&K_0\left(\mathbb{C}\right)\oplus K_1\left(\mathbb{C}\right)\\
&\cong&\mathbb{Z}.
\end{eqnarray*}
Substituting the above result back into equation (\ref{torus_K}) we find 
\begin{equation}\label{torus_split}
 \begin{diagram}
  \mathbb{Z} &\rTo^{\text{Id}_*-\alpha_*} & \mathbb{Z} & \rTo^{i_*} & K_0(A_{\theta})\\
  \uTo&&&&\dTo\\
  K_0(A_{\theta})&\lTo^{i_*}& \mathbb{Z}&\lTo^{\text{Id}_*-\alpha_*}&\mathbb{Z}
 \end{diagram}.
\end{equation}
Recall that for some irrational $\theta$ we define
\begin{equation*}
 \alpha:C(\mathbb{T})\rightarrow C(\mathbb{T}):f\mapsto f\circ R_{\theta}
\end{equation*}
where $R_{\theta}$ is just rotation by $\theta$, and that we can write the
quantum torus as the crossed product
$A_{\theta}=C(\mathbb{T})\rtimes_\alpha\mathbb{Z}$ as in Theorem
\ref{crossed_product}.  On the level of $K_0$-groups $\alpha$ gives 
\begin{equation*}
 \alpha_*:K_0(A_0)^+\rightarrow K_0(A_0)^+.
\end{equation*}
Since $\alpha$ is a *-automorphism $\alpha_*$ is a semigroup automorphism.  Since $K_0(A_0)=\mathbb{Z}$ it implies that $K_0(A_0)^+=\{0,1,2,\cdots\}=\mathbb{N}_0$ because $K_0(A_0)^+$ is the only semigroup with identity such that $\mathbb{Z}=\{m-n: m,n\in K_0(A_0)^+\}$.  So in effect the automorphism $\alpha_*$ maps $\mathbb{N}_0$ to itself.  Now suppose that $\alpha_*(1)=m$, then $\alpha_*(n)=mn$.  So the only way for $\alpha_*$ to be surjective is when $m=1$, hence $\alpha_*=\text{Id}_{\mathbb{N}_0}$.  If we now extend $\alpha_*$ to $\mathbb{Z}$ then clearly $\alpha_*=\text{Id}_{\mathbb{Z}}$, or more precisely
\begin{equation*}
\alpha_*=\text{Id}_{\mathbb{Z}}:K_0(A_0)=\mathbb{Z} \rightarrow \mathbb{Z}.
\end{equation*}

We would like to extend the above reasoning to the $K_1$-groups and work with
group homomorphisms of the form
\begin{equation*}
 \alpha_*:K_1(A_{0})\rightarrow K_1(A_{0}). 
\end{equation*}
From the definition of the $K_1$-group and using the $\widetilde{K}_0$-notation for the unitization we write
\begin{equation*}
 \alpha_{*}:\widetilde{K}_0(S(A_0))\rightarrow \widetilde{K}_0(S(A_0)).
\end{equation*}
Consider $\alpha_*$ defined on $K_0\left(\widetilde{S(A_0)}\right)$ and restricted to $K_0\left(\widetilde{S(A_0)}\right)^+$
\begin{equation*}
 \widetilde{\alpha}|_{K_0\left(\widetilde{S(A_0)}\right)^+}:K_0\left(\widetilde{S(A_0)}\right)^+\rightarrow K_0\left(\widetilde{S(A_0)}\right)^+.
\end{equation*}

\begin{remark}
 Suppose we have a *-homomorphism
\begin{equation*}
 \tau:\widetilde{S(A_0)}\rightarrow \mathbb{C},
\end{equation*}
then we can write
\begin{equation*}
 \tau^{+}_{*}:K_0\left(\widetilde{S(A_0)}\right)^+\rightarrow K_0\left(\mathbb{C}\right)^+=\mathbb{N}_0.
\end{equation*}
We can extend this to
\begin{equation*}
 \tau_*:K_0\left(\widetilde{S(A_0)}\right)\rightarrow K_0\left(\mathbb{C}\right).
\end{equation*}
From \cite[p. 262]{murphy} we know that
$\widetilde{K}_0\left(S(A_0)\right)=\mathbb{Z}$.  Then from the definition of
$\widetilde{K}_0\left(S(A_0)\right)$
we have $\widetilde{K}\left(S(A_0)\right):=\ker\left(\tau_*\right)$ which is a
subgroup of $K_0\left(\widetilde{S(A_0)}\right)=\mathbb{Z}$.  For the
$C^*$-algebra $A_0$ we can regard
$\widetilde{K}_0\left(S(A_0)\right)^+=\mathbb{N}_0$.
\end{remark}
So we have to have
\begin{equation*}
 \mathbb{N}_0=\widetilde{K}_0\left(S(A_0)\right)^+\subset K_0\left(\widetilde{S(A_0)}\right)^+
\end{equation*}
since it has to contain all the non negative elements.  Let $m\in
\mathbb{N}_0$.  Then
\begin{equation*}
 \alpha_*(m)=\widetilde{\alpha}_*(m)\in K_0\left(\widetilde{S(A_0))}\right)^+
\end{equation*}
is non negative, but from \cite[p. 262]{murphy} we know that $\alpha_*(m)\in
\mathbb{Z}=K_0\left(S(A_0)\right)$ if $\alpha_*(m)\in
\mathbb{N}_0=\widetilde{K}_0\left(S(A_0)\right)^+$.  Hence, similarl to the case
of 
\begin{equation*}
 \alpha_*:K_0\left(A_0\right)\rightarrow K_0\left(A_0\right)
\end{equation*}
we find that $\alpha_*$ is also the identity homomorphism in the case of
\begin{equation*}
 \alpha_*:K_1\left(A_0\right)\rightarrow K_1\left(A_0\right).
\end{equation*}  

If we now consider equation (\ref{torus_split}) together with the above reasing considering the $\alpha_*$ homomorphisms we see that $\text{Id}_*-\alpha_*=0$ for both the $K_0$ and  $K_1$ groups and we have the following split short exact sequence of $C^*$-algebras
\begin{eqnarray*}
 0\rightarrow \mathbb{Z}\rightarrow K_i\left(A_{\theta}\right)\rightarrow \mathbb{Z}\rightarrow 0.
\end{eqnarray*}
Lemma \ref{split_short_exact} then implies that
\begin{equation*}
 K_i\left(A_{\theta}\right)=\mathbb{Z}\oplus\mathbb{Z}.
\end{equation*}
This proves Proposition \ref{K0groupoftorus}.
%% sit miskien n bewys is wat wys wat die K-teorie van die komplekse getalle is%%

\subsection{Traces and $K_0$}
In chapter 4 we will be using the unique trace of the quantum torus to define
mappings onto the K-groups of the quantum torus.  The following proposition then
naturally comes into play.
\begin{proposition}
A faithful trace on any $C^*$-algebra induces a homomorphism of the $K_0$-groups.
\end{proposition}
\begin{proof}
 Let A be a $C^*$-algebra with faithful trace denoted by $\tau_A:A\rightarrow \mathbb{C}$.  Now we define a mapping $\hat{\tau}_A:K_0(A)\rightarrow \mathbb{C}$ by
\begin{equation}\label{regmaak}
 \hat{\tau}_A([p]):=\tau_A(p):=\sum^n_{i=1}\tau_A(p_{ii}),
\end{equation}
where $p=(p_{kl})\in P[A]$ is an $n\times n$ matrix as defined in equation
(\ref{P[A]}).  The middle term in equation \ref{regmaak} represents an
extension of the trace to $M_n(A)$.  Such an extension will be made where
convenient, and the same notation will be used. Let $[p]=[q]\in K_0(A)$.  We
would like to show that $\tau_A$ is well defined. From Definition
\ref{stably_equivalent} there exists an $m\times n$ matrix $u$ with entries in
$A$ and an integer $r$ such that 
\begin{equation}\label{uu}
 \I_r\oplus p =u^*u\  \text{  and  }\  \I_r\oplus q=uu^*.
\end{equation}
The square matrices $u^*u$ and $uu^*$ are in different matrix algebras over $A$.  Let us first consider $u^*u$
\begin{eqnarray*}
\left( 
\begin{array}{lll}
u^*_{11} & \cdots & u^*_{m1} \\
\vdots &  & \vdots\\
u^*_{1n} & \cdots & u^*_{mn}
 \end{array}
 \right)
\left( 
\begin{array}{lll}
 u_{11} & \cdots & u_{1n} \\
\vdots &  & \vdots\\
u_{m1} & \cdots & u^*_{mn}
 \end{array}
 \right)=
\left( 
\begin{array}{lll}
 \sum^{m}_{i=1}u^*_{i1}u_{i1} & \cdots & \sum^{m}_{i=1}u^*_{i1}u_{in} \\
\vdots &  & \vdots\\
\sum^{m}_{i=1}u^*_{in}u_{i1} & \cdots & \sum^{m}_{i=1}u^*_{in}u_{in}
 \end{array}
 \right).
\end{eqnarray*}
We can then write the trace of $u^*u$ as
\begin{equation*}
 \tau_{A}\left(u^*u\right)=\sum^{n}_{j=1}\sum^{m}_{i=1}\tau_{A}\left(u^*_{ij}u_{ij}\right)
\end{equation*}
and the same result holds for $uu^*$ after making use of the cyclic property of the faithful trace
\begin{equation*}
 \tau_{A}\left(u^*_{ij}u_{ij}\right)=\tau\left(u_{ij}u^*_{ij}\right).
\end{equation*}
We determine the traces of (\ref{uu})
\begin{eqnarray*}
 \tau_A(\I_r\oplus p)=\tau_A(u^*u)=\tau_A(uu^*)=\tau_A(\I_r\oplus q).
\end{eqnarray*}
Furthermore from standard properties of the trace we see that
\begin{equation*}
 \tau_A(\I_r)+\tau_A(p)=\tau_A(\I_r)+\tau_A(q).
\end{equation*}
Clearly this implies that the traces of the matrices $p$ and $q$ are the same.  Hence $\hat{\tau}_A([p])=\hat{\tau}_A([q])$ so $\hat{\tau}_{A}$ is well defined.  Now let $[p],[q]\in K_0(A)$ with $p\in M_n(A)$ and $q\in M_m(A)$ and consider
\begin{eqnarray*}
 \hat{\tau}_A([p]+[q])&=&\hat{\tau}_A([p\oplus q])\\
&=&\sum^{n+m}_{i=1}\tau_A\left((p\oplus q)_{ii}\right)\\
&=&\sum^{n}_{i=1}\tau_A((p)_{ii})+\sum^{m}_{i=1}\tau_A((q)_{ii})\\
&=&\hat{\tau}_A([p])+\hat{\tau}_A([q]).
\end{eqnarray*}
So $\hat{\tau}_A$ is a group homomorphism.
\end{proof}

\begin{lemma}
 Projections in quantum tori are determined up to unitary equivalence by their traces
\end{lemma}
\begin{proof}
 For details of the proof see \cite[Corollary 2.5]{Rieffel1983}.
\end{proof}
\begin{lemma}\label{trace_iso}
 $K_{0}(A_{\theta})$ is mapped isomorphically to the ordered group $\mathbb{Z}+\theta\mathbb{Z}$ by the unique normalized trace on $A_{\theta}$
\end{lemma}
\begin{proof}
For details of the proof see \cite[Proposition 1.4]{Rieffel1981} and \cite[Corollary 2.6]{PV1980b}.
\end{proof}
\begin{lemma}\label{proj}
 The range of the trace $\tau$ on projections from $A_{\theta}$ to itself is
precisely $\left(\mathbb{Z}+\theta\mathbb{Z}\right)\bigcap [0,1]$
\end{lemma}
\begin{proof}
The details of the proof can be found in \cite[p.170-180]{Davidson}.  The
original results were published by Rieffel \cite{Rieffel1981}, Pimsner and
Voiculescu \cite{PV1980}.  
\end{proof}
The following definition follows naturally from the previous Lemma.
\begin{definition}\textsl{(Ordering of $K_0\left(A_{\theta}\right)$)}\label{posel}\newline
Identify the positive elements of $\mathbb{Z}+\theta\mathbb{Z}$ with
$K_0\left(A_{\theta}\right)^+$.  Let $g,h\in K_0\left(A_{\theta}\right)$.  Then
if
\begin{equation*}
 h-g\in K_0\left(A_{\theta}\right)^+
\end{equation*}
we have $g\leq h$.
\end{definition}

\begin{remark}\label{id_class}
 We know that the isomorphism $\psi:K_0(A_{\theta})\rightarrow \mathbb{Z}+\theta\mathbb{Z}$ is induced by the trace on the quantum torus.  We can now easily determine what element of $K_0(A_{\theta})$ will be mapped to $1$:
\begin{eqnarray*}
 \psi([\I^{(1)}_{A_{\theta}}])&=&\tau(\I^{(1)}_{A_{\theta}})\\
&=&\sum^{1}_{i=1}\tau(\I_{A_{\theta}})\\
&=&1.
\end{eqnarray*}
Where $\I^{(1)}_{A_{\theta}}$ denotes the $1\times 1$ matrix with the identity elements of $A_{\theta}$ along the diagonal.
\end{remark}
\begin{lemma}\label{group_incl}
Let $G$ and $H$ be additive subgroups of $\mathbb{R}$ both containing the number
$1$.  If $\eta:G\rightarrow H$ is an order preserving group homomorphism such
that $\eta(1)=1$, then $\eta(g)=g$ for any $g\in G$.
\end{lemma}
\begin{proof}
Since $\eta(1)=1$, we can write $\eta(\mathbb{Z})=\mathbb{Z}$, and so the result
is trivial for integer multiples of the identity.  Suppose that for some $g\in
G$ we have $\eta(g)>g$, in other words $\eta(g)-g>0$.  Then there exists a
natural number $n$ such that
\begin{equation*}
n\left[\eta(g)-g\right]=n\eta(g)-ng>1.
\end{equation*}
Hence there exists an integer $m$ such that
\begin{equation*}
 \eta(ng)=n\eta(g)>m>ng.
\end{equation*}
But since $m\in\mathbb{Z}$ we have $m=\eta(m)>\eta(ng)$ which implies that
\begin{equation*}
 \eta(ng)>\eta(ng)
\end{equation*}
which is a contradiction, so we cannot have $\eta(g)>g$ for some $g\in G$.

Now suppose that $\eta(g)<g$ for all $g\in G$.  Since $\eta$ is a homomorphism we can write
\begin{equation*}
 \eta(-g)>-g
\end{equation*}
which reduces to the previous argument and we once again have a contradiction.  Hence we conclude that $\eta(g)=g$ for any $g\in G$.
\end{proof}
We can extend the above result to the following more general scenario.
\begin{lemma}\label{group}
Let $G$ and $H$ be additive subgoups of $\mathbb{R}$ both containing the number $p$.  If $\eta:G\rightarrow H$ is an order preserving group homomorphism such that $\eta(p)=rp$ with $r\in\mathbb{Z}$ then $\eta(g)=rg$ for any $g\in G$. 
\end{lemma}
\begin{proof}
Suppose that for some $g\in G$ we have $\eta(g)>rg$, in other words we can write
$\eta(g)-rg>0$.  Then there exists a natural number $n$ such that
\begin{equation*}
 n\left[\eta(g)-rg\right]=n\eta(g)-nrg>p.
\end{equation*}
Hence there exists an integer $m$ such that
\begin{equation*}
n\eta(g)>mp>nrg.
\end{equation*}
But $\eta(mp)=m\eta(p)=mrp>\eta(nrg)=nr\eta(g)$, and hence $mp>n\eta(g)$ we
have $n\eta(g)>n\eta(g)$, which is a contradiction.

Conversely, suppose that $\eta(g)<rg$ for some $g\in G$.  Since $\eta$ is a
homomorphism, we can write this as $\eta(-g)>r(-g)$ and then the above argument
once again leads to a contradiction.  In the end we conclude that $\eta(g)=rg$
for all $g\in G$.
\end{proof}

There remains a final scenario regarding the group homomorphisms which is of importance.
\begin{lemma}\label{zeromap}
 Let $G$ and $H$ be additive subgoups of $\mathbb{R}$.  If $\eta:G\rightarrow H$ is an order preserving group homomorphism such that $\eta(1)=0$ then $\eta(g)=0$ for any $g\in G$.
\end{lemma}
\begin{proof}
The proof is similar to those of Lemmas \ref{group} and \ref{group_incl}.
\end{proof}

\begin{proposition}\label{group_homo}
 If $\varphi:A_{\Theta}\rightarrow A_{\theta}$ is a unital *-homomorphism between quantum tori then $\varphi_*:K_0(A_{\Theta})\rightarrow K_0(A_{\theta})$ is the inclusion map, in other words, $\varphi_*(g)=g$ for any $g\in \mathbb{Z}+\Theta\mathbb{Z}$.
\end{proposition}
\begin{proof}
We have $\varphi(\I_{A_{\Theta}})=\I_{A_{\theta}}$ since $\varphi$ is unital.  We know that for any $[p]\in K_0(A_{\Theta})$, $\varphi_*$ is defined by
\begin{equation*}
 \varphi_*([p])=[\varphi(p)]
\end{equation*}
So in particular we have
\begin{equation*}
 \varphi_*([\I_{A_{\Theta}}])=[\varphi(\I_{A_{\Theta}})]=[\I_{A_{\theta}}]
\end{equation*}
Then clearly the class of the identity is sent to the class of the identity. 
Furthermore from Lemma (\ref{welldefinedhomomorphism}) we know that $\varphi_*$
is a group homomorphism.  It only remains to show that the order is preserved. 
Let $0<g \in\mathbb{Z}+\Theta\mathbb{Z}$.  Then according to Lemma \ref{proj}
there is some $n\times n$ matrix $p$ with entries in $A_{\Theta}$ which is a
projection such that $\psi([p])=\tau(p)=g$.  Furthermore, we have
\begin{equation*}
 \varphi_*([p])=[\varphi(p)]=[\varphi(p)^*\varphi(p)]\in K_0(A_{\theta})
\end{equation*}
But by Lemma \ref{trace_iso}
$[\varphi(p)^*\varphi(p)]\mapsto\tau(\varphi(p)^*\varphi(p))\geq0$.  So
$[\varphi(p)]\in K_0\left(A_{\theta}\right)^+$ and $\varphi_*$ is an order
preserving group homomorphism sending the class of the identity to the class of
the identity.  Lemma \ref{group_incl} then implies that
\begin{equation*}
 \varphi_*(g)=g.
\end{equation*}
\end{proof}

\section{Morita Equivalence}\noindent
This section follows closely the wonderful book \cite{Raeburn} by Iain Raeburn
and Dana Williams.  The aim is not to prove all the theorems, but rather to put
forth the basics of the theory.  Morita equivalence will only be used in proving
the existence of *-homomorphisms of the form
\begin{equation*}
 \varphi:A_{\Theta}\rightarrow M_n\left(A_{\theta}\right)
\end{equation*}
where $\Theta$ and $\theta$ are both irrational.
\begin{definition}(Right $A$-module)\cite[p. 8]{Raeburn}\label{module}.  \newline
Let $A$ be a $C^*$-algebra.  By a \emph{right $A$-module}, we shall mean a vector space $X$ together with the linear pairing
\begin{equation*}
 (x,a)\mapsto x\cdot a : X\times A\rightarrow X.
\end{equation*}
\end{definition}
We will often write $X_{A}$ to emphasize the fact that we are viewing $X$ as a right $A$-module.
\begin{definition}(Inner Product $A$-Module)\cite[p. 8]{Raeburn}\label{ipmodule}\newline
A right \emph{inner product $A$-module} is a (right) $A$-module $X$ with a pairing
\begin{equation*}
 \langle\cdot,\cdot\rangle_{A}:X\times X\rightarrow A
\end{equation*}
such that the following five properties hold:
\begin{enumerate}
 \item $\langle x,\lambda y +\mu z \rangle_A=\lambda\langle x, y \rangle_A+\mu\langle x, z \rangle_A$
 \item $\langle x,y\cdot a\rangle_A=\langle x,y \rangle_A a$
 \item $\langle x,y\rangle^{*}_A=\langle y,x\rangle_A$
 \item $\langle x,x\rangle_A\geq 0$ (As an element of $A$)
 \item $\langle x,x\rangle_A=0$ implies that $x=0$
\end{enumerate}
for all $x,y,z\in X$, $\lambda,\mu\in \mathbb{C}$ and $a\in A$
\end{definition}
\begin{remark}
 From conditions (1) and (3) we can show that $\langle\cdot,\cdot\rangle_A$ is conjugate linear in the first variable:
\begin{eqnarray*}
\langle \lambda x+\mu y,z\rangle_{A} &=& \langle z,\lambda x+\mu y\rangle^*_A\\
&=&\left(\lambda\langle z,x\rangle_A+\mu\langle z,y\rangle_A\right)^*\\
&=&\overline{\lambda}\langle x,z\rangle_A+\overline{\mu}\langle y,z\rangle_A 
\end{eqnarray*}
Using (2) and (3) we can also show the following:
\begin{eqnarray*}
 \langle x\cdot a,y\rangle_A&=&\langle y,x\cdot a\rangle^*_A\\
&=&\left(\langle y,x\rangle_Aa\right)^*\\
&=&a^*\langle x,y\rangle_A
\end{eqnarray*}
\end{remark}
\begin{lemma}
$I=\text{span}\{\langle x,y\rangle_A\}$ is a two-sided ideal in $A$. 
\end{lemma}
\begin{proof}
 Let $a\in A$ be arbitrar.  Then we know from the above remark that
\begin{equation*}
 a\langle x,y\rangle_A=\langle x\cdot a^*,y\rangle_A.
\end{equation*}
But from Definition \ref{module} we know that $x\cdot a^*\in X$ and hence
$a\langle x,y\rangle_A\in I$.  Similarly from property (2) of Definition
\ref{ipmodule} we know that $\langle x,y\rangle_Aa\in I$.  This implies that $I$
is a two sided ideal in $A$. 
\end{proof}
\begin{remark}\label{leftipmodule}
Later we will also be concerned with left $A$-modules.  They differ from Definition \ref{module} in that we now have a pairing
\begin{equation*}
 A\times X\rightarrow X:(a,x)\mapsto a\cdot x
\end{equation*}
and we will use the notation $_AX$ to emphasize that we are dealing with a left $A$-module.  For left inner product $A$-modules Definition \ref{ipmodule} also has to be modified.  The pairing $_A\langle\cdot,\cdot\rangle:X\times X\rightarrow A$ is defined to be linear in the first variable and conjugate linear in the second variable, furthermore property (2) is replaced by
\begin{equation*}
 _A\langle a\cdot x,y\rangle=a\langle x,y\rangle  
\end{equation*}
The remaining properties remain unchanged.  We define an $A-B$-\textsl{bimodule} to be both a left $A$ and a right $B$ module.
\end{remark}
\begin{remark}
From \cite[Corollary 2.7]{Raeburn} we know that if $X$ is a right Hilbert $A$-module then
\begin{equation*}
 \|\cdot\|_A:=\|\langle\cdot,\cdot\rangle_A\|^{1/2}
\end{equation*}
defines a norm on $X$.  This then leads us to the following definition:
\end{remark}
\begin{definition}(Full, Hilbert $A$-module)\cite[p. 11]{Raeburn}\label{full}\newline
A \emph{Hilbert $A$-module} is an inner product $A$-module $X$ which is complete in the norm induced by the pairing $\langle\cdot,\cdot\rangle_A$ and defined by \cite[Corollary 2.7]{Raeburn}
\begin{equation}
 \|x\|_A:=\|\langle x,x\rangle_A\|^{1/2}
\end{equation}
Furthermore, we say that a Hilbert inner product $A$-module is \emph{full} if the ideal
\begin{equation*}
 I_A=\text{span}\{\langle x,y\rangle_A:x,y\in X\}
\end{equation*}
is dense in $A$.
\end{definition}
\begin{definition}($A-B$-Imprimitivity Bimodule)\cite[p. 42]{Raeburn}\newline
Let $A$ and $B$ be $C^*$-algebras.  Then an \emph{$A-B$-imprimitivity bimodule} is an $A-B$-bimodule such that
\begin{enumerate}
 \item $X$ is a full left Hilbert $A$-module and a full right Hilbert
$B$-module,
 \item for all $x,y\in X,a\in A$ and $b\in B$ we have
\begin{eqnarray*}
 \langle a\cdot x,y\rangle_B&=&\langle x, a^*\cdot y\rangle_B\\
 _A\langle x\cdot b,y\rangle&=&_A\langle x,y\cdot b^*\rangle,
\end{eqnarray*}
 \item for all $x,y,z\in X$, we have
\begin{equation*}
 _A\langle x,y\rangle\cdot z=x\cdot\langle y,z\rangle_B.
\end{equation*}
\end{enumerate}
\end{definition}
\begin{definition}(Morita Equivalence)\cite[proposition 3.16]{Raeburn}\label{Morita}\newline
 Two $C^*$-algebras $A$ and $B$ are \emph{Morita equivalent} if there is an $A-B$-imprimitivity bimodule $X$.
\end{definition}
The next theorem is cited without proof and will be quoted when we show that Morita Equivalence is an equivalence relation on a $C^*$-algebra.
%% I have to include the definitions of the terms I am using in this theorem, see Rocco's notes
\begin{definition}(Pre-inner product)\newline
 A pre-inner product $\langle\cdot,\cdot\rangle$ has all the properties of a
normal inner product without the property that
\begin{equation*}
 \langle x,x\rangle=0 \text{ if and only if } x=0.
\end{equation*}
\end{definition}

\begin{lemma}\cite[Lemma 2.16]{Raeburn}\label{complete01}\newline
 Suppose $A_0$ is a dense *-subalgebra of a $C^*$-algebra $A$, and that $X_0$ is a right $A_0$-module.  We suppose that $X_0$ is a pre-inner product $A_0$-module with pre-inner product $\langle\cdot,\cdot \rangle_0$.  Then there is a Hilbert $A$-module $X$ and a linear map 
\begin{equation*}
 q:X_0\rightarrow X
\end{equation*}
such that $q(X_0)$ is dense, $q(x)\cdot a=q(x\cdot a)$ for all $x\in X_0, a\in A_0$, and $\langle q(x),q(y)\rangle_A=\langle x,y\rangle_0$.  We call $X$ the completion of the pre-inner product module $X_0$.
\end{lemma}

\begin{lemma}\cite[Proposition 3.12]{Raeburn}\label{complete02}\newline
 Let $A$ and $B$ be $C^*$-algebras and $A_0\subset A$ and $B_0\subset B$ dense
*-subalgebras.  If $X_0$ is an $A_0-B_0$-imprimitivity bimodule.  Then there is
an $A-B$-imprimitivity bimodule $X$ and an $A_0-B_0$-imprimitivity bimodule
homomorphism
\begin{equation*}
 q:X_0\rightarrow X
\end{equation*}
such that $q(X_0)$ is dense and
\begin{equation*}
 \langle q(x),q(y)\rangle_B=\langle x,y\rangle_{B_0}, \quad _A\langle q(x),q(y)\rangle=_{A_0}\langle x,y\rangle
\end{equation*} 
for all $x,y\in X,a\in A_0$, and $b\in B$.
\end{lemma}
In the following theorem we mention completion of an imprimitivity bimodule, this can be understood in terms of Lemmas \ref{complete01} and \ref{complete02}.
\begin{theorem}\cite[p. 48]{Raeburn}\label{relation}\newline
 Let $A,B$ and $C$ be $C^*$-algebras.  Suppose that $X$ is an $A-B$-imprimitivity bimodule and $Y$ is a $B-C$-imprimitivity bimodule.  Then $Z:=X\odot_{B}Y$ is an $A-C$-bimodule, and there are unique $A$ and $C$ valued pre-inner products $_A\langle\langle\cdot,\cdot\rangle\rangle$ and $\langle\langle\cdot,\cdot\rangle\rangle_C$ respectively on $Z$ satisfying
\begin{eqnarray*}
 \langle\langle x\otimes_{B}y,z\otimes_{B}w \rangle\rangle_{C}&=&\langle\langle z,x\rangle_{B}\cdot y,w\rangle_{C}\\
_{A}\langle\langle x\otimes_{B}y,z\otimes_{B}w \rangle\rangle&=&_{A}\langle x,z\cdot_{B}\langle w,y\rangle\rangle.
\end{eqnarray*}
With respect to these pre-inner products, $Z$ is an $A-C$-pre-imprimitivity bimodule.  The completion of $Z$ is an $A-C$-imprimitivity bimodule, which we also denote by $X\otimes_{B}Y$ and call the internal tensor product.
\end{theorem}

\begin{proposition}\cite[Proposition 3.18]{Raeburn}\newline
Morita equivalence is an equivalence relation on $C^*$-algebras. 
\end{proposition}
\begin{proof}
 To prove transitivity we use the internal tensor product.  If $_AX_{B}$ and
$_BY_C$ are Morita equivalences, then according to Theorem \ref{relation} we see
that $_A(X\otimes_BY)_C$ implements a Morita equivalence between $A$ and $C$. 
It is routine to show that for any $C^*$-algebra $A$, $_AA_A$ is an
$A-A$-imprimitivity bimodule \cite[Example 3.5]{Raeburn} and so the relation is
reflexive.  To observe the symmetric nature of the relation we introduce the
dual module.  If $X$ is an $A-B$-imprimitivity bimodule, let $\widetilde{X}$ be
the conjugate vector space, so that there is by definition an additive bijection
\begin{equation*}
 \varphi:X\rightarrow\widetilde{X}
\end{equation*}
such that
\begin{equation*}
 \varphi(\lambda\cdot x)=\overline{\lambda}\cdot\varphi(x).
\end{equation*}
Then $\widetilde{X}$ is a $B-A$-imprimitivity bimodule with
\begin{eqnarray*}
  b\cdot\varphi(x)&=&\varphi(x\cdot b^*) \\
  _B\langle\varphi(x),\varphi(y)\rangle&=&\langle x,y\rangle_{B} \\
 \varphi(x)\cdot a&=&\varphi(a^*\cdot x)\\
 \langle\varphi(x),\varphi(y)\rangle_{A}&=&_A\langle x,y\rangle
\end{eqnarray*}
for $x,y\in X$, $a\in A$ and $b\in B$.
\end{proof}\noindent
Now we will look at a particular case which will be of of interest in the next chapter.  Consider the vector space
\begin{equation*}
F(A):=\{(a_1,\dots,a_n):a_i\in A, 1\leq i\leq n\}
\end{equation*} 
consisting of row vectors with entries in the $C^*$-algebra $A$.  
\begin{lemma}\label{left_Amod}
 $F(A)$ is a full left Hilbert $A$-module.
\end{lemma}
\begin{proof}
Consider the pairing $(A,F(A))\rightarrow F(A)$ given by
\be
(c,\overline{a})\mapsto c\cdot \overline{a}=(ca_1,\dots,ca_n) 
\ee
where $\overline{a}\in F(A)$ and $c\in A$.  We define the left inner product by
\begin{equation}\label{leftA}
_A\langle \overline{a},\overline{b}\rangle:=\overline{a}\overline{b}^*
\end{equation}
where $\overline{b}^*$ denotes the conjugate transpose of the row vector $\overline{b}$ and $\overline{a}\overline{b}^*$ denotes the scalar product of the two vectors.  Now we show that equation (\ref{leftA}) satisfies the left version of Definition \ref{ipmodule} (see Remark \ref{leftipmodule}):
\begin{eqnarray*}
_A\langle \lambda\overline{a}+\mu\overline{b},\overline{c}\rangle&=&\left(\lambda\overline{a}+\mu\overline{b}\right)\overline{c}^*\\
&=&\lambda\overline{a}\overline{c}^*+\mu\overline{b}\overline{c}^*\\
&=&\lambda_A\langle\overline{a},\overline{c}\rangle+\mu_A\langle\overline{b},\overline{c}\rangle
\end{eqnarray*}
for $\overline{a},\overline{b}$ and $\overline{c}$ in $F(A)$ and $\lambda, \mu$ in $\mathbb{C}$.  We have linearity in the first variable, so property (1) is satisfied.  Let $\overline{a},\overline{b}\in F(A)$ and $c\in A$, then
\begin{eqnarray*}
_A\langle c\cdot\overline{a},\overline{b}\rangle=\left(c\overline{a}\right)\overline{b}^*
=c\overline{a}\overline{b}^*
=c_A\langle\overline{a},\overline{b}\rangle \\
_A\langle \overline{a},\overline{b}\rangle^*=\left(\overline{a}\overline{b}^*\right)^*=\overline{b}\overline{a}^*=~ _A\langle \overline{b},\overline{a}\rangle\\
_A\langle \overline{a},\overline{a}\rangle=\overline{a} ~ \overline{a}^*=\sum^{n}_{i=1}a_ia^{*}_{i}\geq0
\end{eqnarray*}
But $\sum^{n}_{i=1}a_ia^*_i=0$ if and only if each $a_i=0$ for $1\leq i \leq n$
since each $a_ia_i^*$ is positive for any $a_i\in A$ and the sum of positive
elements of $A$ remains positive.  Hence if $\sum^{n}_{i=1}a_ia^*_i=0$ the
result follows.  So the five properties of Definition \ref{ipmodule} are
satisfied.
Now we consider the ideal $I_A$, as defined in Definition \ref{full} and show
that it is dense in $A$.
\begin{equation*}
I_A=\{_A\langle \overline{a},\overline{b}\rangle=\overline{a}\overline{b}^*=\sum^{n}_{i=1}a_ib^*_i:a_i,b_i\in A\}
\end{equation*}
But the span of elements of the form $a_ib^*_i$ is dense in $A$.  This can
easily be seen for both the unital and non-unital case.  In the former we can
choose $b=\I$ and in the latter we can choose $b=u_{\lambda}$ where
$u_{\lambda}$ is an approximate identity as defined in \cite[p. 77]{murphy}. 
Then we can write any element of $A$ as a linear combination of elements of the
form $ab^*$.  From \cite[p. 73]{Kreyszig} we know that finite dimensional normed
spaces are complete.  We know that $_A\langle\cdot,\cdot\rangle$ induces a norm
on $F(A)$.  Furthermore $F(A)$ is complete in this norm.  This shows that $F(A)$
is a full left Hilbert $A$-module.
\end{proof}
\begin{remark}
 The interested reader can have a look at \cite[Examples 2.10, 2.14]{Raeburn}.  The first shows an example of a right Hilbert $A$-module which is not full and the second deals with Hilbert modules of direct sums.
\end{remark}

\begin{lemma}\label{right_Amod}
$F(A)$ is a full right Hilbert $M_n(A)$-module. 
\end{lemma}
\begin{proof}
 Consider the pairing $(F(A),M_n(A))\rightarrow F(A)$ defined by
\be
(\overline{a},M)\mapsto \overline{a}M
\ee
where $\overline{a}\in F(A)$ and $M\in M_n(A)$ and where the product is defined as in normal matrix operation on vectors.  We define the right inner product by
\begin{equation}\label{rightA}
 \langle\overline{a},\overline{b}\rangle_{M_n(A)}:=\overline{a}^*\overline{b}
\end{equation}
We now show that equation (\ref{rightA}) satisfies the properties of Definition \ref{ipmodule}:
\begin{eqnarray*}
\langle \overline{a},\lambda\overline{b}+\mu\overline{c}\rangle_{M_n(A)}&=&\overline{a}^*\left(\lambda\overline{b}+\mu\overline{c}\right)\\
&=&\lambda\overline{a}^*\overline{b}+\mu\overline{a}^*\overline{c}\\
&=&\lambda\langle\overline{a},\overline{b}\rangle_{M_n(A)}+\mu\langle\overline{b},\overline{c}\rangle_{M_n(A)}
\end{eqnarray*}
We have linearity in the second variable, so property (1) is satisfied. 
Moreover
\begin{eqnarray*}
\langle\overline{a},\overline{b}\cdot M\rangle_{M_n(A)}=\overline{a}^*\overline{b}M
=\langle\overline{a},\overline{b}\rangle_{M_n(A)} M \\
\langle \overline{a},\overline{b}\rangle^*_{M_n(A)}=\left(\overline{a}^*\overline{b}\right)^*=\overline{b}^*\overline{a}=\langle \overline{b},\overline{a}\rangle_{M_n(A)}.
\end{eqnarray*}
We next show that 
\begin{equation*}
 \langle \overline{a},\overline{a}\rangle_{M_n(A)}=\overline{a}^*\overline{a}=
\left(
\begin{array}{lll}
a_1a^*_1 &\dots  &a_1a^*_n \\
\vdots & \ddots & \vdots\\
a_na^*_1 & \dots & a_na^*_n
\end{array}
\right)
\geq 0
\end{equation*}
Let $H^n=\oplus^{n}_{i=1}H$ with $H$ a Hilbert space.  Let $\pi:A\rightarrow H$ be a faithful representation on $H$.  Then we can define the representation
\begin{equation*}
\varphi:M_n(A)\rightarrow H^n:a_{ij}\mapsto \pi(a_{ij}) 
\end{equation*}
where $a_{ij}$ is the entry of the matrix in the $(i,j)$ position.  Since $\pi$
is faithful by assumption this implies that $\varphi$ must also be faithful. 
Furthermore, let $h\in H^n$ be such that 
\begin{equation*}
h=\left(
\begin{array}{l}
  h_1\\
 \vdots\\
 h_n
\end{array}
\right).
\end{equation*}
So if we can show that $\langle h,\varphi(\langle\overline{a},\overline{a}\rangle_{M_n(A)})h\rangle\geq0$ then $\varphi(\langle\overline{a},\overline{a}\rangle_{M_n(A)})$ is positive, in other words we make use of the standard characterization of positive operators on Hilbert spaces, which in turn implies that $\langle\overline{a},\overline{a}\rangle_{M_n(A)}\geq0$.  Consider the following
\begin{eqnarray*}
 \langle h,\varphi(\langle\overline{a},\overline{a}\rangle_{M_n(A)}h)\rangle&=&
\langle h,\left(\begin{array}{l}\varphi(a_1)^*\\ \vdots\\ \varphi(a_n)^* \end{array}\right)\left(\varphi(a_1),\dots,\varphi(a_n)\right)\left(
\begin{array}{l} h_1\\ \vdots\\ h_n \end{array} \right) \rangle\\
&=&
\langle h,\left(\begin{array}{l}\varphi(a_1)^*\\ \vdots\\ \varphi(a_n)^* \end{array}\right)\sum^{n}_{i=1}\varphi(a_i)h_i \rangle\\
&=&
\sum^{n}_{i,j=1}\langle h_j,\varphi(a_j)^*\varphi(a_i)h_i \rangle\\
&=&\sum^{n}_{i,j=1}\langle \varphi(a_j)h_j,\varphi(a_i)h_i \rangle.
\end{eqnarray*}
But we also have
\begin{eqnarray*}
 \|\sum^{n}_{i=1}\varphi(a_i)h_i\|^2&=& \sum^{n}_{i,j=1}\langle\varphi(a_j)h_j,\varphi(a_i)h_i\rangle
\end{eqnarray*}
which implies that we can write
\begin{equation*}
 \langle h,\varphi(\langle\overline{a},\overline{a}\rangle_{M_n(A)})h\rangle=\|\sum^{n}_{i=1}\varphi(a_i)h_i\|^2\geq0.
\end{equation*}
If $\langle\overline{a} ,\overline{a}\rangle_{M_n(A)}=0$ (the $n\times n$ matrix
with only zero entries) then $\overline{a}=\overline{0}$.  Conversely, if
$\overline{a}=\overline{0}$, then clearly
$\overline{a}^*\overline{a}=0=\langle\overline{a}
,\overline{a}\rangle_{M_n(A)}$.  So the five properties of Definition
\ref{ipmodule} are satisfied.  We know that finite dimensional normed spaces are
complete, so the norm induced by $\langle\cdot,\cdot\rangle_{M_n(A)}$ is
complete.  It only remains to verify that $F(A)$ is full.  Consider the ideal
\be
I_{M_{n}(A)}=\left\{\langle\overline{a},\overline{b}\rangle_{M_n(A)}:\overline{a},\overline{b}\in F(A)
\right\}
\ee
Let $b_j$ be an approximate identity \cite[p. 77]{murphy} of the $C^*$-algebra
$A$ and choose two vectors in $F(A)$, $\overline{a}$ and $\overline{b}$ with
entries $a_i$ and $b_j$ at positions $i$ and $j$ respectively and zero's
everywhere else.  Using these vectors we can create matrices with a single entry
of $a_ib^*_j$ at position $(i,j)$ and zero's everywhere.  It is then clear that
the span of such matrices is dense in $M_n(A)$.  This concludes the proof.
\end{proof}
\begin{proposition}\label{Morita_matrix}
 Let $A$ be a $C^*$-algebra.  $A$ is Morita equivalent to $M_n(A)$.
\end{proposition}
\begin{proof}
 This is clear from the above two lemmas.
\end{proof}
The result of Lemma \ref{Morita_matrix} will be used in chapter 4 when we consider homomorphisms from a particular quantum torus to the matrix algebra with entries in another quantum torus.

\section{The Connection Between Morita \\Equivalence and $K_0$-groups}\noindent
The next theorem is of great importance since it establishes the fact that when two $C^*$-algebras are Morita equivalent they will have the same $K_0$-groups which we will use in proving the existence of *-homomorphims between different quantum tori.
\begin{definition}(\emph{Strictly Positive Element})\cite[p. 337]{Brown1977}\newline
A positive element $e$ of a $C^*$-algebra $A$ is called \emph{strictly positive} if $\varphi(e) > 0$ for every state $\varphi$ of $A$.
\end{definition}

\begin{theorem}\label{morita_stably}\cite[Theorem 1.2]{Rieffel1977}\newline
 Let $B$ and $E$ be unital $C^*$-algebras.  $B$ and $E$ are stably isomorphic if and only if then they are Morita equivalent.
\end{theorem}
\begin{proof}
As usual, let $K$ denote the algebra of compact operators on some separable infinite dimensional Hilbert space and let $p$ be a rank-one projection in $K$.  We can construct the mappings
\begin{eqnarray*}
&& \psi:B\otimes p\rightarrow B:(b\otimes p)\mapsto b\\
&& \phi:B\rightarrow B\otimes p:b\mapsto b\otimes p.
\end{eqnarray*}
They are both *-homomorphisms and we clearly have $\psi\phi=$Id and similarly $\phi\psi=$Id.  So $B$ is isomorphic to $B\otimes p$.  Now we define $\mathcal{P}:=\I\otimes p$.  
\begin{eqnarray*}
 \mathcal{P}^2=\left(\I\otimes p\right)\left(\I\otimes p\right)=\I\otimes p\\
 \mathcal{P}^*=\left(\I\otimes p\right)^*=\I\otimes p
\end{eqnarray*}
So $\mathcal{P}$ is a projection in $B\otimes p$, and we find that for any
$b\in B$ and $k\in K$ we have
\begin{eqnarray*}
 \left(\I\otimes p\right)\left(b\otimes k\right)\left(\I\otimes p\right)&=&\left(\I\otimes p\right)\left(b\otimes kp\right)\\
&=&b\otimes pkp\\
&=&b\otimes \varphi(k)p
\end{eqnarray*}
where $\varphi$ is the state on $K$ defined by $\varphi(k)=\langle
\zeta,k\zeta\rangle$ with $\zeta$ a unit vector in the range of $p$.  Then 
$\varphi(k)p=pkp$.  Now since $\varphi(k)p$ is simply a scalar multiple of $p$
it follows that $\mathcal{P}\left(B\otimes_* K\right) \mathcal{P}=B\otimes p$. 
This implies that $B\otimes p$ is a corner of $B\otimes K$.  We will now show
that it is a full corner which in turn implies that $B$ is Morita equivalent to
$B\otimes K$ \cite[p. 50]{Raeburn}.  It is enough to show that
span$\{\left(B\otimes K\right)\mathcal{P}\left(B\otimes K\right)\}$ is dense in
$B\otimes K$.  Let $b_i,b_j\in B$ and $k_l,k_m\in K$.  Then
\begin{eqnarray*}
 \left(b_i\otimes k_l\right)\mathcal{P}\left(b_j\otimes k_m\right)&=&\left(b_i\otimes k_l\right)\left(b_j\otimes pk_m\right)\\
&=&b_ib_j\otimes k_lpk_m.
\end{eqnarray*}
But the span of elements of the form $b_ib_j$ is dense in $B$ and similarly the span of elements of the form $k_ipk_j$ is dense in $K$.  So we can conclude that span$\{\left(B\otimes K\right)\mathcal{P}\left(B\otimes K\right)\}$ is dense in $B\otimes K$ and $B$ is Morita equivalent to $B\otimes K$.  Similarly we can also show that $E$ is Morita equivalent to $E\otimes K$.  Hence if $B\otimes K\cong E\otimes K$ then $B$ and $E$ will be Morita equivalent.  

For the converse, see \cite[Theorem 5.55]{Raeburn}
\end{proof}
By Proposition \ref{isoK} we know that when two $C^*$-algebras are stably isomorphic they will have isomorphic $K_0$-groups.
\chapter{Existence Theorems}\label{chapter_5}\noindent
\begin{center}
 \begin{quotation}
  \textsl{``Can the existence of a mathematical entity be proved without defining it?''\newline-Jacques Salomon Hadamard (1865-1963)}
 \end{quotation}
\begin{quotation}
 \textsl{``Nothing exists except atoms and empty space; everything else is opinion.''\newline-Democritus (460-370 B. C)}
\end{quotation}
\end{center}
The theory developed in the second chapter made use of the existence of certain
maps between the different quantum tori.  Up until now we have only been
concerned about the $M_n(\mathbb{C})$ case and in this chapter we extend the
theory to the quantum torus.  In this chapter we state theorems and give proofs
of the existence of *-homomorphisms between different quantum tori.  We also
investigate under which circumstances these *-homomorphisms become
*-isomorphisms.  These existence proofs are of fundamental importance since
without the existence of such *-homomorphisms we cannot create a noncommutative
$\sigma$-model.  This section forms the most technical part of this thesis and
combines results from K-theory and Morita equivalence to prove the existence of
the *-homomorphisms.
\section{The Action of GL$(2,\mathbb{Z})$ on Irrationals}\noindent
Let $GL\left(2,\mathbb{Z}\right)$ be the group of all $2\times 2$ matrices with integer entries and determinant $\pm 1$.  The action of an element of GL$(2,\mathbb{Z})$ on an irrational number $\theta$ is defined by 
\begin{equation}\label{GLaction}
 g\theta=
\left(
\begin{array}{ll}
a & b\\
c & d
\end{array}
\right)\theta:=\frac{a\theta+b}{c\theta +d}
\end{equation}
\begin{lemma}
If $g\in$ GL$(2,\mathbb{Z})$ and $\theta$ is irrational, then $g\theta$ is a group action of GL$\left(2,\mathbb{Z}\right)$ on $\mathbb{R}/\mathbb{Q}$.
\end{lemma}
\begin{proof}
 Suppose $g$ is of the same form as in equation (\ref{GLaction}), furthermore suppose that
\begin{equation*}
 \frac{a\theta+b}{c\theta +d}=\frac{x}{y}
\end{equation*}
for some integers $x$ and $y$.  Then we can write
\begin{eqnarray*}
 \theta\left(ay-cx\right)&=&dx-by\\
\theta&=&\frac{dx-by}{ay-cx}.
\end{eqnarray*}
for $ay-cx\neq0$.  So then $\theta$ is rational and we have a contradiction.  When $ay-cx=0$ it implies that $dx-by=0$ and we can write
\begin{equation*}
 \frac{x}{y}=\frac{a}{c}=\frac{b}{d}.
\end{equation*}
Hence $ad-bc=0$.  But this is a contradiction since
\begin{equation*}
 \left|
\begin{array}{cc}
 a & b\\
 c & d
\end{array}
\right|=\pm 1.
\end{equation*}
Hence for any irrational $\theta$ we find that $g\theta$ is also irrational.

Let $k=\left(
\begin{array}{ll}
a & b\\
c & d
\end{array}
\right)$ and $l=\left(
\begin{array}{ll}
e & f\\
g & h
\end{array}
\right)$ and consider the calculations
\begin{eqnarray*}
\left(
\begin{array}{ll}
e & f\\
g & h
\end{array}
\right)\left(
 \left(
\begin{array}{ll}
a & b\\
c & d
\end{array}
\right)\theta\right)&=&\left(
\begin{array}{ll}
e & f\\
g & h
\end{array}
\right)\frac{a\theta+b}{c\theta+d}\\
&=&\frac{(ae+cf)\theta+(eb+df)}{(ga+ch)\theta+(gb+dh)}.
\end{eqnarray*}
And similarly we also calculate
\begin{eqnarray*}
\left(
\left(
\begin{array}{ll}
e & f\\
g & h
\end{array}
\right)
 \left(
\begin{array}{ll}
a & b\\
c & d
\end{array}
\right)\right)\theta&=&
\left(
\begin{array}{ll}
ae+fc & eb+fd \\
ga+hc & gb+hd
\end{array}
\right)\theta\\
&=&\frac{(ae+cf)\theta+(eb+df)}{(ga+ch)\theta+(gb+dh)}.
\end{eqnarray*}
Hence $k\left(l\theta\right)=\left(kl\right)\theta$.  This concludes the proof.
\end{proof}
These results will be used when we consider homomorphisms of the form 
\begin{equation*}
 \varphi:A_{\Theta}\rightarrow M_{n}(A_{\theta}).
\end{equation*}

\begin{definition}(Orbit)\newline
 Let $G$ be a group acting on a set $X$.  The \textsl{orbit} of a point $x$ in $X$ is the set of elements of $X$ to which $x$ can be moved by the elements of $G$.  The \text{orbit} of $x$ is denoted by $Gx$ where
\begin{equation*}
 Gx=\{g x: g\in G\}.
\end{equation*}
\end{definition}

It is a familiar result that the matrices 
\begin{equation}\label{genGL}
\left(
\begin{array}{ll}
0&-1\\
1&0
\end{array}
\right)
, \quad
\left(
\begin{array}{ll}
 1&1\\
0&1
\end{array}
\right)
\end{equation}
generate GL$(2,\mathbb{Z})$ \cite[Appendix B]{kurosh}.  For $\theta$ irrational we have
\begin{eqnarray*}
 \left(
\begin{array}{ll}
0&-1\\
1&0
\end{array}
\right)\theta&=&-\frac{1}{\theta}\\
\left(
\begin{array}{ll}
 1&1\\
0&1
\end{array}
\right)\theta&=&\theta +1.
\end{eqnarray*}

\begin{proposition}
 If $\alpha$ and $\beta$ are irrational numbers which are in the same orbit of the action of GL$(2,\mathbb{Z})$, then the quantum tori $A_{\alpha}$ and $A_{\beta}$ are Morita equivalent.
\end{proposition}
\begin{proof}
If two irrationals $\alpha$ and $\beta$ are in the same orbit of the action of GL$(2,\mathbb{Z})$ then we can write
\begin{equation*}
\beta=g\alpha
\end{equation*}
for some $g\in$GL$\left(2,\mathbb{Z}\right)$.  Now since GL$(2,\mathbb{Z})$ is generated by the matrices in equation (\ref{genGL}) we only need to consider the case where 
\begin{equation*}
g=\left(
\begin{array}{ll}
0&-1\\
1&0
\end{array}
\right)
\quad \text{and} \quad h=\left(
\begin{array}{ll}
 1&1\\
0&1
\end{array}
\right).
\end{equation*}
So we have two cases
\begin{equation*}
 \beta=\left(
\begin{array}{ll}
0&-1\\
1&0
\end{array}
\right)\alpha=-\frac{1}{\alpha}
\end{equation*}
and
\begin{equation*}
 \beta=\left(
\begin{array}{ll}
 1&1\\
0&1
\end{array}
\right)\alpha=\alpha+1.
\end{equation*}
From \cite[p. 420-1]{Rieffel1981} we know that $A_{\alpha}$ is Morita equivalent to $A_{\alpha^{-1}}$.  By Lemma \ref{A_iso} we know that $A_{\alpha}$ is Morita equivalent to $A_{\alpha+1}$.  Hence, $A_{\alpha}$ is Morita equivalent to $A_{\beta}$.
\end{proof}

\section{Classification of the Homomorphisms}\noindent
Now we are ready to look at the homomorphisms between quantum tori and also under which circumstances these homomorphisms are in fact isomorphisms.  The following subsection is the first stepping stone in the theory.
\subsection{Unital *-Homomorphism Between Noncommutative Tori}
The first case we consider is the simplest and most natural.  Here we will show under what circumstances unital *-homomorphisms exist between different noncommutative tori.  This case has a very clear physical interpretation from the point of view of string theory as explained in the introduction and is the most important result of this chapter.

\begin{theorem}\cite[Theorem 2.1]{MR}\label{stelling2.1}\newline
Fix $\Theta$ and $\theta$ in $(0,1)$, both irrational.  There exists an injective unital *-homomorphism $\varphi:A_{\Theta}\rightarrow A_{\theta}$ if and only if $\Theta=c\theta+d$ for some $c,d\in \mathbb{Z}, c\neq0$.  Such a *-homomorphism $\varphi$ can be chosen to be an isomorphism onto its image if and only if $c=\pm1$.
\end{theorem}
\begin{proof}
 Suppose there exists a unital *-homomorphism
\begin{equation*}
 \varphi:A_{\Theta}\rightarrow A_{\theta}.
\end{equation*}
 By Lemma \ref{trace_iso}, $K_0(A_{\Theta})$ is mapped isomorphically to
$\mathbb{Z}+\Theta\mathbb{Z}$.  From Lemma \ref{proj} we know the range of the
trace on projections from the quantum torus to itself is
$\left(\mathbb{Z}+\Theta\mathbb{Z}\right)\bigcap[0,1]$ and Lemma
\ref{group_homo} shows that $\varphi$ induces an order preserving map
$\varphi_*$ of the $K_0$-groups sending the class of the identity to the class
of the identity.

 Remark \ref{id_class} shows us that we can associate the number $1\in\mathbb{Z}+\Theta\mathbb{Z}$ with the $1\times 1$ identity matrix $\I^{(1)}_{A_{\Theta}}$.  By Proposition \ref{group_homo} we know that we can view the map $\varphi_*$ as the inclusion of the subgroup $K_0(A_{\Theta})$ into $K_0(A_{\theta})$ with 1 going to 1.  So we can now view $\varphi_*$ as an inclusion and write
\begin{equation}
 \varphi_*:\mathbb{Z}+\Theta\mathbb{Z}\rightarrow\mathbb{Z}+\theta\mathbb{Z}.
\end{equation}
So if we regard $\Theta$ as being the generator of the group $K_0(A_{\Theta})\cong\mathbb{Z}+\Theta\mathbb{Z}$ we see by the inclusion that we can write 
\begin{equation*}
 \Theta=c\theta+d
\end{equation*}
for some $c,d\in\mathbb{Z}$.

Conversely by Lemma \ref{A_iso} we see that $A_{c\theta+d}\cong A_{c\theta}$, so we can say that $A_{c\theta+d}$ is generated by the unitaries $U$ and $V$ satisfying the commutation relation
\begin{equation*}
 UV=e^{2\pi ic\theta}VU.
\end{equation*}
While $A_{\theta}$ is generated by unitaries $\tilde{U}$ and $\tilde{V}$ satisfying 
\begin{equation*}
 \tilde{U}\tilde{V}=e^{2\pi i\theta}\tilde{V}\tilde{U}
\end{equation*}
Now we define the homomorphism $\varphi$ to act on the generating unitaries in the following way:
\begin{eqnarray*}
 \varphi(U)&=&\tilde{U}^c\\
 \varphi(V)&=&\tilde{V}
\end{eqnarray*}
Since we have
\begin{equation*}
 \varphi(\I_{A_{c\theta}})=\varphi(V^*V)=\tilde{V}^*\tilde{V}^*=\I_{A_{\theta}}
\end{equation*}
and
\begin{eqnarray*}
 \varphi(U^*)&=&\left(\tilde{U}^*\right)^c=\left(\tilde{U}^c\right)^*\\
 \varphi(V^*)&=&\tilde{V}^*
\end{eqnarray*}
we see that $\varphi$ is indeed a unital *-homomorphism.  

Furthermore we can calculate
\begin{eqnarray*}
 \varphi(UV)&=&\varphi(e^{2\pi i c\theta}VU)\\
&=&e^{2\pi i c\theta}\varphi(V)\varphi(U)\\
&=&e^{2\pi i c\theta}\tilde{V}^c\tilde{U}
\end{eqnarray*}
Here two special cases arise naturally for us to evaluate.  They are
respectively when $c=+1,-1$.  The equation above reduces to
\begin{eqnarray*}
 \varphi(U)\varphi(V)&=&e^{2\pi i\theta}\varphi(V)\varphi(U)\\
 \varphi(V)\varphi(U)&=&e^{2\pi i\theta}\varphi(U)\varphi(V)
\end{eqnarray*}
for the respective cases.  From the universal property of the irrational
rotation algebra we can see that in the case where $c=\pm1$ the images of $U$
and $V$ generate the quantum torus $A_{\theta}$ and the *-homomorphism is
bijective. 

We know that if $\varphi$ is surjective then $\varphi_*$ will also be
surjective.  Consider $c\neq\pm1$.  Since $\varphi_*$ is the inclusion map we
have
\begin{eqnarray*}
 \varphi_*\left(\mathbb{Z}+\Theta\mathbb{Z}\right)&=&\mathbb{Z}+\Theta\mathbb{Z}\\
&=&\mathbb{Z}+\left(c\theta +d\right)\mathbb{Z}\\
&=&\mathbb{Z}+c\theta\mathbb{Z}
\end{eqnarray*}
which is strictly included in $\mathbb{Z}+\theta\mathbb{Z}$.  So if $|c|\neq1$ then $\varphi$ cannot be surjective.  So by a contrapositive argument we can see when $\varphi_*$ is not surjective it implies that $\varphi$ cannot be surjective.  We conclude that indeed the constructed homomorphism can be chosen to be an isomorphism if and only if $c=\pm 1$.
\end{proof}

\subsection{Unital *-Homomorphisms to Matrix Algebras over $A_{\theta}$}
We would like to generalize the results of Theorem \ref{stelling2.1} to the case
where the target space becomes the set of all $n\times n$ matrices with entries
in $A_{\theta}$.  This case has a more distorted physical interpretation in the
sense that in the classical limit we can think of the parameter space as being
replaced by $\mathbb{T}^2\times \{1,\cdots, n\}$ since
$M_n\left(A_{\theta}\right)\cong A_{\theta}\otimes M_n(\mathbb{C})$.  However,
the application to string theory is not clear.  These results are of a more
mathematical nature with applications in the general theory of $\sigma$-models
and therefore from the point of view of physics does not form such a cardinal
part of this thesis.

We are interested in unital *-homomorphisms of the form
\begin{equation}
 \varphi:A_{\Theta}\rightarrow M_n\left(A_{\theta}\right).
\end{equation}

\begin{remark}(Reducing the nonunital to the unital case)\newline
Let $A$ be a $C^*$-algebra and consider the nonunital *-homomorphism
\begin{equation*}
 \varphi:A_{\Theta}\rightarrow M_n(A_{\theta}).
\end{equation*}
Let $\I_{A_{\Theta}}$ be the unit of $A_{\Theta}$ then we have
\begin{equation*}
 \varphi(\I_{A_{\Theta}})=\varphi(\I_{A_{\Theta}}\I^*_{A_{\Theta}})=\varphi(\I^2_{A_{\Theta}})
\end{equation*}
or equivalently
\begin{equation*}
 \varphi(\I_{A_{\Theta}})=\varphi(\I_{A_{\Theta}})^*=\varphi(\I_{A_{\Theta}})^2.
\end{equation*}
So we may write
\begin{equation*}
 \varphi(\I_{A_{\Theta}})=p
\end{equation*}
where $p$ is some projection in $M_n\left(A_{\theta}\right)$.  Clearly $\varphi\left(A_{\Theta}\right)\subset pM_n\left(A_{\theta}\right)p$ and it can be shown that $pM_n(A_{\theta})p$ is a unital $C^*$-algebra for any projection $p$ in $M_n(A_{\theta})$.  From \cite[Corollary 2.6]{Rieffel1983} we see that since $pM_n(A_{\theta})p$ is unital and Morita equivalent to $A_{\theta}$ we may write
\begin{equation*}
 pM_n(A_{\theta})p\cong M_k(A_{\theta})
\end{equation*}
for some integer $k$.  Then putting everything together we find that the *-homomorphism $\varphi$ reduces to
\begin{equation*}
\varphi:A_{\Theta}\rightarrow M_k(A_{\theta}) 
\end{equation*}
and is unital.  So we are then once again reduced to the unital case.
\end{remark}

\begin{lemma}\label{mor}
Given that $\Theta=\frac{c\theta+d}{n}$ with $\theta$ irrational, we can embed
$A_{\Theta}$ into a matrix algebra over $A_{\theta}$. 
\end{lemma}
\begin{proof}
 By \cite[p. 421]{Rieffel1981} we know that $A_{\frac{c\theta+d}{n}}$ is Morita
equivalent to $A_{\frac{n}{c\theta+d}}$.  According to Theorem \ref{stelling2.1}
we can embed $A_{\frac{n}{c\theta+d}}$ unitally into $A_{\frac{1}{c\theta+d}}$
if and only if $\frac{n}{c\theta+d}=x\frac{1}{c\theta+d}+y$ for some integers
$x$ and $y$ such that $x\neq 0$.  By setting $x=n$ and $y=0$ the conditions of
Theorem \ref{stelling2.1} are satisfied and we embed $A_{\frac{n}{c\theta+d}}$
unitally into $A_{\frac{1}{c\theta+d}}$, which in turn is Morita equivalent to
$A_{c\theta+d}$.  By Lemma \ref{A_iso} we can see that $A_{c\theta+d}\cong
A_{c\theta}$, which once again according to Theorem \ref{stelling2.1} can be
embedded unitally into $A_{\theta}$.  If we denote Morita equivalence by $\sim$
and let $\cong$ denote that two $C^*$-algebras are *-isomorphic, then
schematically we can write the procedure as follows:
\begin{equation*}
 A_{\frac{c\theta+d}{n}}  \sim  A_{\frac{n}{c\theta+d}}  \hookrightarrow   A_{\frac{1}{c\theta+d}} \sim  A_{c\theta+d}  \cong  A_{c\theta}  \hookrightarrow  A_{\theta}.
\end{equation*}
Since all the $C^*$-algebras we are dealing with above are unital we can embed  $A_{\frac{c\theta+d}{n}}$ into a full corner of the algebra of $k\times k$ matrices over $A_{\frac{n}{c\theta+d}}$ \cite[Proposition 2.1]{Rieffel1981}.  $M_k(A_{\frac{n}{c\theta+d}})$ embeds unitally into $M_k(A_{\frac{1}{c\theta+d}})$ which in turn embeds into a full corner of the $l\times l$ matrices over $A_{c\theta+d}$.  $M_l(A_{c\theta+d})$ is isomorphic to $M_l(A_{c\theta})$ which finally embeds unitally into $M_l(A_{\theta})$.  Finally we have embedded $A_{\Theta}$ into $M_l(A_{\theta})$.
\end{proof}

\begin{remark}\label{mor_mul}
In the theorem that follows we will be investigating unital homomorphisms of the form
\begin{equation*}
 \varphi:A_{\Theta}\rightarrow M_n(A_{\theta})
\end{equation*}
where we will find that
\begin{equation*}
 \Theta=\frac{c\theta+d}{n}.
\end{equation*}
With $c,d$ and $n$ integers.  We want to investigate the respective K-theories involved.  For that purpose we give a schematic representation of the argument.
\begin{equation*}
\begin{diagram}
A_{\frac{c\theta+d}{n}}    & \rTo^{\eta_1} & M_k\left(A_{\frac{n}{c\theta+d}}\right)& \rTo^{\eta_2} & M_k\left(A_{\frac{1}{c\theta+d}}\right)\\
                           &               &                                        &               & \dTo_{\eta_3} \\
M_l\left(A_{\theta}\right) & \lTo{\eta_4}  & M_l\left(A_{c\theta}\right)            & \cong         & M_l\left(A_{c\theta+d}\right)   
\end{diagram}
\end{equation*}
where the $\eta_j$ mappings implement the Morita equivalence by embedding into full corners of the corresponding matrix algebras.  If we define
\begin{eqnarray*}
 \varphi&:&A_{\Theta}\rightarrow M_l(A_{\theta})
\end{eqnarray*}
to be the composition map 
\begin{equation*}
 \varphi:=\eta_1\circ\eta_2\circ\eta_3\circ\eta_4
\end{equation*}
then, on the level of the $K_0$-groups, we can write
\begin{equation*}
 \varphi_{*}=\eta_{1*}\circ\eta_{2*}\circ\eta_{3*}\circ\eta_{4*}
\end{equation*}
and here $\varphi_*$ is just multiplication by $n$. % This can be seen by
%\begin{equation*}
% \frac{n}{c\theta+d}\times\left(c\theta+d\right)=n
%\end{equation*}
%as it shown in Lemma \ref{mor}.
\end{remark}

\begin{lemma}\label{mult}
 Let $\theta$ be some irrational number.  On the level of the $K_0$-groups $K_0(A_{\frac{\theta}{n}})$ and $K_0(A_{\frac{n}{\theta}})$, the Morita equivalence of $A_{\frac{\theta}{n}}$ and $A_{\frac{n}{\theta}}$ is implemented through multiplication by $\frac{n}{\theta}$.
\end{lemma}
\begin{proof}
 We know that when two quantum tori are Morita equivalent their $K_0$-groups are
isomorphic and these groups are induced by the unique traces on the respective
quantum tori.  Since the $C^*$-algebras $A_{\frac{\theta}{n}}$ and
$A_{\frac{n}{\theta}}$ are unital and Morita equivalent, say with equivalence
bimodule $X$, then the canonical trace $\tau$ on $A_{\frac{\theta}{n}}$ can be
induced by $X$ to give a trace $\hat{\tau}$ on $A_{\frac{n}{\theta}}$
\cite[Proposition 2.2]{Rieffel1981}.  If we denote the linking algebra of the
equivalence bimodule $X$ by $A$ then the trace on $A_{\frac{\theta}{n}}$ has a
unique extension to a trace on $A$ and the restriction of this trace to
$A_{\frac{n}{\theta}}$ is $\hat{\tau}$ \cite[Proposition 2.3]{Rieffel1981} which
in general is a non normalized trace.  Rieffel shows in \cite[Corollary
2.6]{Rieffel1981} that the ranges of $\tau_*$ and $\hat{\tau}_*$ are the same,
in other words, on the $K_0$-groups we have
\begin{eqnarray*}
 \tau_*\left(K_0(A_{\frac{\theta}{n}})\right)&=&\hat{\tau}_*\left(K_0(A_{\frac{n}{\theta}})\right)\\
 \mathbb{Z}+\frac{\theta}{n}\mathbb{Z}&=&\hat{\tau}_*\left(K_0(A_{\frac{n}{\theta}})\right) 
\end{eqnarray*}
Since $\hat{\tau}$ is the non-normalized trace on $A_{\frac{n}{\theta}}$ it will only differ by a scalar multiple from the normalized trace.  We can then write
\begin{equation}
 \mathbb{Z}+\frac{\theta}{n}\mathbb{Z}=\frac{1}{r}\left(\mathbb{Z}+\frac{n}{\theta}\mathbb{Z}\right)\label{equality}
\end{equation}
where $r$ is the proportionality factor.  The only way equality can be realized in (\ref{equality}) is if the Morita equivalence is associated with multiplication by $r=\frac{n}{\theta}$ since
\begin{equation*}
 \frac{n}{\theta}\left(\mathbb{Z}+\frac{\theta}{n}\mathbb{Z}\right)=\mathbb{Z}+\frac{n}{\theta}\mathbb{Z}.
\end{equation*}
This concludes the proof.
\end{proof}

\begin{theorem}\cite[Theorem 2.3]{MR}\label{stelling2.3}\newline
Fix $\Theta$ and $\theta$ in $(0,1)$ both irrational and $n\in \mathbb{N},n\geq1$.  There is a unital *-homomorphism
\begin{equation*}
 \varphi:A_{\Theta}\rightarrow M_n(A_{\theta})
\end{equation*}
 if and only if $n\Theta=c\theta +d$ for some $c,d\in\mathbb{Z}$ and $c\neq0$.  Such a *-homomorphism can be chosen to be an isomorphism onto its image if and only if $n=1$ and $c=\pm1$.
\end{theorem}
\begin{proof}
 Suppose there is a unital *-homomorphism
\begin{equation*}
 \varphi:A_{\Theta}\rightarrow M_n(A_{\theta}).
\end{equation*}
 By Lemma \ref{Morita_matrix} we know that $M_n(A_{\theta})$ is Morita equivalent to $A_{\theta}$.  Theorem \ref{morita_stably} shows that two unital $C^*$-algebras are Morita equivalent if and only if they are stably isomorphic and by Proposition \ref{isoK} we know that 
\begin{equation*}
K_0(M_n(A_{\theta}))\cong K_0(A_{\theta})\cong \mathbb{Z}+\theta\mathbb{Z}
\end{equation*}
Proposition \ref{group_homo} shows that the order on $K_0(M_n(A_{\theta}))$ is the same as on $\mathbb{Z}+\theta\mathbb{Z}$.  Let $\I^{(n)}_{A_{\Theta}}$ be the $n\times n$ matrix with the identity of $A_{\Theta}$ along the diagonal and zeros elsewhere.
\begin{eqnarray*}
 \psi([\I^{(n)}_{A_{\Theta}}])&=&\tau(\I^{(n)}_{A_{\Theta}})\\
&=&\sum^n_{i=1}\tau(\I_{A_{\Theta}})\\
&=&n.
\end{eqnarray*}
So the class of the identity is represented by the number $n$ when we are dealing with $K_0(M_n(A_{\theta}))$.  By Lemma \ref{group} we can regard $\varphi_*$ as multiplication by $n$.  If we again regard $\Theta$ as the generator of the group $K_0(A_{\Theta})$ and apply the map $\varphi_*$ to $\Theta$ we find that
\begin{equation*}
 n\Theta=c\theta+d
\end{equation*}
for some $c,d\in\mathbb{Z}$.  By Lemma \ref{mor} we know that we can embed $A_{\Theta}$ into a matrix algebra over $A_{\theta}$.  In the end we have a non-zero *-homomorphism
\begin{equation*}
 \varphi:A_{\Theta}\rightarrow M_{l}(A_{\theta}).
\end{equation*}
This *-homomorphism will in general not be unital and $l\neq n$.  By Lemma \ref{mult} and Remark \ref{mor_mul} we see that the Morita equivalence is implemented by multiplication by $n$ on the level of the $K_0$-groups.  

By \cite[Corollary 2.5]{Rieffel1983} we know that if we have two projections, say $p$ and $q$ in $M_n(A_{\theta})$ with $\theta$ irrational that have the same trace then they are unitarily equivalent which means we can write
\begin{equation*}
 p=\mathcal{U}q\mathcal{U}^*
\end{equation*}
for some unitary matrix $\mathcal{U}$ in $M_n(A_{\theta})$.  But we know that projections in $M_n(A_{\theta})$ that have the same trace will be mapped to the same element in $K_0(A_{\theta})$.  Hence, projections are determined up to unitary equivalence by their classes in $K_0(A_{\theta})$ which implies that we can conjugate by a unitary and arrange for $\varphi$ to map $A_{\Theta}$ unitally to $M_n(A_{\theta})$.  Finally in \cite[Theorem 3]{Rieffel1981} we see that if we have two matrix algebras $M_m(A_{\alpha})$ and $M_n(A_{\beta})$ with $\alpha$ and $\beta$ irrational then the only way $M_m(A_{\alpha})$ can be isomorphic to $M_n(A_{\beta})$ is when $n=m$ and $\alpha =\beta$.  So we conclude that $A_{\Theta}$ can only be isomorphic to $M_n(A_{\theta})$ when $n=1$.
\end{proof}

\chapter{Outlook}\noindent
\vspace{-2cm}
\begin{center}
 \begin{quotation}
  \textsl{``Every end is a new beginning.''}
 \end{quotation}
\end{center}
\vspace{2cm}
In constructing our noncommutative $\sigma$-model we used one of the most basic noncommutative spaces, the noncommutative torus.  For more general theories we would like to consider noncommutative $\sigma$-models on other noncommutative spaces.  One of the first generalizations to look at might be in extending the results to higher dimensional quantum tori as defined in section \ref{highT}.  We could also consider *-homomorphisms between quantum tori of different ``dimension'' to build a more general theory of noncommutative $\sigma$-models.

Recent consideration in T-duality suggest that we should not just consider spacetimes which are noncommutative tori, but bundles of noncommutative tori over some base space, such as the $C^*$-algebra of the discrete Heisenberg group \cite{Lowe2003,MR2006}.  A theory of noncommutative principle torus bundles has been developed by \cite{Echterhof2009}.  The following definition sufficiently describes such torus bundles:
\begin{definition}(\textsl{Noncommutative Torus Bundle Algebra})\newline
 Let $Z$ be a compact space and let $\Theta: Z \rightarrow \mathbb{T}$ be a continuous function from $Z$ to the circle group. We define the noncommutative torus bundle algebra associated to $(Z,\Theta)$ to be the universal $C^*$-algebra $A = A(Z,\Theta)$ generated over a central copy of $C(Z)$ (continuous functions vanishing on the base space, $Z$) by two unitaries $u$ and $v$, which can be thought of as continuous functions from $Z$ to the unitaries on a fixed Hilbert space $H$, satisfying the commutation rule
\begin{equation*}
 u(z)v(z) = \Theta(z)v(z)u(z).
\end{equation*}
\end{definition}

The theory we have developed in this thesis can be seen as a noncommutative free field theory.  To include interactions we have to modify the Polyakov action with another well known correction called the Wess-Zumino term.  Lots of work has been done on noncommutative Wess-Zumino theory also known as Wess-Zumino-Witten theory, however as of yet there is no established theory where parameter space and world-time are replaced by noncommutative $C^*$-algebras.  Vargese Mathai and Jonathan Rosenberg have started to develop the theory \cite{MR}.

In their paper Vargese Mathai and Johnathan Rosenberg explore a physical model using the theory that this thesis is based on.  They determine the partition function
\begin{equation*}
 Z=\frac{\int d[\varphi]e^{-iS(\varphi)}}{\int d[\varphi]}.
\end{equation*}
However this integral is much too difficult to evaluate even in the commutative
case.  To gain some ground they over-simplify by considering a semi-classical
approximation.  The partition function is then approximated by a sum of the form
\begin{equation*}
 Z\approx \sum e^{-iS(\varphi)}.
\end{equation*}

The study of noncommutative path integrals is still in its infancy with the first steps only recently taking place to give a mathematically rigorous description of these objects which are ``taken for granted'' in quantum field theory and string theory.  We conclude this thesis by stating that there is great room for improving the mathematical foundations of quantum field theory and string theory.  Such endeavors can be seen as a course for future study.
\vspace{1cm}
\begin{center}
\huge{- The End -}
\end{center}

\appendix

\chapter{Regarding Exact Sequences}\label{A}
Exact and short exact sequences play an essential role in the K-theory of $C^*$-algebras and was used in connection with K-theory in Chapter 3.
\begin{definition}(Exact Sequence of $C^*$-algebras)\index{Exact sequence}\newline
 Let $J,A$ and $B$ be $C^*$-algebras.  Suppose that $j:J\rightarrow A$ and $\pi:A\rightarrow B$ are *-homomorphisms such that Im$(j)=\text{ker}(\pi)$ then the sequence
\begin{equation*}
\begin{diagram}
 J& \rTo^{j} &A&\rTo^{\pi} &B
\end{diagram}
\end{equation*}
is \emph{exact}.
\end{definition}
\begin{definition}(Short Exact Sequence of $C^*$-algebras)\cite[p. 211]{murphy}\label{shortexact}\newline
 Let $J,A$ and $B$ be $C^*$-algebras.  Suppose that $j:J\rightarrow A$ is an injective *-homomorphism and that $\pi:A\rightarrow B$ is a surjective *-homomorphism, and that Im$(j)=\text{ker}(\pi)$.  Then 
\begin{equation*}
\begin{diagram}
0&\rTo& J &\rTo^{j}& A & \rTo^{\pi} &B &\rTo &0 & 
\end{diagram}
\end{equation*}
is a \emph{short exact sequence} of $C^*$-algebras.
\end{definition}
Most of the time we will be working with a surjective *-homomorphism $\pi:A\rightarrow B$ of $C^*$-algebras.  We define $J:=\text{ker}(\pi)$ and let $j:J\rightarrow A$ be the inclusion map.  This clearly satisfies the conditions mentioned in definition \ref{shortexact}.
\begin{definition}(Split Short Exact Sequence of $C^*$-algebras)\newline
 The short exact sequence 
\begin{equation}\label{split}
\begin{diagram}
0&\rTo& J&\rTo^{j}&A&\rTo^{\pi}&B&\rTo &0 
\end{diagram}
\end{equation}
is said to \emph{split} if there is a *-homomorphism $\psi:B\rightarrow A$ such that $\pi\psi=\I_{B}$.  Where $\I_{B}$ denotes the identity in $B$.
\end{definition}
We sometimes write equation \ref{split} as
\begin{equation*}
\begin{diagram}
0 & \rTo & J & \rTo^{j} & A & \rTo^{\pi}_{\psi} & B & \rTo & 0 
\end{diagram}
\end{equation*}
to emphasize that we are dealing with a split short exact sequence.  The next lemma is important in determining the K-theory of the quantum torus:
\begin{lemma}\label{split_short_exact}
 If
\begin{equation*}
\begin{diagram}
0& \rTo &J& \rTo^{j} &A& \rTo^{\pi}_{\psi} &B& \rTo &0 
\end{diagram}
\end{equation*}
defines a split short exact sequence of $C^*$-algebras then $A$ is isomorphic to $J\oplus B$.
\end{lemma}
\begin{proof}
 Let $a\in A$ then
\begin{eqnarray*}
 \pi(a-(\psi\circ\pi)(a))&=&\pi(a)-\pi((\psi\circ\pi)(a))\\
&=&\pi(a)-(\pi\psi)(\pi(a))\\
&=&\pi(a)-\pi(a)\\
&=&0
\end{eqnarray*}
This implies that $a-(\psi\circ\pi)(a)\in$ker$(\pi)$.  So then we can write
\begin{equation*}
 j(x)=a-(\psi\circ\pi)(a)
\end{equation*}
for some $x\in J$, since Im$(j)=$ker$(\pi)$.  This implies that we can write
\begin{equation*}
 A=j(J)+\psi(B)
\end{equation*}
and we can show that this is in fact a direct sum.  Let $a\in j(A)\bigcap\psi(B)$, so we can write $a=j(x)=\psi(b)$ for some $x\in J$ and $b\in B$.
\begin{equation*}
 b=(\pi\circ\psi)(b)=(\pi\circ j)(x)=\pi(a)=0
\end{equation*}
We know that $\psi$ is injective, so $\psi(b)=\psi(0)=0$ which implies that
$a=0$.  $j$ is injecive by definition so we have $x=0$.  So the only element in
the intersection is 0.  This implies that we can write
\begin{equation*}
 A\cong J\oplus B
\end{equation*}
which concludes the proof.
\end{proof}

\backmatter

\bibliographystyle{amsplain}
\thispagestyle{empty}
\bibliography{verwysings}
\vspace{1cm}
\address{Department of Physics\\NW1 5-62\\University of Pretoria}\newline
\email{mauritzvdworm@gmail.com}
  
\printindex

\end{document}